\documentclass[12pt,reqno]{amsart}

\numberwithin{equation}{section}

\usepackage{amsbsy}

\usepackage{ifpdf}
\ifpdf \usepackage[pdftex,pdfstartview=FitH,pdfpagemode=none,colorlinks,bookmarks,linkcolor=blue]{hyperref} \else  \usepackage[hypertex]{hyperref} \fi
\usepackage[nobysame,abbrev,alphabetic]{amsrefs}

\usepackage{enumerate}
\usepackage{amssymb}
\usepackage{amsmath}
\usepackage{amscd}
\usepackage{amsthm}
\usepackage{amsfonts}
\usepackage{graphicx}
\usepackage[all]{xy}
\usepackage{verbatim}
\usepackage{hyperref}
\usepackage{gensymb}

\newtheorem{theorem}{Theorem}[section]

\newtheorem{lemma}[theorem]{Lemma}
\newtheorem{corollary}[theorem]{Corollary}
\theoremstyle{definition}\newtheorem{definition}[theorem]{Definition}

\newtheorem{example}[theorem]{Example}

\newtheorem{notation}[theorem]{Notation}

\theoremstyle{definition}

\theoremstyle{definition}
\theoremstyle{definition}\newtheorem{remark}[theorem]{Remark}
\theoremstyle{definition}

\newcommand{\al}{\alpha}
\newcommand{\be}{\beta}
\newcommand{\ga}{\gamma}
\newcommand{\Ga}{\Gamma}
\newcommand{\del}{\delta}
\newcommand{\Del}{\Delta}
\newcommand{\lam}{\lambda}
\newcommand{\Lam}{\Lambda}
\newcommand{\eps}{\epsilon}

\newcommand{\ka}{\kappa}
\newcommand{\sig}{\sigma}
\newcommand{\Sig}{\Sigma}
\newcommand{\om}{\omega}
\newcommand{\Om}{\Omega}
\newcommand{\vphi}{\varphi}

\newcommand{\cB}{\mathcal{B}}
\newcommand{\cC}{\mathcal{C}}

\newcommand{\cE}{\mathcal{E}}
\newcommand{\cF}{\mathcal{F}}
\newcommand{\cG}{\mathcal{G}}
\newcommand{\cH}{\mathcal{H}}

\newcommand{\cK}{\mathcal{K}}
\newcommand{\cL}{\mathcal{L}}

\newcommand{\cO}{\mathcal{O}}
\newcommand{\cP}{\mathcal{P}}
\newcommand{\cQ}{\mathcal{Q}}
\newcommand{\cR}{\mathcal{R}}
\newcommand{\cS}{\mathcal{S}}
\newcommand{\cT}{\mathcal{T}}
\newcommand{\cU}{\mathcal{U}}

\newcommand{\bC}{\mathbb{C}}

\newcommand{\bG}{\mathbb{G}}

\newcommand{\bR}{\mathbb{R}}
\newcommand{\bZ}{\mathbb{Z}}
\newcommand{\bQ}{\mathbb{Q}}
\newcommand{\bF}{\mathbb{F}}

\newcommand{\bN}{\mathbb{N}}
\newcommand{\bH}{\mathbb{H}}
\newcommand{\bT}{\mathbb{T}}

\newcommand{\gog}{\mathfrak{g}}

\newcommand{\gol}{\mathfrak{l}}

\newcommand{\gos}{\mathfrak{s}}
\newcommand{\got}{\mathfrak{t}}

\newcommand{\SL}{\operatorname{SL}}
\newcommand{\PGL}{\operatorname{PGL}}
\newcommand{\GL}{\operatorname{GL}}
\newcommand{\PSL}{\operatorname{PSL}}

\newcommand{\nugauss}{\nu_{\on{Gauss}}}

\newcommand{\onh}{\on{h}}
\newcommand{\qi}{\on{QI}}
\newcommand{\defi}{{\, \stackrel{\textrm{{\tiny def}}}{=}\, }}
\newcommand{\GSS}{G_{S^*}}

\newcommand{\lie}{\operatorname{Lie}}
\newcommand\norm[1]{\left\|#1\right\|}
\newcommand\set[1]{\left\{#1\right\}}
\newcommand\pa[1]{\left(#1\right)}

\newcommand\idist[1]{\langle#1\rangle}
\newcommand\av[1]{\left|#1\right|}
\newcommand\on[1]{\operatorname{#1}}
\newcommand\diag[1]{\operatorname{diag}\left(#1\right)}
\newcommand\smallmat[1]{\pa{\begin{smallmatrix}#1\end{smallmatrix}}}
\newcommand\br[1]{\left[#1\right]}
\newcommand\tbf[1]{\textbf{#1}}
\newcommand\height[1]{\on{ht}(#1)}
\newcommand\mat[1]{\pa{\begin{matrix}#1\end{matrix}}}
\newcommand\enu[1]{\begin{enumerate}#1\end{enumerate}}
\newcommand\eq[2]{\begin{equation}\label{#1}#2\end{equation}}

\newcommand\wt[1]{\widetilde{#1}}

\newcommand{\ra}{\rightarrow}

\newcommand{\hra}{\hookrightarrow}

\newcommand{\onto}{\xymatrix{\ar@{>>}[r]&}}
\newcommand{\da}[4]{\xymatrix{#1 \ar@<.5ex>[r]^{#2} \ar@<-.5ex>[r]_{#3} & #4}}

\newif\ifdraft\drafttrue

\begin{document}
\title{On the evolution of continued fractions in a fixed quadratic field}
\author{Menny Aka}
\author{Uri Shapira}
\begin{abstract}
We prove that the statistics of the period of the continued fraction expansion of certain sequences of quadratic irrationals from a fixed quadratic field approach the `normal' statistics given by the Gauss-Kuzmin measure. 
As a by-product, the growth rate of the period is analyzed and, for example, it is shown that for a fixed integer $k$ and a quadratic irrational $\al$, the length of the period
of the continued fraction expansion of $k^n\al$  equals $c k^n +o(k^{(1-\frac{1}{16})n})$ for some positive constant $c$. This improves results of Cohn, Lagarias, and Grisel, and settles a conjecture of Hickerson.
The results are derived from the main theorem 
of the paper, which establishes an equidistribution result
regarding single periodic geodesics along certain paths in the Hecke graph. The results are effective and give rates of convergence and the main tools are spectral gap (effective decay of matrix coefficients) and  dynamical analysis on $S$-arithmetic homogeneous spaces.
\end{abstract}
\maketitle
\section{Introduction}
\subsection{Continued fractions}
The elementary theory of continued fractions starts by assigning to each real number $x\in[0,1]\smallsetminus\bQ$ an 
infinite sequence of positive integers\footnote{We shall completely ignore the rational numbers, which correspond to finite sequences as well as real numbers outside the unit interval, for which an additional integer  digit $a_0$ is needed.}\footnote{This correspondence is in fact a homeomorphism when $\bN^\bN$ is considered with the product topology.} referred to as the \textit{continued fraction expansion} of $x$ (abbreviated hereafter by c.f.e). Namely, to each number  $x$ 
corresponds a sequence $a_n=a_n(x)$, $n=1,2\dots$ which
is characterized by the requirement $x=\lim_{n\to\infty}\frac{1}{a_1+\frac{1}{\cdots+\frac{1}{a_n}}}$.
 We refer to the  numbers $a_n(x)$ as the digits of the c.f.e of $x$. When $x$ is understood 
we usually write $a_i$ for the $i$'th digit of the c.f.e of $x$.  

Given a number $x$, it is natural to ask for information regarding the statistical properties of its c.f.e; that is, for any finite sequence of natural numbers $w=(w_1,\dots ,w_k)$ (referred to hereafter as a \textit{pattern}) one is interested in the asymptotic 
\textit{frequency of appearance} of the pattern $w$ in the c.f.e of $x$, or in other words in the existence and the 
value of the limit
\begin{equation}\label{frequency}
D(x,w)=\lim_{N}\frac1N\#\set{1\le n \le N : w=(a_{n+1},\dots,a_{n+k})}.
\end{equation}
We claim that for 
Lebesgue almost any $x$ the limit in~\eqref{frequency} exists and equals some explicit integral (depending only on the pattern $w$).

To see this, note that the c.f.e correspondence $x\leftrightarrow \set{a_n(x)}$ fits in the commutative diagram 
\begin{equation}\label{cd}
\xymatrix{\bN^\bN \ar[r]^{\sigma}\ar[d]& \bN^\bN\ar[d] \\ [0,1]\smallsetminus \bQ\ar[r]^{S}&[0,1]\smallsetminus \bQ\;,}
\end{equation}
where 
$S(x)=\{\frac{1}{x}\}=\frac{1}{x}-\lfloor \frac{1}{x}\rfloor$ 
is the so-called Gauss map and $\sigma$ is the shift map $\sig(a_1,a_2,\dots)=(a_2,a_3,\dots)$. It is well known
that $S$ preserves the Gauss-Kuzmin measure on the unit interval which is given by 
\begin{equation}\label{Gauss-Kuzmin}
\nugauss(A)\defi\frac{1}{\log(2)}\int_A \frac{1}{1+x}dx.
\end{equation}
The map $S$ is ergodic with respect to $\nugauss$ which implies by the pointwise ergodic theorem (see for example~\cite[\S2.6,\S9.6]{EW}) that 
for $\nugauss$ (or equivalently Lebesgue) almost any $x$ and any pattern $w=(w_1,\dots ,w_k)$, the frequency $D(x,w)$ defined in~\eqref{frequency} exists. More precisely, if we let
\begin{equation}\label{U w}
I_w=\set{x\in [0,1]\smallsetminus \bQ: w=(a_1(x),\dots,a_k(x))}
\end{equation}
denote the interval consisting of those points for which the c.f.e starts with the pattern $w$,
then the pointwise ergodic theorem tells us that the ergodic averages of the characteristic function of 
$I_w$ converge almost surely to $\nugauss(I_w)$; that is 
\begin{equation}\label{ergodic avg}
\lim_{N\ra \infty} \frac{1}{N}\sum_{i=0}^{N-1} \chi_{I_w}(S^i(x))=\nugauss(I_w) , 
\end{equation}
for $\nugauss$-almost any $x$. As the set of possible patterns is countable we conclude that for Lebesgue almost any $x$~\eqref{ergodic avg} holds for any pattern $w$.  It is 
straightforward to check using the commutation in~\eqref{cd} that the limit in~\eqref{ergodic avg} is equal to the limit in~\eqref{frequency}.
\subsection{Quadratic irrationals}\label{subsection: statement}
Let $$\qi\defi\set{\al\in\bR: [\bQ(\al):\bQ]=2}$$ be the set of real quadratic irrationals.
By Lagrange's Theorem (see for example \cite[\S3.3]{EW}) $\qi$ is characterized as the set of
$x\in\bR$ for which the c.f.e is eventually periodic. For quadratic irrationals (which clearly form a Lebesgue-null set) it is clear that the limit in~\eqref{frequency} always exists and is different from the almost sure value of the frequency.

In this paper we investigate the behavior of 
$D(x,w)$ where $x$ varies in some fixed quadratic field. We make the convention to consider $x\on{mod} 1$ instead of $x$. This influences only the $0$'th digit in the classical 
discussion on continued fractions and does not effect  any statistical property of the c.f.e. As will become clear shortly, our 
approach manages to deal with sequences $\set{x_n}$ whose elements are 
arithmetically related in a way that ``involves only finitely many primes".

Before preparing the grounds for more general statements we state Theorem~\ref{thm:Main Theorem} which demonstrates the flavor of our results regarding continued fractions. 
To the best of our knowledge, all the results in the literature regarding the evolution of the period of the c.f.e of quadratics involve averaging. In this respect, the results we present are of a new kind.
\begin{notation}
Throughout this paper we use the notation $\ll$ in the following manner: Given two quantities $A,B$ depending on some  set of parameters $P$, we denote $A\ll B$ if there exists some absolute constant $c>0$ (independent of any varying parameter) such that $A\le cB$. Given a subset $P'$ of the parameter set $P$, we denote  $A\ll_{P'} B$ if there exists a constant $c_{P'}>0$, depending possibly on the parameters in $P'$, such that $A\le c_{P'} B$.
\end{notation}
We denote  by $\av{I_w}$ the length of the interval $I_w$ defined in~\eqref{U w}. For $\al\in\qi$ we denote by $\av{P_\al}$ the length of the period of the c.f.e of $\al$. 
 The following Theorem follows from Corollary~\ref{effective pattern frequency new} and Theorem~\ref{period cor new} as explained in Remark~\ref{short r.1 new}. 
\begin{theorem}\label{thm:Main Theorem}
Let $\alpha\in\qi$, $k\in\bN$ be given. Then, for any finite pattern of natural numbers $w=(w_1,\dots, w_k)$ we have that $D(k^n\al,w)\to\nugauss(I_w)$ as $n\to\infty$. Moreover, there exists a constant $c_{\al,k}$ such that  for any $n\in\bN$ the following holds
\begin{align}
\label{first eqq1}&\av{D(k^n\al,w)-\nugauss(I_w)}\ll_{\al,k}\av{I_w}^{-1} k^{-\frac{n}{32}};\\
\label{first eqq2} &\av{P_{k^n\al}}=c_{\al,k} k^n+O_{\al,k}(1)k^{(1-\frac{1}{16})n}.
\end{align}
\end{theorem}
\subsection{Structure of the paper}
The results presented in this paper are split into two; results regarding the distribution of closed geodesics and results regarding continued fractions. The gist of the paper is
concerned with the distribution of certain closed geodesics in (the unit tangent bundle of) the modular surface and the results regarding continued fractions are translations
of our results about geodesics utilizing the connection between the two. Although this connection is considered 
well understood, we  believe that some of the results we present that allow this translation are new and may find further applications
(e.g.\ Theorem~\ref{toest}).

Although the statements of our main results (Theorems~\ref{general case}, \ref{mtcf}, \ref{gpth}) require quite a bit of preparation, some of their consequences are fairly easy to state (e.g.\ Theorem~\ref{thm:Main Theorem}), and will hopefully motivate the reader traversing through the necessary preparations needed for the statements and proofs of the more general results. 

In~\S\ref{motivation and results} we begin fixing the notation, 
state Theorems~\ref{mtcf11}, \ref{period cor new} which deal  with continued fractions, and present some examples and open problems. 
In~\S\ref{preliminaries} we fix further notation. 
In~\S\ref{S hecke} we discuss the notions of \textit{$S$-Hecke graphs} and \textit{generalized branches} which play a key role in the statement of  
the  main Theorem~\ref{general case}.
In~\S\ref{other arguments} we discuss the relationship of Theorem~\ref{general case} to existing results and state Lemma~\ref{total growth}. 
This Lemma explains to some extent the phenomenon behind our results but is only used in the proof of growth  statements such as~\eqref{first eqq2} and is not needed
for the proof of statements such as~\eqref{first eqq1}. In~\S\ref{soft version} we prove  our main result, Theorem~\ref{general case}, where the main tool in the argument is the decay of matrix coefficients. In~\S\ref{proof of total growth} we give an elementary proof of Lemma~\ref{total growth}. In~\S\ref{applications} we prove our main results regarding  continued fractions, Theorems~\ref{mtcf}, \ref{gpth}, and deduce Theorems~\ref{mtcf11}, \ref{period cor new} which are stated in~\S\ref{motivation and results}. 
Theorems~\ref{mtcf}, \ref{gpth} are the translation to the language of continued fractions of Theorem~\ref{general case} and Lemma~\ref{total growth}. The technical tool we develop in order for this translation to carry 
through is Theorem~\ref{toest} which allows us to translate statements with an error term from the world of  closed  geodesics to the continued fractions world. 
Finally, in sections~\S\ref{potoest},\ref{construction of phi} we prove Theorem~\ref{toest} elaborating on the classical connection between the geodesic flow and continued fractions. 
\section{Some results and open problems}\label{motivation and results}
\subsection{Early preliminaries}
\begin{definition}
Given a commutative unital ring $\cR$ we let $$\GL_2(\cR)\defi\set{\smallmat{a&b\\c&d}\in\on{Mat}_{2\times 2}(\cR):ad-bc\in\cR^\times}$$ and $\PGL_2(\cR)=\GL_2(\cR)/Z$, where $Z\defi\set{\smallmat{a&0\\0&a}\in\GL_2(\cR)}$ is the center of $\GL_2(\cR)$.
When $\cR_1\hra \cR_2$, we have a natural embedding of $\PGL_2(\cR_1)\hra\PGL_2(\cR_2)$.
 We usually abuse notation and treat the elements of $\PGL_2(\cR)$ as matrices rather than equivalence classes of matrices.
\end{definition}
Recall that $\PGL_2(\bR)$ acts on the real line by M\"obius transformations; for $g=\smallmat{a&b\\c&d}\in\PGL_2(\bR)$ and $x\in\bR$, $gx\defi \frac{ax+b}{cx+d}$.
Recall the following basic result~\cite[Theorem 2]{SzuszRockett}
\begin{theorem}\label{tail equiv}
For any $x\in\bR$ the orbit of $x$ under $\PGL_2(\bZ)$ is exactly the set of numbers having c.f.e with the same tail as the c.f.e of x. Equivalently, 
$\PGL_2(\bZ)x=\set{y\in\bR:\exists m,n>0\, \forall i\ge0\, a_{n+i}(y)=a_{m+i}(x)}.$
\end{theorem}
As the c.f.e of $\al\in\qi$ is eventually periodic, it follows from~\eqref{cd} that the orbit $\{S^n\alpha\}_{n\in\bN}$ of $\al$ under the Gauss map is eventually periodic.
\begin{definition}\label{defnuiota}
Let $\al\in\qi$. We denote by $P_\al$ the  period of $\al\mod 1$ under the Gauss map; that is, $P_\alpha\defi\{x_1,\ldots,x_\ell\}\subset [0,1]$ where   
for some $n\ge 0$, $S^n(\al\mod 1)=x_1$ and, $S(x_i)=x_{i+1}$ for all $i<\ell$ and $S(x_\ell)=x_1$. 
We denote by $\nu_{\al}$ the normalized counting measure on $[0,1]$ supported on the period $P_\al$. 
\end{definition}
Let $\iota:\bR\to\PGL_2(\bZ)\backslash\bR$ be the quotient map to the `set of orbits'.
By Theorem~\ref{tail equiv}, for any $\al\in\qi$ we have that
\begin{equation}\label{mdch}
\textrm{For any $\be\in\iota(\al)$, $P_\al=P_\be, \nu_\al=\nu_\be$.}
\end{equation}
 We sometimes write $\nu_{\iota(\al)}, P_{\iota(\al)}$ when we wish to stress this fact.
\begin{lemma}\label{l 143}
Let $\al\in\qi$, for any pattern $w$, the frequency of appearance $D(\al,w)$ of the pattern $w$ in the c.f.e of $\al$ equals $\nu_\al(I_\om)=\frac{\av{P_\al\cap I_w}}{\av{P_\al}}$.
\end{lemma}
\begin{proof}
Let $(a_1,\dots a_k)$ be the period of the c.f.e of $\al$. By~\eqref{cd}, $P_\al=\set{x_1,\dots x_k}$ where $x_i\in[0,1]$ is the number whose c.f.e is given by the infinite concatenation
of   the pattern $(a_i,\dots a_k,a_1,\dots a_{i-1})$. The statement of the Lemma now follows easily from~\eqref{frequency} and~\eqref{U w}.
\end{proof}
 \begin{definition}
Given a finite set of primes $S$ we denote by  $\cO_S\defi\bZ\br{p^{-1}:p\in S}$ the ring of $S$-\textit{integers}. We denote by $\cO_S^\times\defi\set{\pm\prod_{p\in S}p^{n_p}:n_p\in\bZ} $ the group of $S$-units; 
that is, the group of invertible elements in $\cO_S$. 
\end{definition}
There is a natural embedding $\cO_S^\times\hra\PGL_2(\cO_S)$ given by $q\mapsto \diag{q,1}$. We denote by $\ga_q=\diag{q,1}$ the image of $q$ under
this embedding. Note that $q\al=\ga_q\al$.
\begin{definition}
For $\ga\in\PGL_2(\bQ)$ let $\smallmat{a&b\\c&d}\in \on{Mat}_{2\times 2}(\bZ)$ be the unique representative of $\ga$ with co-prime entries. 
We define the \textit{height} of $\ga$ to be $\on{ht}(\ga)\defi\av{\det \smallmat{a&b\\c&d}}.$
Given a rational number $q= \pm\prod_1^\ell p_i^{e_i}$, where the $p_i$'s are distinct primes and $e_i\in\bZ$,   
we define the \textit{height} of $q$ to be
$\on{ht}(q)\defi\height{\ga_q}=\prod_1^\ell p_i^{\av{e_i}}.$
\end{definition}
For $\ga\in\PGL_2(\bQ)$, $\height{\ga}$ measures how far $\ga$ is from $\PGL_2(\bZ)$. As the $\PGL_2(\bZ)$ action does not change the period, it is natural to expect that a statement 
regarding the evolution of $\nu_{\ga\al}$ will depend on $\height{\ga}$. This is indeed the case as will be seen shortly.
\subsection{Results}\label{results 11}
Our results are concerned with the convergence  $\nu_{\ga\al}\to\nu_{\on{Gauss}}$ and the growth of the length of the period $\av{P_{\ga\al}}$ as $\height{\ga}\to\infty$ and $\ga\in\PGL_2(\cO_S)$ for a fixed finite set of primes $S$ and $\al\in\qi$. 
We give estimates
on error terms and so refer to our results as \textit{effective}. In these estimates there appears an exponent $\frac{25}{64}\le \del_0\le \frac{1}{2}$ whose exact value is not known (although according to the Ramanujan conjecture $\del_0=\frac{1}{2}$). The bigger it is the stronger the statements are and the best known lower bound
for it to this date is $\del_0\ge \frac{25}{64}$; a bound given by Kim and Sarnak in the appendix of~\cite{KimSarnak}\footnote{This parameter relates to the representation theory of $\GL_2$
(see~\S\ref{effective mixing}).}.

Naturally, our results involve comparison of integrals with respect to measures which are mutually singular, and in order to make sense of an error term we need to restrict our attention to integrals of functions with some controlled behavior. This is usually done by looking at smooth functions and considering Sobolev norms. We choose to work with the more primitive notion of Lipschitz functions.
\begin{definition}
Let $(X,d)$ be a metric space. For any $\ka>0$ we denote by $\on{Lip}_\ka(X)$ the space of Lipschitz functions $f:X\to\bC$ with $\ka$ as a Lipschitz constant. We sometimes refer to such functions as $\ka$-Lipschitz.
\end{definition}
The following Theorem is deduced from Theorem~\ref{mtcf} in~\S\ref{applications}.
\begin{theorem}\label{mtcf11}
Let $S$ be a finite set of primes, $\al\in \qi$. 
\begin{enumerate}
\item\label{part1mt}  If $\set{q_n}\subset\cO_S^\times$ is a sequence such that $\height{q_n}\to\infty$ then $\nu_{q_n\al}\to\nu_{\on{Gauss}}.$ More precisely, given
$\eps>0$, $q\in\cO_S^\times$,
and $f\in\on{Lip}_\ka([0,1])$ the following estimate holds
\begin{equation}\label{eqmtcf1}
\av{\int_0^1fd\nu_{\on{Gauss}}-\int_0^1fd\nu_{q\al}}\ll_{\al,S,\eps}\max\set{\norm{f}_\infty,\ka}\on{ht}(q)^{-\frac{\del_0}{6}+\eps}.
\end{equation}
\item\label{part2mt}  Assume that all the primes in $S$ do not split in the extension $\bQ(\al)$ of $\bQ$.
Then, if $\set{\ga_n}\subset\PGL_2(\cO_S)$ is a sequence such that $\height{\ga_n}\to\infty$ then $\nu_{\ga_n\al}\to\nu_{\on{Gauss}}$.  More precisely, given $\eps>0$, $\ga\in\PGL_2(\cO_S)$,
and $f\in\on{Lip}_\ka([0,1])$ the following estimate holds
\begin{equation}\label{eqmtcf2}
\av{\int_0^1fd\nu_{\on{Gauss}}-\int_0^1fd\nu_{\ga\al}}\ll_{\iota(\al),S,\eps}\max\set{\norm{f}_\infty,\ka}\on{ht}(\ga)^{-\frac{\del_0}{6}+\eps},
\end{equation}
\item\label{part3mt} If one of the primes in $S$ splits in the extension $\bQ(\al)$ of $\bQ$, then there exists sequences $q_n\in\cO_S^\times, \ga_n\in\PGL_2(\bZ))$ with $\height{q_n}\to\infty$
such that $\nu_{q_n\ga_n\al}$ does not converge to $\nugauss$ and in particular, the implicit constant in~\eqref{eqmtcf1} cannot be taken to be uniform on the orbit 
$\iota(\al)$ in contrast with~\eqref{eqmtcf2}\footnote{In fact, it is possible to choose $q_n$ so that $\nu_{q_n\ga_n\al}$ is a constant sequence.}.
\end{enumerate}
\end{theorem}
\begin{remark}\label{r.250}
\begin{enumerate}
\item Under the assumption that no prime in $S$  splits in $\bQ(\al)$, \eqref{eqmtcf1} follows from~\eqref{eqmtcf2} by choosing  $\ga=\ga_q$.   
\item It is an exercise to show that any $\ga\in\PGL_2(\cO_S)$ can be written as a product $\ga=\ga_1\ga_q\ga_2$, where $\ga_i\in\PGL_2(\bZ)$ 
and $q\in\cO_S^\times$ (see the proof of Corollary~\ref{sphere again new}). It now follows from~\eqref{mdch} that $\nu_{\ga\al}=\nu_{q\be}$ where $\be=\ga_2\al$, 
and so although it seems more restrictive at first glance, instead of studying 
the evolution of $\nu_{\ga\al}$ as $\ga\in\PGL_2(\cO_S)$, it is enough to consider the evolution of $\nu_{q\be}$ as $q\in\cO_S^\times$, $\be\in\iota(\al)$.
\end{enumerate}
\end{remark}

When we use Theorem~\ref{mtcf11} to try and estimate the frequency of a pattern in the period of the c.f.e of $\ga\al$ we obtain the following 
\begin{corollary}\label{effective pattern frequency new}
Let $S$ be a finite set of primes and $\al\in\qi$. For any finite pattern $w=(w_1\dots w_k)$ of digits, and any $q\in\cO_S$
\begin{equation}\label{hopefully last new}
\av{D(q\al,w)-\nu_{\on{Gauss}}(I_w)}\ll_{\al,S,\eps}\av{I_w}^{-1}\on{ht}(q)^{-\frac{\del_0}{12}+\eps}.
\end{equation}
Moreover, if all the primes in $S$ do not split in the extension $\bQ(\al)$ of $\bQ$, then for any $\ga\in\PGL_2(\cO_S)$
\begin{equation}\label{hopefully last  new ns}
\av{D(\ga\al,w)-\nu_{\on{Gauss}}(I_w)}\ll_{\iota(\al),S,\eps}\av{I_w}^{-1}\on{ht}(\ga)^{-\frac{\del_0}{12}+\eps}.
\end{equation}
\end{corollary}
\begin{proof}
By Lemma~\ref{l 143},  $D(\be,w)=\nu_{\iota(\be)}(I_w)$ for any $\be\in\qi$.
The exponent in~\eqref{eqmtcf1} (resp.\ \eqref{eqmtcf2}) is cut in half in~\eqref{hopefully last new} (resp.\ \eqref{hopefully last  new ns}) as a result of the fact that $\chi_{I_w}$ is not Lipschitz and one needs to use an approximation of it 
in order to apply Theorem~\ref{mtcf11}. We leave the details to the reader.
\end{proof}
 Theorem~\ref{mtcf11} raises a natural question: Is it true that $\nu_{q_n\al}\to\nu_{\on{Gauss}}$ for any sequence of rationals $q_n$ with $\on{ht}(q_n)\to\infty$? 
  The following example which was essentially communicated to us by A. Ubis shows that the answer is negative and so the assumption that $q_n\in\cO_S$ for a fixed finite set of primes $S$ is crucial.
\begin{example}
Let $D$ be a fundamental discriminant such that the negative Pell equation $x^2-Dy^2=-1$ has an integer solution (see~\cite{LagariasComputationalComplexity},\cite{EtienneAnnals} for example). A solution $x=k_1,y=n_1$ to the equation corresponds to a unit $\eps_1\defi k_1+n_1\sqrt{D}$ in the ring $\bZ(\sqrt{D})$ of norm $-1$ and in turn, the odd powers $\eps_1^j=k_j+n_j\sqrt{D}$  give rise
 to infinitely many further solutions
of the negative Pell equation. For odd $j$ let $\al_j$ solve the equation $x=2k_j+\frac{1}{x}$. That is, the c.f.e of $\al_j$ is purely periodic with period of length $1$ of digit 
$2k_j$ (note that here we abuse the notation introduced above and we do record the 0 digit). Solving for $x$ in the above equation we see that $\al_j$ could be chosen to be $k_j+\sqrt{k_j^2+1}$. As 
$(k_j,n_j)$ solve the negative Pell equation for $D$ we get $\al_j=k_j+n_j\sqrt{D}$ which shows that the measures $\nu_{n_j\sqrt{D}}$ are not converging to the Gauss-Kuzmin measure and in fact are atomic measures supported on single points.   
\end{example}
In~\S\ref{applications} we prove Theorem~\ref{gpth} which discusses the  length of the period of the c.f.e of $\ga\al$, where 
$\al\in\qi$ is fixed and $\ga\in\PGL_2(\cO_S)$ varies. We state below Theorem~\ref{period cor new} which is an adaptation
of Theorem~\ref{gpth} that uses only the terminology presented so far. 
It is deduced from Theorem~\ref{gpth} in~\S\ref{applications}.
Theorem~\ref{period cor new} solves a conjecture of Hickerson~\cite{Hickerson} and 
strengthens~\cite[Theorem 3]{Cohn}. 
Note that  by the argument in Lemma~\ref{l 143}, for any $\al\in\qi$, $\av{P_\al}$ is the \textit{length of the period of the c.f.e of $\al$}. 

In the Appendix of~\cite{LagariasComputationalComplexity}, using the methods of Dirichlet~\cite{Dirichlet}, Lagarias shows 
that, under some restrictive assumptions on $\al$,
one has  that for any integer $k$ there exists a constant $C$ for which, $C\frac{k^n}{n}<\av{P_{k^n\al}}$. 
Under some restrictive assumptions on $\al$, Grisel~\cite{Grisel} proved a stronger estimate of the form $C_1k^n\le \av{P_{k^n\al}}\le C_2 k^n$.
The following Theorem strengthens these results in several respects. 
\begin{theorem}\label{period cor new}
Let $S$ be a finite set of primes. There exists  a positive function $c(\al,\ga)$ on the set $\qi\times \PGL_2(\cO_S)$ satisfying the following: For any $\al\in\qi$,
\begin{enumerate}
\item 
\enu{
\item\label{gpthintro1} 
For any  $\eps>0$, and $q\in\cO_S^\times$
\begin{equation}\label{eqgpthintro1}
\av{P_{q\al}}=c(\al,\ga_q)\height{\ga_q} +O_{\al,S,\eps}(1)\height{\ga_q}^{1-\frac{\del_0}{6}+\eps}.
\end{equation}
Moreover, if all the primes in $S$ do not split in the quadratic field $\bQ(\al)$, then for any $\eps>0$, $\ga\in\PGL_2(\cO_S)$
\begin{equation}\label{eqgpthintro2}
\av{P_{\ga\al}}=c(\al,\ga)\height{\ga} +O_{\iota(\al),S,\eps}(1)\height{\ga}^{1-\frac{\del_0}{6}+\eps}.
\end{equation}
\item\label{gpthintro4}  The function $c$ attains only finitely many values on $\cO_S^\times$; that is, \\
$\av{\set{c(\al,\ga_q):q\in\cO_S^\times}}<\infty$.
\item\label{gpthintro1c} If $q_n=\ell_1^{(n)}/\ell_2^{(n)}$, where $\ell_i^{(n)}\in \cO_S^\times\cap\bN$  satisfies $\ell_i^{(n)}|\ell_i^{(n+1)}$ for $i=1,2$, then $c(\al,\ga_{q_n})$ stabilizes.
}
\item\label{gpthintro3} $\sup\set{c(\al,\ga):\ga\in\PGL_2(\cO_S)}\ll_{\iota(\al),S} 1$.
\item\label{gpthintro5} All the primes in $S$ do not split in the quadratic field $\bQ(\al)$ if and only if 
\begin{equation}\label{eqgpthintro3}
\inf\set{ c(\al,\ga):\ga\in \PGL_2(\cO_S)}>0.
\end{equation}
\end{enumerate}
\end{theorem}
As an immediate corollary we have for example the following
\begin{corollary}
For any $\al\in\qi$ and any positive integer $k$,
$\lim_n \frac{\av{P_{k^n\al}}}{k^n}$ exists and is a positive real number. 
\end{corollary}
\begin{remark}\label{short r.1 new}
The first part of Theorem~\ref{thm:Main Theorem} is obtained from Corollary~\ref{effective pattern frequency new} by taking $q=k^n$, the Kim-Sarnak exponent $\del_0=\frac{25}{64}$, and choosing $\eps = \frac{1}{768}$ so that 
$-\frac{\del_0}{12}+\eps=-\frac{1}{32}$. The second part of Theorem~\ref{thm:Main Theorem} is obtained similarly from Theorem~\ref{period cor new}\eqref{gpthintro1}.
\end{remark}

\subsection{References to existing results}\label{refs}
Although the question of the evolution of the c.f.e along arithmetically defined sequences in a fixed quadratic field is extremely natural, we did not find too many relevant 
papers to cite. Some earlier works studying
the statistics of the period `in average' (and also not in a fixed field), were initiated by Arnold (see~\cite{Arnold1},\cite{Arnold2},\cite{Lerner} and the references therein).
See also~\cite{Pollicott}. Other works, mostly related to the length of the period,
which the reader might find related, may be found for example in many of the papers of Golubeva (such as~\cite{Golubeva}) and in~\cite{Grisel},\cite{BugeaudFlorian}\cite{MendesFranceProblems},\cite{CorvajaZannier},\cite{Cohn},\cite{Hickerson},\cite{Matthews}. Standing out in
this context is the recent paper of McMullen which provides examples of sequences of quadratic irrationals in a fixed quadratic field with uniformly bounded c.f.e digits~\cite{McMullenUniformlyBounded}. 
We suspect that it should be very interesting to compare in detail how McMullen's results fit together with the results of the present paper. 

As for results regarding periodic geodesics the situation is completely different and we will not attempt to summarize the relevant results that appeared in the literature.
We  comment though, that as will be explained in~\S\ref{other arguments}, our main Theorem~\ref{general case} is closely related to the work of Benoist and Oh~\cite{BenoistOhGafa}. 
\subsection{Some open problems}\label{open problems}
We list below a few questions which emerge from our discussion and remain unsolved. Each of the problems below have a corresponding problem stated in terms of
periodic geodesics on the modular surface.
\begin{enumerate}
\item Give satisfactory sufficient conditions on a sequence of rationals $q_n$ to ensure that for a quadratic irrational $\al$, the sequence of measures $\nu_{q_n\al}$ equidistribute
to the Gauss-Kuzmin measure $\nu$. It might be interesting to replace the 
quantifiers
and allow the conditions to depend on $\al$.
\item Is it true that for a quadratic irrational $\al$ which is not a unit in the ring of integers of $\bQ(\al)$, 
the sequence of measures $\nu_{\al^n}$ always equidistribute to the Gauss-Kuzmin measure along the subsequence of $n$'s for which
$\al^n$ is irrational (see~\cite{CorvajaZannier}). Note that our results deal with the case $\al=\sqrt{d}$.
\item Let $p_n$ be an enumeration of the primes. Are there any quadratic irrationals $\al$ for which $\nu_{p_n\al}$ equidistribute to the Gauss-Kuzmin measure.
\item Is it true that for any quadratic irrational $\al$ there exist a sequence of distinct primes $p_n$ so that $\nu_{p_n\al}$ equidistribute to the Gauss-Kuzmin measure.
\end{enumerate}
\subsection{Acknowledgments}
We benefited from numerous stimulating conversations with Manfred Einsiedler, Elon Lindenstraus and various other people which we try to name below. We would like to express our gratitude and appreciation to them for the generosity 
in which they exchange valuable ideas and create an `open source' healthy environment in the community doing Homogeneous Dynamics. Thanks are also due to 
Yann Bugeaud, \'Etienne Fouvry, Tsachik Gelander, Alex Gorodnik, Chen Meiri, Philippe Michel, Shahar Mozes, Hee Oh, Peter Sarnak, Adrian Ubis, Akshay Venkatesh, and Barak Weiss. 

Both Authors enjoyed the warm hospitality of the Centre Interfacultaire Bernoulli and the GANT semester held there.

M.A acknowledges the support of the Advanced research Grant 226135 from the European Research Council, the ISEF foundation, the Ilan and Asaf Ramon memorial foundation, and the Hoffman Leadership and Responsibility fellowship program at the Hebrew University of Jerusalem.

U.S acknowledges the  support of the Advanced research Grant 228304 from the European Research Council.

%
%
%
%
%
%
%
%
\section{Preliminaries}\label{preliminaries} 
In this section we fix the notation that we will use in~\S\ref{S hecke}--\S\ref{applications}. In~\S\ref{potoest},\ref{construction of phi} our notation will slightly vary as will be explained at the beginning of~\S\ref{potoest}.

For a prime $p$ we let $\bQ_p$ denote the field of  p-adic numbers and let $\bZ_p$ be the ring of $p$-adic integers. We sometimes denote $\bQ_\infty = \bR$. 
Let $\textbf{P}\defi\set{p\in\bN:p\textrm{ is a prime}}$. Given $S\subset \tbf{P}$ we denote $S^*=S\cup\set{\infty}$. The set $\tbf{P}^*$ will be referred to as the set of 
\textit{places} of $\bQ$ -- the primes being the \textit{finite places}. 

Let $S\subset \textbf{P}^*$ be given. Throughout $S_f=S\smallsetminus\set{\infty}$. We denote by $\bQ_S,\bZ_S$  
the product rings $\prod_{v\in S}\bQ_v,\prod_{v\in S}\bZ_v$ respectively (the latter makes sense only when $\infty\notin S$). Let $\bG$ denote  the algebraic group 
$\PGL_2$. We denote $G_S=\bG(\bQ_S)$. We  denote
an element $g\in G_S$ by a sequence $g=(g_v)_{v\in S}$ where $g_v$ is a $2\times 2$ matrix over $\bQ_v$. Keep in mind  the slight abuse of notation arising from the fact that $g_v$ is in fact an equivalence class of matrices. If $\infty\in S$ we usually abbreviate  
and write $g=(g_\infty, g_f)$ where $g_f$ denotes the tuple of the components  corresponding to the finite places in $S$. 
The identity elements in the various groups are 
denoted by $e$ with the corresponding subscript. Thus for example $e_S=(e_v)_{v\in S}$ and if $\infty\in S$, $e_S=(e_\infty,e_f)$. 

Hereafter $S\subset \tbf{P}$. We may view the group $\Ga_S\defi\bG(\cO_S)$ as a subgroup of $G_{S^*}=G_\infty\times\prod_{v\in S}G_v$ (embedded diagonally). When $S=\varnothing$, $\cO_S=\bZ$ and we 
denote $\Ga_S=\bG(\bZ)$ by $\Ga_\infty$. 
It is well known that $\Ga_S$ is a lattice in $G_{S^*}$.
We set  $$X_S\defi\Ga_S\backslash \GSS,\; X_\infty\defi\Ga_\infty\backslash G_\infty.$$ 
The gist of our discussion will be concerned with these homogeneous spaces.
We denote by $m_S$ (resp.\ $m_\infty$) the $\GSS$-invariant (resp.\ $G_\infty$-invariant) probability measure on $X_S$ (resp.\ $X_\infty$). 
The \textit{real quotient} $X_\infty$
is a \textit{factor} of $X_S$ in a natural way: 
Let $K_\infty\defi \on{PO}_2(\bR)$ denote the maximal compact subgroup of $G_\infty$. For 
a finite place $p\in\textbf{P}$ we let $K_p=\bG(\bZ_p)$. We then let $K_S\defi\prod_{v\in S} K_v$. 
As we will explain shortly, the double coset space $X_S/K_S=\Ga_S\backslash \GSS/K_S$ is naturally
identified with $X_\infty$. We denote by $\pi: X_S\to X_\infty$ the natural projection. This identification relies on two facts (i) $G_S=\Ga_S K_S$, and (ii) $\Ga_\infty=\Ga_S\cap K_S$. Relying on these facts the identification is as follows: Given a double coset
$\Ga_S(g_\infty,g_f)K_S$, by (i) we may assume without loss of generlity that $g_f=e_f$ and identify this double coset with $\Ga_\infty g_\infty\in X_\infty$. The reader will easily check that (ii) implies that this map is indeed well defined and bijective. We leave the verification of conditions (i),(ii) to the reader ((ii) is straightforward and an argument similar to that giving (i) may be extracted from 
 the proof of Lemma~\ref{ergodicity 10} for example). 
\begin{remark}\label{12.9.r.1}
In practice, given $x=\Ga_S(g_\infty,g_f)\in X_S$ with representative $(g_\infty,g_f)$ such that $g_f\in K_S$, the projection $\pi(x)$ is $\Ga_\infty g_\infty$. 
 In other words, $\pi^{-1}(\Ga_\infty g_\infty)=\set{\Ga_S(g_\infty, g_f):g_f\in K_S}$. Another useful observation to keep in mind here is that two points $x_1=\Ga_S(g_\infty, g_f),x_2=\Ga_S(g_\infty,h_f)$ are in the same fiber 
(that is $\pi(x_1)=\pi(x_2)$) if and only if the quotient $g_f^{-1} h_f$ belongs to $K_S$.
\end{remark}
The group $\GSS$ (and all its subgroups) act on $X_S$ by right translation. In particular, if $T\subset S^*$, we may view $G_T$ (and its subgroups) as a subgroup of $\GSS$ and thus it acts on $X_S$.
Note that $\pi:X_S\to X_\infty$ intertwines the $G_\infty$-actions. Of particular interest to us will be the action of the real diagonal group  $A_\infty\defi\set{\diag{e^t,1}:t\in\bR}$, the elements of which
we often write as $a_\infty(t)=\diag{e^t,1}$.

We say that an orbit $xL$ of a closed subgroup $L<\GSS$ through a point $x\in X_S$ is \textit{periodic} if it supports an $L$-invariant probability measure. Such a measure is unique and we refer to it as the Haar measure on the periodic orbit. Compact orbits are always periodic. Given a measure $\mu$ on $X_S$ and $g\in \GSS$ we let $g_*\mu$ denote the pushed forward measure by right translation by $g$. 
This notation is a bit awkward as $(gh)_*\mu=h_*(g_*\mu)$. This will not bother us as we will only use commutative subgroups to push measures.

For $v\in\tbf{P}^*$, the Lie algebra of $G_v$ will be denoted by $\gog_v$ and is naturally identified with the space of traceless $2\times 2$ matrices over $\bQ_v$. Similarly to the notation introduced above we will denote by $\gog_{S^*}=\gog_\infty\oplus_{v\in S}\gog_v$ the Lie algebra of $\GSS$. A basic fact that we will 
use is that if $S$ is finite and $L<G_S$ is a closed subgroup then $L$ contains an open product subgroup $\prod_{v\in S} L_v$ which allows us to speak of the Lie algebra of $L$ which will be denoted 
$\lie{(L)}$. The exponential map $\exp_v:\gog_v\to G_v$ is defined for any place $v$ by the usual power series and in fact, is only well defined for finite places on a certain neighborhood of $0$. 
We denote its inverse by $\log_v$ (it is defined on a small enough neighborhood of $e_v$) and use the obvious notation $\exp_S, \exp_{S^*}, \log_S,\log_{S^*}$ to denote the corresponding product maps from the corresponding product domains
in $\gog_S, \gog_{S^*}, G_S, \GSS$ respectively. 

Given 
an element $g\in \GSS$ and an element $u$ (either of $\GSS$ or of $\gog_{S^*}$), we denote  $$u^g\defi g^{-1}u g.$$ If $g$ is semisimple we denote by
 $(\gog_{S^*})_g^{\on{ws}}$ the 
\textit{weak stable} subalgebra of $\gog_{S^*}$. It is defined as the direct sum of the eigenspaces (of the operator $u\mapsto u^g$) of modulus $\le 1$ or equivalently
\begin{align*}
(\gog_{S^*})_g^{\on{ws}}&=\set{u\in \gog_{S^*}: \set{u^{(g^n)}}_{n>0}\textrm{ is bounded in }\gog_{S^*}}.
\end{align*}

For each place $v$ we equip $G_v,\gog_v$ with metrics in the following way: For $v=\infty$ we start with an inner product on $\gog_\infty$ which is right $K_\infty$-invariant and use left translation to make it into a left invariant Riemanian metric on $G_\infty$ which is also right $K_\infty$-invariant. This Riemannian metric induces left $G_\infty$-invariant, bi-$K_\infty$-invariant metric on $G_\infty$.
For a finite place $v$, we start with a bi-$K_v$-invariant metric $\on{d}_{K_v}$ on $K_v$ (such that $K_v$ 
equals the closed unit ball around $e_v$) and make it into a left invariant metric on $G_v$ (which is also right $K_v$-invariant) by setting $\on{d}_{G_v}(g_1,g_2)=2$ if $g_1^{-1}g_2\notin K_v$ and 
$\on{d}_{G_v}(g_1,g_2)=\on{d}_{K_v}(g_1^{-1}g_2,e_v)$ otherwise. On the Lie algebra $\gog_v$ we take the metric given by $\on{d}_{\gog_v}(u,w)=\max\set{\av{u_{ij}-w_{ij}}_v: 1\le i,j\le 2}$ where the indices $i,j$ stand for the entries of the corresponding matrix. 
We usually denote the distance from $0\in \gog_v$ by $\norm{u}\defi\on{d}_{\gog_v}(u,0)$ and refer to it as the \textit{norm} of $u$. 
We define the metrics $\on{d}_{\GSS},\on{d}_{\gog_{S^*}}$ on $\GSS,\gog_{S^*}$ respectively by taking the maximum of the metrics defined above over the places in $S^*$. 
The metric $\on{d}_{\GSS}$ induces a right-$K_\infty\times K_S$-invariant metric on $X_S$ by setting $\on{d}_{X_S}(\Ga_Sg_1,\Ga_Sg_2)=\inf_{\ga\in \Ga_S}\on{d}_{\GSS}(\ga g_1,g_2)$.

In a metric space $(X,\on{d}_X)$ we denote $B_r^X(x)$ the open ball of radius $r$ around $x$. In case the space is a group, we denote by $B_r^X$ the corresponding ball around the trivial element.

\section{The $S$-Hecke graph and the main theorem}\label{S hecke}
Throughout this section we use the notation introduced in~\S\ref{preliminaries}. We fix a finite set of finite places $S\subset\tbf{P}$. 
The space $X_\infty$ can be thought of as  the moduli space of equivalence classes of 2-dimensional
lattices in the plane $\bR^2$ up to homothety. We  will refer below to a point $x\in X_\infty$ as a class; here, the class  $\Ga_\infty g$ is composed of the lattice spanned by the rows of the matrix $g$ (which is well defined up to scaling) and all its homotheties.

Our first aim is to state Theorem~\ref{general case}. Briefly, we will fix a class
$x$ with periodic $A_\infty$-orbit and consider a class $x'$ on the $S$-Hecke graph (soon to be defined) through $x$ and  prove an effective equidistribution statement regarding the periodic orbit  $x'A_\infty$ as $x'$ drifts
away from $x$ in the graph. 
\subsection{Hecke friends} Given a class $x\in X_\infty$, we say that a class $x'$ is a \textit{Hecke friend} of $x$ if one can choose lattices $\Lam_x\in x,\Lam_{x'}\in x'$ 
such that $\Lam_{x'}<\Lam_{x}$. After fixing the lattice $\Lam_x$ there is a unique choice of $\Lam_{x'}\in x'$ such that $\Lam_{x'}<\Lam_x$ is \textit{primitive}; that is, such that the  index $[\Lam_x:\Lam_{x'}]$ is minimal. 
We denote this minimal index by $\on{ind}(x,x')$.  
We say that $x'$ is an $S$-\textit{Hecke friend} of $x$ if  $\on{ind}(x,x')\in\cO_S^\times$. It is elementary to check that the $S$-Hecke friendship relation is an equivalence relation and that 
furthermore, if $x,x'$ are Hecke friends then $\on{ind}(x,x')=\on{ind}(x',x)$.
\subsection{The graph}\label{TDI}
For a class $x\in X_\infty$ we define
\begin{equation}\label{Hecke graph}
\cG_S(x)=\set{x'\in X_\infty:x,x'\textrm{ are $S$-Hecke friends}}
\end{equation} 
The set $\cG_S(x)$ has the structure of a graph\footnote{When $S$ contains only one
prime, this is the well known $p$-\textit{Hecke tree} through $x$. In general, this graph is the product of the various $p$-Hecke trees for $p\in S$.}: We join $x_1,x_2\in\cG_S(x)$ with an edge if there exists $\Lam_i\in x_i$ such that $\Lam_1$  is a sublattice of $\Lam_2$ of index $p$
for some $p\in S$ (note that as $p$ is prime this forces $\Lam_1$ to be a primitive sublattice of $\Lam_2$). In this case we declare the length of this edge to be $\log(p)$. This induces a distance function on the graph which we denote $\on{d}_{\cG}(\cdot,\cdot)$ for which $\on{d}_\cG(x_1,x_2)=\log(\on{ind}(x_1,x_2))$. 
We will refer to $x$ as the \textit{root} of $\cG_S(x)$ and call $\cG_S(x)$ the \textit{$S$-Hecke graph through $x$}. 
Note that $\set{\on{ind}(x,x'): x'\in\cG_S(x)}=\cO_S^\times\cap \bN.$ 
We refer to numbers in $\cO_S^\times\cap \bN$ as  \textit{admissible radii} and denote for $\onh\in\cO_S^\times\cap\bN$ by
$\cS_{\on{h}}(x)\defi\set{x'\in\cG_S(x):\on{ind}(x,x')=\on{h}}$ the \textit{sphere} of radius $\on{h}$ around the root $x$.

\subsection{The sphere} 
For  $q\in\cO_S^\times$ let us define 
\begin{equation}\label{algebraic relation of lifts'}
a_f(q)\defi\pa{e_\infty,\smallmat{1&0\\0&q},\dots,\smallmat{1&0\\0&q}}\in G_S.
\end{equation}
Given $x\in X_\infty$ we wish to have a convenient algebraic description of the classes on the sphere
$\cS_{\on{h}}(x)$ for admissible radii $\on{h}$. We obtain this description using the extension $\pi:X_S\to X_\infty$ in the following way: Lemma~\ref{get the sphere} below shows that 
the various points on $\cS_{\on{h}}(x)$ are obtained by choosing a lift $y\in\pi^{-1}(x)$ of $x$,  and projecting $ya_f(\on{h})$ via $\pi$ back to $X_\infty$.
\begin{lemma}\label{get the sphere}
For $x=\Ga_\infty g\in X_\infty$ and $\onh$ an admissible radius we have 
\begin{align}\label{sphere}
\cS_{\on{h}}(x)&=\pi\pa{\set{\Ga_S(g,\ga):\ga\in\Ga_\infty}a_f(\on{h})}\\
\nonumber &=\pi\pa{\pi^{-1}(x)a_f(\on{h})}.
\end{align}
\end{lemma}
\begin{proof}
Recall that the elementary divisors theorem attaches to any pair of lattices $\Lam_1<\Lam_2$ in the plane, a pair of integers $d_1,d_2$  which are characterized by the following two properties: (1) the divisibility $ d_2|d_1$ holds, (2)
there exists a basis $v_1,v_2$ of $\Lam_2$ such that $d_1v_1,d_2v_2$ forms a basis of $\Lam_1$. Note that $\Lam_1$ is a primitive sublattice of $\Lam_2$ if and only if the second divisor satisfies
$d_2=1$. We conclude from here that given a class $x=\Ga_\infty g$, then a class $x'$ lies on the sphere $\cS_{\onh}(x)$ if and only if there exists a lattice $\Lam_{x'}\in x'$ which is a sublattice
of the lattice $\Lam_x\in x$ spanned by the rows of $g$ such that the elementary divisors are $d_1=\onh, d_2=1$. In other words we have the equality
\begin{equation}\label{sphere again}
\cS_{\onh}(x)=\set{\Ga_\infty\diag{\onh,1}\ga g:\ga\in\Ga_\infty}.
\end{equation}
The following identity is crucial for us. It shows how the lattice $\Ga_S$ causes the desired interaction between the real and $p$-adic components in the extension $X_S$ of $X_\infty$:
\begin{align}\label{heart''}
\Ga_\infty\diag{\onh,1}\ga g&=\pi\pa{\Ga_S(\diag{\onh,1}\ga g,e_f)}\\ 
\nonumber &=\pi\pa{\Ga_S\underbrace{\ga^{-1}\diag{1,\onh}}_{\in \Ga_S}(\diag{\onh,1}\ga g,e_f)} = \pi\pa{\Ga_S(g,\ga^{-1})a_f(\onh)}.
\end{align}
From equations~\eqref{sphere again},\eqref{heart''} we immediately conclude that $$\cS_{\onh}=\pi\pa{\set{\Ga_S(g,\ga):\ga\in\Ga_\infty}a_f(\onh)},$$ which is the first equality in~\eqref{sphere}.
Using the first equality, the second equality follows once we show
that for any given $\omega\in K_S$ there exist $\gamma\in\Gamma_{\infty}$
such that $\pi(\Gamma_{S}(g,\gamma)a_{f}(\on h))=\pi(\Gamma_{S}(g,\omega)a_{f}(\on h))$.
A short calculation using Remark~\ref{12.9.r.1} shows that this happens precisely
when 
\begin{equation}
\gamma^{-1}\omega\in a_{f}(\on h)K_Sa_{f}(\on h)^{-1}\text{.}\label{eq:nts}
\end{equation}
Thus, let $\omega=(\omega_{p})_{p\in S_f}\in K_S$ be given and write $\omega_{p}=\theta_{p}\cdot \diag{1,\det(\omega_{p})},$
with $\theta_{p}\in \SL_{2}(\mathbb{Z}_{p}).$ 
Let $U_n^p< \SL_2(\bZ_p)$ be the subgroup consisting of elements congruent to the identity modulo $p^n$.
By the strong approximation
Theorem for $\SL_2$ (see\cite[\S 7.4]{PR94}), for any $n\in\mathbb{N}$ there exist $\gamma_{n}\in\Gamma_{\infty}$
such that for all $p\in S_{f}$
\[
\gamma_{n}^{-1}\theta_{p}\in U_{n}^{p}.
\]
Note that there exist $N=N(\on h)\in\mathbb{N}$
such that for all $n>N$ we have that the image of $\prod_{p\in S}U_{n}^{p}$ in $G_S$ lies in $a_{f}(\on h)K_Sa_{f}(\on h)^{-1}$. As $a_{f}(\on h)K_Sa_{f}(\on h)^{-1}$ is a group
that contains $\pa{\diag{1,\det(\omega_{p})}}_{p\in S_f}$ we conclude that 
$\prod_{p\in S_f}U_{n}^{p}\cdot \diag{1,\det(\omega_{p})}\subset a_{f}(\on h)K_Sa_{f}(\on h)^{-1}$.
Therefore any $\gamma_{n}$ with $n>N$ will satisfy equation (\ref{eq:nts}). This concludes the proof of the Lemma.

\end{proof}
\begin{corollary}\label{sphere again new}
For $x=\Ga_\infty g\in X_\infty$ and $\onh$ an admissible radius
another description of the sphere is given by
$$\cS_{\onh}=\set{\Ga_\infty\ga g: \ga\in\Ga_S,\height{\ga}=\onh}.$$
\end{corollary}
\begin{proof}
Similarly to the proof of Lemma~\ref{get the sphere} one can show that any element $\ga\in\Ga_S$ can be written as a product $\ga=\ga_1\diag{\on{h},1}\ga_2$, where $\ga_i\in\Ga_\infty$ and 
$\onh\in\cO_S^\times \cap\bN$. This implies first that $\height{\ga}=\onh$ and moreover, together with~\eqref{sphere again} we obtain that 
$\cS_{\onh}=\set{\Ga_\infty\ga g: \ga\in\Ga_S,\height{\ga}=\onh}$ as desired.
\end{proof}
\begin{definition}\label{branches}
Let $x\in X_\infty$ be given. Let $g_x\in G_\infty$ be a choice of a representative for $x$ so that $x=\Ga_\infty g_x$. For any choice $\om\in K_S$ we define the \textit{generalized branch} 
$\cL_{g_x,\om}\subset \cG_S(x)$ to be the set 
\begin{equation}\label{g.branches}
\cL_{g_x,\om}=\pi\pa{\set{\Ga_S(g_x,\om)a_f(\on{h}):\on{h}\textrm{ is an admissible radius}}}.
\end{equation}
When $\om$ is a rational element (i.e.\ for any $p\in S$ the $p$'th component  $\om_p$ of $\om$ satisfies $\om_p\in K_p\cap\PGL_2(\bQ)$) we call the generalized branch $\cL_{g_x,\om}$ a \textit{rational generalized branch}.
\end{definition}
The reader should think of the generalized branches as prescribed ways to go to infinity in the graph $\cG_S(x)$. 
When $S$ is composed of a single prime the generalized branches are exactly the branches on the Hecke tree that start from the root $x$. 
\begin{remark}\label{j.27}
We wish point out a few things regarding the definition of generalized branches and fix some notation that will be used in the sequel. Let $x=\Ga_\infty g_x\in X_\infty$ be given.
\begin{enumerate}
\item\label{j.27.1} For any $\om\in K_S$ and any admissible radius $\onh$ we denote $y_{\om,\onh}=\Ga_S(g_x,\om)a_f(\onh)\in X_S$, $x_{\om,\onh}=\pi(y_{\om,\onh})\in X_\infty$. 
With this notation the generalized branch $\cL_{g_x,\om}$ intersects the sphere $\cS_{\on{h}}(x)$ in a single point, namely
\begin{equation}\label{sphere and branch}
\set{x_{\om,\onh}}=\cL_{g_x,\om}\cap\cS_{\on{h}}(x).
\end{equation} 
When the generalized branch is fixed (that is when $\om$ is fixed) we sometimes denote $x_{\onh}=x_{\om,\onh}$.
We stress here the dependency on the representative $g_x$ of $x$. Note that we do not recall this dependency in the notation $x_{\om,\onh},y_{\om,\onh}$.
\item Two generalized branches $\cL_{g_x,\om_1},\cL_{g_x,\om_2}$ intersect  the sphere 
$\cS_{\on{h}}(x)$ at the same point, that is, $x_{\om_1,\onh}=x_{\om_2,\onh}$, if and only if the points $y_{\om_i,\onh}$ lie in the same fiber of $\pi$.
This is in turn equivalent to saying that the conjugation $(\om_2^{-1}\om_1)^{a_f(\on{h})}$ lies in $K_S$ (see Remark~\ref{12.9.r.1}). 
This happens if and only if the lower left coordinate of each of the components of $\om_2^{-1}\om_1$ is divisible by $\on{h}$ in the corresponding 
ring $\bZ_p$. In particular, it follows that it is divisible by any integer that divides $\on{h}$ which means by the same reasoning, that the two branches intersect all the 
spheres $\cS_{\on{h}'}$ at the same points, for any choice of admissible radius $\on{h}' $ dividing $\on{h}$. Moreover, it follows from here that given 
$\om_1,\om_2\in K_S$, the two generalized branches $\cL_{g_x,\om_i}$ are identical if and only if the quotient $\om_2^{-1}\om_1$ is an upper triangular element of $K_S$.
\item\label{j.27.4} From the above it follows that the collection of generalized branches may be identified with the quotient $K_S/B$, where $B<K_S$ denotes the group of upper
triangular elements (this identification depends of course on the choice of the representative $g_x$). 
\item  If we replace $g_x$ by another representative $\ga g_x$ for $\ga\in\Ga_\infty$, then it readily follows that for any $\om\in K_S$, 
$\cL_{\ga g_x,\om}=\cL_{g_x,\ga^{-1}\om}$. In particular, the notion of rationality of a generalized branch is well
defined.
\end{enumerate}
\end{remark}

\subsection{Periodic $A_\infty$-orbits}
\begin{definition}\label{def per}
Let $x\in X_\infty$ be a class with a periodic $A_\infty$-orbit and let $g_x\in G_\infty$ be a representative, so that $x=\Ga_\infty g_x$. We denote by
\begin{enumerate}
\item  $t_x$ the \textit{length of the period}, i.e.\ the minimal positive $t$ for which $xa_\infty(t)=x$; 
\item  $\mu_x$  the unique $A_\infty$-invariant probability measure supported on $xA_\infty$;
\item $\ga_x$ the unique element of $\Ga_\infty$ solving the equation $\ga_x^{-1}g_x=g_xa(t_x)$;
\item $\bF_x$ the quadratic extension of $\bQ$ that is generated by the eigenvalues of $\ga_x$.
\end{enumerate}
Note that $\ga_x$ depends on the choice of the representative $g_x$ and thus is only well defined up to conjugation in $\Ga_\infty$. The quadratic field $\bF_x$ on
the other hand, only depends on this conjugacy class and so is well defined.
\end{definition}

Let $x\in X_\infty$ be a class with a periodic $A_\infty$-orbit. It is straightforward to argue that any $x'\in \cG_S(x)$ has a periodic orbit as well. We are interested in understanding 
the way the orbit $x'A_\infty$ is distributed
in $X_\infty$ as $\on{d}_\cG(x,x')$ goes to $\infty$. 
\begin{remark}\label{r.j30.1}
It turns out that the answer to this question has to do with the question of whether or not the primes $p\in S$ split in the quadratic extension $\bF_x$ of $\bQ$. 
Let $\ga\in\Ga_\infty$ be a matrix such that the roots of its characteristic polynomial generate $\bF_x$. Recall that a prime $p$ splits in $\bF_x$ if and only if $\ga$ is 
diagonalizable
over $\bQ_p$. A short exercise in linear algebra shows that $\ga\in\Ga_\infty$ is diagonalizable over $\bQ_p$ if and only if it can be triangulized over $\bZ_p$.
\end{remark}
\begin{definition}\label{split}
Let $x$ be a class with a periodic $A_\infty$-orbit, $g_x$  a representative so that $x=\Ga_\infty g_x$, and $\om\in K_S$.
\begin{enumerate}
\item We say that the generalized branch $\cL_{g_x,\om}$ is \textit{degenerate} (for $S$) if there exists $p\in S_f$ such that $\om_p^{-1}\ga_x^n \om_p$ is upper triangular for some positive integer $n$ (here $\om_p$ 
is the $p$'th component of $\om\in K_S$).\footnote{This is equivalent to saying that the Lie algebra of the closure of the group generated by $\ga_x^\om$ in $G_p$ is upper triangular.}
\item We say that the class $x$ is \textit{split} (for $S$) if there exists $p\in S$ which splits over $\bF_x$. By Remark~\ref{r.j30.1},  this is equivalent to the existence of  a degenerate generalized 
branch.
\end{enumerate}
\end{definition}
 We are now ready to state our main theorem.
\begin{theorem}\label{general case}
Let $x=\Ga_\infty g_x\in X_\infty$ be such that $xA_\infty$ is periodic.
\begin{enumerate}
\item\label{g.c.1} 
 Let $\cL=\cL_{g_x,\om}$ be a non-degenerate generalized branch of the graph $\cG_S(x)$ and
  $\on{h}$  an admissible radius, then for any $\vphi_0\in\on{Lip}_\ka(X_\infty)\cap L^2(X_\infty,m_\infty)$
and any $\eps>0$ the following holds
\begin{equation}\label{on the graph}
\av{\int_{X_\infty}\vphi_0 d\mu_{x_{\onh}}-\int_{X_\infty}\vphi_0 dm_\infty}\ll_{x,S,\cL,\eps}\max\set{\norm{\vphi_0}_2,\ka} \on{h}^{-\frac{\del_0}{2}+\eps}.
\end{equation}
\item\label{g.c.2} If $x$ is non-split (i.e.\ all generalized branches are non-degenerate), the implicit constant in~\eqref{on the graph} may be chosen to be independent of  the generalized branch and we have uniform rate of equidistribution along the full graph. 
\item\label{g.c.3} If $x$ is split and  $\cL_{g_x,\om}$ is a degenerate  generalized branch, then there is a sequence of admissible radii $\on{h}_n\to\infty$ 
such that for  the sequence of classes $x_{\onh_n}$, the lengths $t_{x_{\onh_n}}$ of the orbits $x_{\onh_n}A_\infty$ are bounded and in particular, the orbits do not equidistribute. 
\item\label{g.c.1'} Rational generalized branches are always non-degenerate and so~\eqref{on the graph} holds automatically. 
\item\label{g.c.5} Nonetheless, in case $x$ is split,  
the implicit constants in~\eqref{on the graph} cannot be taken to be uniform for the rational generalized branches.
\end{enumerate}
\end{theorem} 

\section{Relations to other arguments}\label{other arguments}
Before turning to the proof of Theorem~\ref{general case} we wish to make some comments that will clarify its relation to arguments giving equidistribution of collections of periodic orbits.
The result of Benoist and Oh~\cite[Theorem 1.1]{BenoistOhGafa} imply that given a class $x$ with a periodic $A_\infty$-orbit, then the collection of orbits $\set{x' A_\infty:x'\in \cS_{\on{h}}(x)}$ (counted without multiplicities) is becoming equidistributed as $\on{h}\to\infty$.

Ignoring the effectivity of Theorem~\ref{general case} and just interpreting it as saying that $\mu_{x'}\to m_\infty$ as $x'$ drifts away from the root $x$ along a non-degenerate generalized branch, it seems tempting to think that it is  considerably stronger than the result of Benoist and Oh, as it deals with
the equidistribution of  single orbits as opposed to the equidistribution of the full collection. We will show in~\S\ref{total growth section} below that this (non-effective) equidistribution  in fact follows quite 
elementarily from the work of Benoist and Oh. Nonetheless, the argument we give for Theorem~\ref{general case} is independent of~\cite{BenoistOhGafa} and 
as far as we know the effective statements in Theorem~\ref{general case} do not follow easily from known results.
\subsection{Total vs.\ individual growth}\label{total growth section}
Let $x\in X_\infty$ be a class with a periodic $A_\infty$-orbit and consider the union of the periodic orbits $x'A_\infty$ for $x'\in\cS_{\on{h}}(x)$ (where $\on{h}$ is an admissible radius). We denote the total
length of this union by $\textbf{t}_x(\on{h})$; that is, $\textbf{t}_x(\on{h})\defi\sum t_{x'}$ where the sum is taken over a set of representatives of the classes on the sphere giving rise to different orbits. The following Lemma shows that the growth rate of the length of individual periodic orbits along a non-degenerate generalized branch is the same as 
the growth rate of the total length. Although we only use this Lemma 
 in the course of the proofs regarding the growth rate of the periods (cf.\ Theorems~\ref{period cor new} and~\ref{gpth}) 
it explains a phenomenon that to some extent stands behind all of our results.  
Its proof is given in~\S\ref{proof of total growth}. 

\begin{lemma}\label{total growth}
Let $x\in X_\infty$ be a class with a periodic $A_\infty$-orbit. For any generalized branch $\cL$ of $\cG_S(x)$, let 
$c_\cL(\on{h})\defi t_{x_{\onh}}/\onh$, where $x_{\on{h}}$ is the class in $\cS_{\on{h}}\cap\cL$.
\begin{enumerate}
\item\label{t.g.1} The total length $\textbf{t}_x(\on{h})$ satisfies  $\on{h}\ll_{x,S}\textbf{t}_x(\on{h})\ll_{x,S} \on{h}$.
\item\label{t.g.2} If $\cL$ is a non-degenerate  generalized branch then
$c_{\cL}(\on{h})$ attains only finitely many values and moreover, if $\on{h}_n$ is a divisibility sequence of admissible radii 
(that is $\on{h}_n|\on{h}_{n+1}$), then $c_{\cL}(\on{h}_n)$ stabilizes. 
\item\label{t.g.3} The class $x$ is non-split for $S$ if and only if 
\begin{equation}\label{9.9.1}
\inf\set{c_{\cL}(\on{h}):\cL\textrm{ non-degenerate},\;\on{h}\in\cO_S^\times\cap \bN} >0.
\end{equation}
\end{enumerate}
\end{lemma}
The first two parts of Lemma~\ref{total growth} show that if $\cL$ is a non-degenerate generalized branch, then a single orbit $x_{\on{h}}A_\infty$ through the class $x_{\on{h}}\in \cL\cap\cS_{\on{h}}(x)$ actually occupies a positive proportion (bounded 
below by a constant independent of $\on{h}$) of the 
full collection $\set{x' A_\infty:x'\in \cS_{\on{h}}(x)}$. Relying on~\cite{BenoistOhGafa} we may argue the non-effective 
version of Theorem~\ref{general case} 
(that is, that $\mu_{x_{\on{h}}}\to m_\infty$ as $\on{h}\to\infty$) in the following way: Let $\on{h}_i\to\infty$ be a sequence of 
admissible radii such that $\mu_{{x_{\on{h}_i}}}$ converges to say $\mu_\infty$ (which is an $A_\infty$-invariant measure). We need 
to argue that $\mu_\infty=m_\infty$. Let  $\eta_{\on{h}}$ be the natural $A_\infty$-invariant probability measure supported on the 
collection of periodic orbits 
$\set{x' A_\infty:x'\in \cS_{\on{h}}(x)}$. By~\cite{BenoistOhGafa} $\eta_{\on{h}}\to m_\infty$. By the first two parts of Lemma~\ref{total growth} we can write $\eta_{\on{h}_i}$ as a convex combination of $A_\infty$-invariant probability measures in 
the following way: $\eta_{\on{h}_i}=c'_{\on{h}_i}\mu_{x_{\on{h}_i}}+(1-c'_{\on{h}_i})\nu_{\on{h}_i},$ where 
the constants $c'_{\on{h}_i}$ are bounded below by some constant $c'$ independent of $\on{h}_i$. Taking $i$ to $\infty$ (along an appropriate subsequences if necessary) we deduce that in the limit $m_\infty=c'_\infty\mu_\infty+(1-c'_\infty)\nu_\infty$ for some positive constant $c'_\infty\le 1$. By the ergodicity of $m_\infty$ with  respect to the $A_\infty$-action we deduce that the limit $\mu_\infty$
that appears in the above convex combination with positive weight, must be equal
to $m_\infty$. This establishes the desired convergence. 
\section{Proof of Theorem~\ref{general case}.}\label{soft version}
Throughout this section we fix $x\in X_\infty$ to be a class with a periodic $A_\infty$-orbit and a 
representative $g_x\in G_\infty$ such that $x=\Ga_\infty g_x$. Using
the notation of Definition~\ref{def per}, it follows that there exists $\ga_x\in\Ga_\infty$ such that 
\begin{equation}\label{identity 1}
\ga_x g_x a_\infty(t_x)=g_x.
\end{equation}
We briefly discuss the relations between the various parts of Theorem~\ref{general case}.
As the eigenvectors of $\ga_x$ are irrational (and not roots of unity) it follows that $\ga_x$ (or any of its powers) is not triangulizable over $\bQ$ and so all the rational generalized branches are non-degenerate. 
This establishes part~\eqref{g.c.1'} of the theorem. 
Part~\eqref{g.c.5} of the theorem follows from part~\eqref{g.c.3} because of~\eqref{sphere} 
which shows that any class on the $S$-Hecke graph $\cG_S(x)$ lies on a rational 
generalized branch; the sequence $x_{\onh_n}$ produced by part~\eqref{g.c.3} may be viewed as a sequence of classes lying on (varying) rational generalized branches, showing
that a uniform implicit constant for all rational generalized branches in~\eqref{on the graph} is impossible.

We begin with the necessary preparations for the arguments
yielding parts~\eqref{g.c.1},\eqref{g.c.2}, and~\eqref{g.c.3}. We will see below that part~\eqref{g.c.3} is a simple observation once the stage is set correctly and so
the main bulk of the theorem lies in establishing parts~\eqref{g.c.1} and~\eqref{g.c.2}. 

After fixing $g_x$ we fix a generalized branch in $\cG_S(x)$; that is, we fix an element $\om\in K_S$ and set $\cL_\om=\cL_{g_x,\om}$. Although $\om$ is fixed, the reader should bear in mind that 
at some point we will vary the choice of $\om$ in order to change the generalized branch. 
\subsection{The lift of a closed loop}\label{the lifts of a closed loop} The following construction is fundamental to our argument. 
Let 
 $y_{\om}\in X_S$ be defined by $y_{\om}=\Ga_S(g_x ,\om)$. Consider the orbit $y_{\om} A_\infty\subset X_S$ and note that 
\begin{equation}\label{lifted orbits}
\pi(y_{\om})=x,\;x A_\infty=\pi(y_{\om} A_\infty)=\pi(\overline{y_{\om} A_\infty}),
\end{equation}
where the rightmost equality follows from the fact that $x A_\infty$ is compact and the continuity of the projection $\pi$. 

We now analyze the closure $\overline{y_{\om} A}_\infty$. Each $t\in\bR$ can be written in a unique way in the form $t=s+\ell t_x$ for some 
$s\in[0,t_x)$ and $\ell\in\bZ$. It follows from~\eqref{identity 1} that 
\begin{equation}\label{identity 2}
y_{\om} a_\infty(t)=\Ga_S(g_x a^{\ell}_\infty(t_x)a_\infty(s),\om)=\Ga_S(g_x a_\infty(s),\ga_x^\ell \om)=y_{\om}(a_\infty(s),\om^{-1}\ga_x^\ell \om).
\end{equation} 
If we denote for an element $\ga$ in a group $H$ by $\idist{\ga}_H$ the cyclic group generated by $\ga$ in $H$, then it follows from~\eqref{identity 2} that 
\begin{equation}\label{the lift is an orbit}
y_{\om} A_\infty=y_{\om}(A_\infty \times\idist{\om^{-1}\ga_x \om}_{G_S}).
\end{equation}
Let 
\begin{equation}\label{H beta}
H_{\om}=\om^{-1}\overline{\idist{\ga_x}}_{G_S}\om=\overline{\idist{\om^{-1}\ga_x \om}}_{G_S}.
\end{equation}
Clearly, $H_{\om}$ is a compact subgroup of $K_S$. We let
\begin{equation}\label{L beta}
L_{\om} = A_\infty\times H_{\om}.
\end{equation}
\begin{lemma}\label{lifted orbit closure}
 The orbit $y_{\om} L_{\om}$ is compact and  
\begin{equation}\label{n'th orbit}
\overline{y_{\om} A_\infty}=y_{\om} L_{\om}.
\end{equation}
\end{lemma}
\begin{proof}
We first establish~\eqref{n'th orbit}. The inclusion $\supset$ follows readily from~\eqref{the lift is an orbit}. For the reverse inclusion, let $t_n\in\bR$ be such that $y_{\om} a_\infty(t_n)\to_{n\to\infty} y\in \overline{y_{\om} A_\infty}.$ Let $s_n\in[0,t_x),\ell_n\in\bZ$ be as defined before~\eqref{identity 2}; that is $t_n=s_n+\ell_n$. By compactness we may assume without loss of generality (after passing to a subsequence if necessary) that $s_n\to s$ and $\om^{-1}\ga_x^{\ell_n}\om\to h$. We conclude from 
~\eqref{identity 2} that 
\begin{equation}\label{identity 3}
y=\lim y_{\om} a_\infty(t_n)=
\lim y_{\om} (a_\infty(s_n),
\om^{-1}\ga_x^{\ell_n}\om)=
y_{\om}(a_\infty(s),h)\in y_{\om} L_{\om}.
\end{equation} 
The fact that the orbit $y_{\om} L_{\om}$ is compact now follows from the fact that it is a closed set contained in $\pi^{-1}(x A_\infty)$ which is compact by the properness of $\pi$.
\end{proof}
\begin{remark}\label{only a segment}
The above proof actually establishes a bit more: We have shown that in fact, 
\begin{equation}\label{segment}
\overline{y_{\om} A}_\infty=y_{\om}L_{\om}=\set{y_{\om}(a_\infty(t),h):t\in[0,t_x),h\in H_{\om}}.
\end{equation}
\end{remark}
\begin{definition}\label{periodic measures}
Let $\eta_{\om}$ denote the $L_{\om}$-invariant probability measure supported on the compact (and hence periodic) orbit $y_{\om} L_{\om}$. For an admissible radius $\on{h}$ let
$$y_{\om,\on{h}}=y_{\om}a_f(\on{h}), \;L_\om^{a_f(\on{h})}=L_{\om,\on{h}},\;H_{\om,\on{h}}=H_\om^{a_f(\on{h})},$$ and note the identity $y_\om L_\om a_f(\on{h})=y_{\om,\on{h}}L_{\om,\on{h}}=y_{\om,\on{h}}(A_\infty\times H_{\om,\on{h}})$. We denote the unique $L_{\om,\on{h}}$-invariant probability measure supported on the periodic
orbit $y_{\om,\on{h}} L_{\om,\on{h}}$ by $\eta_{\om,\on{h}}$. It follows that $(a_f(\on{h}))_*\eta_\om=\eta_{\om,\on{h}}$. Note that the notation $y_{\om,\onh}$ is consistent with the one introduced in Remark~\ref{j.27}\eqref{j.27.1}.
\end{definition}
\begin{lemma}\label{eta is a lift}
Let $\on{h}$ be an admissible radius and $x_{\om,\onh}\in \cL_\om\cap\cS_{\on{h}}(x)$. Then, 
the pushed orbit $y_{\om}L_{\om}a_f(\on{h})=y_{\om,\on{h}}L_{\om,\on{h}}$ projects to  the periodic orbit $x_{\om,\onh}A_\infty$ and furthermore, the measure $\eta_{\om,\on{h}}$ supported on it 
projects to $\mu_{x_{\om,\onh}}$; i.e.\  $\pi_*\eta_{\om,\on{h}}=\mu_{x_{\om,\onh}}$.
\end{lemma}
Lemma~\ref{eta is a lift} puts us in a desirable situation from the dynamical point of view; instead of studying the orbits $x'A_\infty$ in the space $X_\infty$ as $x'$ drifts away
from $x$ on a generalized branch (the connection between which is not clear apriori), 
we will study the images of the fixed orbit $y_{\om}L_{\om}$ under the action of $a_f(\on{h})$ for admissible radii $\on{h}$, which share a clear algebraic (and geometric) relation. This relation is the reason we needed to introduce the $S$-arithmetic extension $X_S$. 
\begin{proof}
The fact that $x_{\om,\onh}=\pi(y_{\om,\on{h}})$ follows from Definition~\ref{periodic measures} and Remark~\ref{j.27}\eqref{j.27.1}. We have that 
\begin{align}\label{identity 4}
\pi(y_{\om,\on{h}}L_{\om,\on{h}})&=\pi(y_{\om} L_{\om}a_f(\on{h}))=\pi(\overline{y_{\om} A_\infty a_f(\on{h})})\\
\nonumber &=\overline{\pi(y_{\om}  a_f(\on{h})A_\infty)}=\overline{\pi(y_{\om,\on{h}})A_\infty}=x_{\om,\onh} A_\infty, 
 \end{align}
where the  first equality from the left follows from Definition~\ref{periodic measures}, the second, from Lemma~\ref{lifted orbit closure} and that fact that $a_f(\on{h})$ acts on $X_S$ by a homeomorphism, the third,
from the commutation of $A_\infty$ and $a_f(\on{h})$
and from the continuity of $\pi$, the fourth, from the fact that $\pi$ intertwines the $A_\infty$-actions on $X_S,X_\infty$, and finally the fifth equality follows from the fact that the orbit $x_{\om,\onh}A_\infty$ 
is compact. 

As $A_\infty< L_{\om,\on{h}}$, $\eta_{\om,\on{h}}$ is $A_\infty$-invariant. As a consequence, the projection $\pi_*\eta_{\om,\on{h}}$
is an $A_\infty$-invariant probability measure supported on $x_{\om,\onh}A_\infty$. As $\mu_{x_{\om,\onh}}$ is the unique such measure, we conclude that $\pi_*\eta_{\om,\on{h}}=\mu_{x_{\om,\onh}}$ as desired.
\end{proof}
\begin{remark}\label{regarding the length}
It follows from~\eqref{segment} and the definition of $y_{\om,\on{h}},H_{\om,\on{h}}$ that  
$$y_{\om,\on{h}}L_{\om,\on{h}}=\set{y_{\om,\on{h}}(a_\infty(t),h):t\in[0,t_x),h\in H_{\om,\on{h}}}.$$
By~\eqref{identity 4} the following equality follows:
\begin{equation}\label{segment again}
x_{\om,\onh}A_\infty=\pi\pa{\set{y_{\om,\on{h}}(a_\infty(t),h):t\in[0,t_x),h\in H_{\om,\on{h}}}}.
\end{equation}
The meaning of the above equation is that the only reason for the orbit $x_{\om,\onh}A_\infty$ to become long is that the group $H_{\om,\on{h}}$ stretches and `sticks out' of $K_S$. This is illustrated in the following proof.
\end{remark}
\begin{proof}[Proof of part~\eqref{g.c.3} of Theorem~\ref{general case}]
For an admissible radius $\on{h}$ and $p\in S$ denote by $(H_{\om,\on{h}})_p$ the  projection of the group $H_{\om,\on{h}}$ on its $p$-th component. Note that by definition, 
$(H_{\om,\on{h}})_p=\diag{1,\on{h}^{-1}} (H_\om)_p\diag{1,\on{h}}$. 

Assume that the generalized branch $\cL_\om$ is degenerate. It follows that there exists $p\in S$  for which some power of the $p$-th component $(\om^{-1}\ga_x\om)_p$ is upper triangular.
Let $d$ be the minimal positive integer for which $(\om^{-1}\ga_x^d\om)_p$ is upper triangular.  We conclude from~\eqref{H beta} that $(H_{\om})_p$ contains an index $d$ subgroup that consists of upper triangular elements only. Choose $\on{h}_n=p^n$ and note that because of the above $(H_{\om,\onh_n})_p\cap K_p$ is of index at most $d$ in $(H_{\om,\onh_n})_p$. Moreover, note
that as $p$ is a unit in $\bZ_{p'}$ for any prime $p'\ne p$, we have that $(H_{\om,\onh_n})_{p'}<K_{p'}$. It follows that along the chosen sequence $\on{h}_n$ we have that $H_{\om,\on{h}_n}\cap K_S$ has at most index $d$ in $H_{\om,\on{h}_n}$. Let $h_i\in H_{\om,\on{h}_n}, i=1\dots d'$, $d'\le d$, be representatives of the cosets of $H_{\om,\on{h}_n}\cap K_S$ and
denote $y_i=y_{\om,\onh_n}h_i$, $i=1\dots d'$ and $x_i=\pi(y_i)$.
 We can rewrite~\eqref{segment again} as 
\begin{align}\label{segments 10}
\nonumber x_{\om,\onh_n}A_\infty&=\pi\pa{\cup_{i=1}^{d'}\set{y_{\om,\on{h}_n}(a_\infty(t),h_ih):t\in[0,t_x),h\in H_{\om,\on{h}_n}\cap K_S}} \\
\nonumber &=\pi\pa{\cup_{i=1}^{d'}\set{y_i(a_\infty(t),h):t\in[0,t_x),h\in H_{\om,\on{h}_n}\cap K_S}} \\
&=\cup_{i=1}^{d'} \set{x_ia_\infty(t):t\in [0,t_x)},
\end{align}
and so we conclude that  $t_{x_{\om,\onh_n}}\le d't_x$ which finishes the proof.
\end{proof}
In order to finish the proof of Theorem~\ref{general case} we are left to argue parts~\eqref{g.c.1},\eqref{g.c.2}. As said before, these are the main parts of the theorem.
\subsection{Strategy of the proof of Theorem~\ref{general case}\eqref{g.c.1},\eqref{g.c.2}}\label{strategy} In the notation of Lemma~\ref{eta is a lift}, because $\pi_*\eta_{\om,\on{h}}=\mu_{x_{\om,\onh}}$, the validity of
~\eqref{on the graph} is equivalent to saying that given $\vphi\in \on{Lip}_\ka(X_S)\cap L^2(X_S, m_S)$ which is $K_S$-invariant 
(i.e.\ is of the form $\vphi_0\circ \pi$ for $\vphi_0\in \on{Lip}_\ka(X_\infty)\cap L^2(X_\infty, m_\infty)$)  
\begin{equation}\label{short 1}
\av{\int \vphi d\eta_{\om,\on{h}}-\int \vphi dm_S }\ll_{x,S_f,\cL_\om,\eps}\max\set{\ka,\norm{\vphi}_2}\on{h}^{-\frac{\del_0}{2}+\eps}.
\end{equation}
The argument giving this `effective equidistribution' is a combination of an argument which we will refer to as \textit{the mixing trick} and spectral gap (or effective decay of matrix coefficients). 
As far as we know the mixing trick originates  from Margulis' thesis~\cite{MargulisThesis}.
We briefly describe its heuristics: One slightly thickens the initial orbit $y_{\om} L_{\om}$
to an open set $\cT \subset X_S$ in directions which are (weakly) contracted by the action of $a_f(\on{h})$. The set $\cT$ will be called below \textit{a tube around the orbit} $y_{\om} L_{\om}$. Let $m_{\cT}$ denote the normalized restriction of $m_S$ to $\cT$. The pushed measure $(a_f(\on{h}))_*m_{\cT}$ is the normalized restriction of $m_S$ to the pushed tube $\cT a_f(\on{h})$, which is a tube around 
the orbit $y_{\om,\on{h}} L_{\om,\on{h}}$. 
Because the thickening used to construct $\cT$ is taken in directions which are (weakly) contracted by $a_f(\on{h})$, the size of the thickening giving the tube 
$\cT a_f(\on{h})$ is even smaller than the size of the initial thickening.
Hence, there shouldn't be much of a difference between integrating against the measure $\eta_{\om,\on{h}}$ and integrating against $(a_f(\on{h}))_*m_{\cT}$. The fact that the action of $a_f(\on{h})$ is mixing on $(X_S,m_S)$ means that 
the  pushed measure $(a_f(\on{h}))_*m_{\cT}$ is `close' to $m_S$  (here, the effective mixing Theorem~\ref{effective dmc} will allow us to pin down the meaning 
of `close' in a precise way). Combining these things together will give us the desired estimate given in~\eqref{short 1}.

In order to make this strategy into a
rigorous proof we discuss in the next two subsections in detail the construction of tubes and decay of matrix coefficients.  
\subsection{Effective mixing}\label{effective mixing}
Let $\cH=L^2(X_S,m_S)$. Our goal in this section is to prove the following:
\begin{theorem}\label{effective dmc}
Let $\on{h}$ be an admissible radius and $w_1,w_2\in\cH$ be vectors with the following properties: $w_1$ is $K_S$-fixed and $w_2$ is stabilized by a product subgroup $K^*=\prod_{v\in S} K^*_v<K_S$ of index $d$ in $K_S$. Then
 for any $\eps>0$, 
\begin{align}\label{effective dmc inequality}
\av{\idist{w_1,a_f(\on{h})w_2}-\idist{w_1,1}\idist{1,w_2}}\ll_\eps
 & \norm{w_1}\norm{w_2} d^{\frac{1}{2}}\on{h}^{-\del_0+\eps},
\end{align}
\end{theorem}
The meaning of the exponent $\del_0$ that appears in~\eqref{effective dmc inequality} will be explicated shortly. Before turning to the proof of the above theorem, we need to discuss three lemmas.
 For $v\in S$ let $\cH_v$ denote the orthocomplement of the $G_v$-invariant functions in $\cH$.  The following is~\cite[Lemma 9.1]{VenkateshSED}. It is the key input in the proof of Theorem~\ref{effective dmc}.
\begin{lemma}\label{Vlemma}
Let $w_1,w_2\in\cH_v$ $(v\in S)$ be two vectors which are stabilized respectively by finite index subgroups $K^{(1)},K^{(2)}$ of $K_v$,  let $d_i=[K_v,K^{(i)}]$, and $a_v(t)=\diag{1,t}, t\in\bQ_v^\times.$ Then 
the following holds
\begin{equation}\label{Venkatesh}
\av{\idist{w_1,a_v(t)w_2}}\ll_\eps\norm{w_1}\norm{w_2}d_1^{\frac{1}{2}}d_2^{\frac{1}{2}}\max\set{\av{t}_v,\av{t^{-1}}_v}^{-\del_0+\eps}.
\end{equation}
\end{lemma}
The exponent $\del_0$ comes from the following discussion.
Let $\rho_v$ be the unitary representation of $G_v$ on $\cH_v$. Let $\sig_0$ be the smallest 
number so that no complementary series representations of parameter $\ge \sig_0$ is weakly contained in $\rho_v$. Here we  follow~\cite{VenkateshSED} and 
parametrize the complementary series representations by the parameter $\sig\in (0,\frac{1}{2})$; so $\sig_0=0$ corresponds to $\rho_v$ being tempered (the Ramanujan conjecture) and 
$\sig_0=\frac{1}{2}$ corresponds to $\rho_v$ having no almost invariant vectors. The best bound known today towards Ramanujan is given by Kim and Sarnak 
in the appendix of~\cite{KimSarnak} and establishes the bound $\sig_0\le \frac{7}{64}$.  The exponent $\del_0$ that appears in Lemma~\ref{Vlemma} and that appears in our results is defined by 
\begin{equation}
\del_0=\frac{1}{2}-\sig_0,
\end{equation}
so the Kim-Sarnak bound reads as $\del_0\ge \frac{25}{64}$.

Lemma~\ref{Vlemma} is stated for one place $v\in S$ but in Theorem~\ref{effective dmc} we wish to take advantage of the various places $\onh$ is supported on. In order to do this, we will need to use Lemma~\ref{Vlemma} iteratively and the following abstract lemma in Hilbert space theory allows us to do so.
\begin{lemma}\label{ast lemma}
Let $G=G_1\times G_2$ be a group acting unitarily on a Hilbert space $\cH$. Let $K_i<G_i$ be subgroups, $g_i\in G_i$ be two given elements, and $F(g_i)$ two positive numbers satisfying 
the following statement: For each $i$, if $v,w\in\cH$ are $K_i$-fixed vectors, then $$\idist{g_iv,w}\le\norm{v}\norm{w} F(g_i).$$
Then for any $v,w\in\cH$ which are $K_1\times K_2$-fixed we have that $$\idist{g_1g_2v,w}\le\norm{v}\norm{w} F(g_1)F(g_2).$$
\end{lemma}
\begin{proof}
Let us denote for $i=1,2$ $V_i=\set{v\in \cH:v \textrm{ is }K_i\textrm{-fixed}}$ and $U=V_1\cap V_2$. Let $V_i'$ denote the orthocomplement of $U$ in $V_i$ and denote for a subspace $W$ of $\cH$ by
$P_W$ the orthogonal projection on $W$. We first note that $V_1,V_2$ are $K_2,K_1$-invariant respectively (because $K_1,K_2$ commute) and so the projections $P_{V_1},P_{V_2}$ commute 
with the actions of $K_2,K_1$ respectively. It follows from here that given $v_1\in V_1$, say, the projection $P_{V_2}(v_1)$ is fixed by both $K_1$ and $K_2$ i.e.\ $P_{V_2}(v_1)\in U$. This proves that $V_1'$ is orthogonal to
$V_2$ or in a more symmetric manner, $V_1'$ is orthogonal to $V_2'$.

Let now $v,w$ be two $K_1\times K_2$-fixed vectors. As $g_1v$ is $K_2$-fixed, i.e.\ $g_1v\in V_2$, we may write $g_1v=P_U(g_1v)+P_{V_2'}(g_1v)$ and similarly $g_2w=P_U(g_2w)+P_{V_1'}(g_2w)$. 
It follows that 
\begin{align}\label{16.9.1}
\nonumber \idist{g_1v,g_2w}&=\idist{P_U(g_1v)+P_{V_2'}(g_1v),P_U(g_2w)+P_{V_1'}(g_2w)}\\
&=\idist{P_U(g_1v),P_U(g_2w)}\le\norm{P_U(g_1v)}\norm{P_U(g_2w)}.
\end{align}
Let $\tilde{v}=\frac{P_U(g_1v)}{\norm{P_U(g_1v)}}$. Then $\tilde{v}$ is $K_2$-fixed and so by the assumption of the lemma we conclude that 
$$\norm{P_U(g_1v)}=\idist{g_1v,\tilde{v}}\le \norm{v}F(g_1).$$
Similarly, $\norm{P_U(g_2w)}\le\norm{w}F(g_2)$. Plugging this into~\eqref{16.9.1} yields $$\idist{g_1v,g_2w}\le\norm{v}\norm{w}F_1(g_1)F_2(g_2),$$ which is equivalent to the desired statement up
to replacing $g_2$ by its inverse (note that the assumption on $g_i$ implies the corresponding assumption on $g_i^{-1}$).
\end{proof}
The final ingredient needed for the proof of Theorem~\ref{effective dmc} is the following
\begin{lemma}\label{ergodicity 10}
For each place $v\in S$ the group generated by $G_v$ and $K_S$ acts ergodically on $X_S$, that is, $\set{w\in \cH: w \textrm{ is both }G_v,K_S\textrm{-fixed}}$ is 
 the one dimensional space of constant functions.
\end{lemma}
\begin{proof}
Recall that $S^*=S\cup\set{\infty}.$
Let $Y_S=\SL_2(\cO_S)\backslash \prod_{v\in S^*}\SL(\bQ_v)$. The strong approximation property for $\SL_2$ implies that for any $v\in S$ the lattice $\SL_2(\cO_S)$ embeds 
densely in $\prod_{v'\in S^*\smallsetminus\set{v}}\SL_2(\bQ_{v'})$. This is equivalent to saying that $\SL_2(\bQ_v)$ acts minimally on $Y_S$ (i.e.\ that any orbit is dense). In turn, this implies that $\SL_2(\bQ_v)$
acts ergodically on $Y_S$ (by the duality trick for example). Now, consider the natural map $\psi:\SL_2\to \PGL_2$. This map induces a map from $Y_S$ to $X_S$ (which we also denote by $\psi$) 
which intertwines the actions of $\SL_2(\bQ_v)$ and $\psi(\SL_2(\bQ_v))<G_v$ on these spaces respectively. It follows that the action of $\psi(\SL_2(\bQ_v))$ on $\psi(Y_S)$ is ergodic.

Let $w\in \cH$ be a function on $X_S$ which is both $G_v$ and $K_S$-invariant. Its restriction to $\psi(Y_S)$ is constant by the ergodicity proved above. It follows that in order to show that
 $w$ is constant it is enough to show that the translates of $\psi(Y_S)$ by $K_S$ cover $X_S$. We briefly sketch the argument: There is a natural `determinant map' 
 $\det: G_{S^*}\to\prod_{v\in S^*}\bQ_v^\times/ (\bQ_v^\times)^2$. Let us denote $\Del=\prod_{v\in S^*}\bQ_v^\times/ (\bQ_v^\times)^2$ and $\Del'=\det(\Ga_S)<\Del$. 
 It follows that there is a well defined map $\widetilde{\det}:\Ga_S\backslash G_{S^*}=X_S\to
 \Del'\backslash\Del$. We leave it to the reader to show that the space $\psi(Y_S)$ is characterized as the preimage of the identity coset $\Del'$ under $\widetilde{\det}$. Since $\det$ takes $K_S$ onto
 $\Del'\backslash \Del$, we conclude that indeed, translates of $\psi(Y_S)$ under $K_S$ cover $X_S$ as desired.
\end{proof}
\begin{proof}[Proof of Theorem~\ref{effective dmc}]
Let $\cH=L^2(X_S,m_S)$ and for $v\in S$ let $\cH_v$ be the orthocomplement to the $G_v$-fixed vectors. Let $\cH_0=\cap_{v\in S} \cH_v$ and let $w_1,w_2\in\cH$ be as in the statement of the theorem.
 Write $$w_i=P_{\cH_0}(w_i)+P_{\cH_0^\perp}(w_i),$$
and note that the decomposition $\cH=\cH_0+\cH_0^\perp$ is $G_{S}$-invariant. It follows that 
\begin{align}\label{decay 12}
\nonumber \idist{a_f(\onh)w_1,w_2}&=\idist{a_f(\onh)\pa{P_{\cH_0}(w_1)+P_{\cH_0^\perp}(w_1)},P_{\cH_0}(w_2)+P_{\cH_0^\perp}(w_2)}\\
&=\underbrace{\idist{a_f(\onh) P_{\cH_0}(w_1),P_{\cH_0}(w_2)}}_{(*)}+
\underbrace{\idist{a_f(\onh) P_{\cH_0^\perp}(w_1),P_{\cH_0^\perp}(w_2)}}_{(**)}.
\end{align}
Let us first argue that $(**)=\idist{w_1,1}\idist{1,w_2}$. 
The space $\cH_0^\perp$ is the space generated by $\set{\cH_v^\perp}_{v\in S}$. This implies that the vector $P_{\cH_0^\perp}(w_1)$ is in the span of the vectors $P_{\cH_v^\perp}(w_1)$ as  $v$ runs through $S$. For each $v\in S$ the vector $P_{\cH_v^\perp}(w_1)$ is both $G_v$ and $K_S$-fixed and so 
by Lemma~\ref{ergodicity 10} this implies that $P_{\cH_v^\perp}(w_1)\in \cH_c$, where $\cH_c$ denotes here the 1-dimensional space of constant functions. 
We conclude that $P_{\cH_0^\perp}(w_1)\in\cH_c$, or in other words, $P_{\cH_0^\perp}(w_1)=P_{\cH_c}(w_1)=\idist{w_1,1}$. Using this we see that 
\begin{equation}\label{decay 14}
(**)=\idist{\idist{w_1,1},P_{\cH_0^\perp}(w_2)}=\idist{w_1,1}\idist{1,P_{\cH_0^\perp}(w_2)}.
\end{equation} 
In turn, 
$\idist{P_{\cH_0^\perp}(w_2),1}$ is the orthogonal projection of $P_{\cH_0^\perp}(w_2)$ on $\cH_c$, but as $\cH_c\subset \cH_0^\perp$, this projection equals $\idist{w_2,1}$. We conclude from~\eqref{decay 14} that $(**)=\idist{w_1,1}\idist{1,w_2}$ as claimed.

We now analyze $(*)$ in~\eqref{decay 12}. Because the decomposition $\cH=\cH_0+\cH_0^\perp$ is $G_S$-invariant the vectors $P_{\cH_0}(w_1),P_{\cH_0}(w_2)$ are fixed under  $K_S, K^*$
respectively (where $K^*$ is as in the statement of the theorem). Order the primes in $S$ in some way $p_1\dots p_k$ and denote $d_{p_i}=[K_{p_i}:K^*_{p_i}]$, so $d=[K_S:K^*]=\prod_{i=1}^k d_{p_i}$. We leave it to the reader to prove by a simple induction, using Lemmas~\ref{Vlemma}, \ref{ast lemma}
that for $j=1,\dots, k$  
\begin{equation}\label{idn decay}
\idist{\prod_{i=1}^j a_{p_i}(\onh) P_{\cH_0}(w_1),P_{\cH_0}(w_2)}\ll_\eps \norm{P_{\cH_0}(w_1)}\norm{P_{\cH_0}(w_1)}\prod_{i=1}^j d_{p_j}^{\frac{1}{2}}\prod_{i=1}^k\av{\onh}_{p_i}^{-\del_0+\eps}.
\end{equation}
In particular, for $j=k$ we obtain 
\begin{equation}\label{decay 10}
(*)=\idist{a_f(\onh) P_{\cH_0}(w_1),P_{\cH_0}(w_2)}\ll_\eps \norm{w_1}\norm{w_2} d^{\frac{1}{2}} \onh^{-\del_0+\eps}.
\end{equation}
Equations~\eqref{decay 10},\eqref{decay 12} and the analysis carried above for $(**)$ now imply the validity of the theorem.
\end{proof}
\subsection{Tubes}
As explained in~\S\ref{strategy}, we start with a $K_S$-invariant `test function' $\vphi$ and we need to thicken the orbit $y_{\om} L_{\om}$ to a tube $\cT$ and then apply~\eqref{effective dmc inequality} to
the vectors $w_1=\vphi,w_2=\chi_\cT$. In order for the use of~\eqref{effective dmc inequality} to be meaningful we need to control $d$ which is the index of the stabilizer of the tube in $K_S$. Also,
the `width' of the tube (i.e. the size of the thickening of the orbit) should be very small (at least in the real component) in order for the heuristics of~\S\ref{strategy} to take effect. This will hopefully motivate the constructions in this subsection.   
\begin{definition}\label{definition of tubes}
Let $yL\subset X_S$ be a compact orbit of a closed subgroup $L<G_{S^*}$. Let $V=\oplus_{v\in S^*}V_v$ be a linear complement to $\lie(L)$ in $\gog_{S^*}$. Let $U\subset V$ be a small enough 
open neighborhood of $0$ so that the map $yL\times U\to X_S$ defined by $(z,u)\mapsto z\exp_{S^*}(u)$ is a homeomorphism onto its image and its image is 
open
in $X_S$.  The set $$\cT_U(yL)=\set{z\exp_{S^*}(u):z\in yL, u\in U}$$ is called a \textit{tube around the orbit $yL$ of width $U$}. We often denote the tube simply by $\cT$. 
The width $U$ and the tube $\cT$ are said to come from $V$.
\end{definition}
A tube $\cT_U(yL)$ gives us a coordinate system; a point of $\cT$ can be written uniquely  as $z\exp_{S^*} u$. We refer to $z$ as the \textit{orbit coordinate} and to $u$ as the \textit{width coordinate}.
We shall need a few lemmas about tubes which we now turn to describe.
\subsubsection{Measures on tubes} 
Given a tube $\cT=\cT_U(yL)$ around the compact orbit $yL$ coming from $V$, one could construct the following two natural probability measures supported on $\cT$. The first is the normalized Haar measure
$\frac{1}{m_S(\cT)}m_S|_\cT$ which we will denote by $m_\cT$. The second is the (pushforward of) the product measure $\eta\times m_U$ on $yL\times U\simeq\cT$, where $\eta$ is the unique $L$-invariant probability measure on the orbit 
$yL$ and $m_U$ is the normalized restriction of the Haar measure on $V$ to $U$ (that is $m_U=\frac{1}{m_{V}(U)}m_{V}|_U$). We shall need to understand to some extent the connection between these two measures.
\begin{lemma}\label{absolute continuity}
The measure $m_\cT$ is absolutely continuous with respect to $\eta\times m_U$. Moreover, if we denote by $F(z,u)$ the Radon-Nikodym derivative; that is $dm_\cT=F(z,u)d\eta(z)d m_U(u)$, then for $\eta$-almost any $z\in yL$, $\int_U F(z,u)d m_U(u)=1$. 
\end{lemma} 
\begin{proof}
The absolute continuity is left to be verified by the reader. As for the claim about the density $F$, we argue as follows. 
Let $\vphi(z)=\int_U F(z,u)d m_U(u).$  We will show that $\vphi$ is constant $\eta$-almost surely. As $\int_{yL}\vphi(z)d\eta(z)=m_\cT(\cT)=1$ this constant must be equal to one. 

Choose a fundamental domain $\cE$ in $L$ for the orbit $yL$ and identify it with the orbit. Note that with this identification $\eta$ is just the restriction to $\cE$ of a Haar measure\footnote{Note that 
$L$ must be unimodular, hence this measure is both left and right invariant.} on $L$ scaled
so that $\eta(\cE)=1$.
Assume to get a contradiction that $\vphi$ is not constant $\eta$-almost surely. It follows that there are constants $c_2<c_1$ so that the sets 
$E_1=\set{h\in\cE:\vphi(yh)>c_1}, E_2=\set{h\in \cE:\vphi(yh)<c_2}$ are of positive $\eta$-measure.  
There exists $h_0\in L$ so that $\eta(E_1\cap h_0^{-1}E_2)>0$ and so if we let $\tilde{E}_1=E_1\cap h_0^{-1}E_2$ and $\tilde{E}_2=h_0\tilde{E}_1$, then $\tilde{E}_i\subset \cE$ are both of (the same) positive $\eta$-measure and differ from one another by left translation by $h_0$. The following calculation derives the desired contradiction:
\begin{align}\label{derive a contradiction}
\nonumber c_1\eta(\tilde{E}_1)&\le \int_{\tilde{E}_1}\vphi(z)d\eta(z)\\
&=m_S(\tilde{E}_1\exp_{S^*}(U))=m_S(h_0\tilde{E}_1\exp_{S^*}(U))=m_S(\tilde{E}_2\exp_{S^*}(U))\\
\nonumber &=\int_{\tilde{E}_2}\vphi(z)d\eta(z)\le c_2\eta(\tilde{E}_2)=c_2\eta(\tilde{E}_1).
\end{align}
\end{proof} 
Our aim now is to define the relevant family of tubes around the orbit $y_\om L_\om$ that will be of use to us. The first stage is to choose the correct linear complement from which the tubes will come.
\subsubsection{Choosing the linear complement}\label{clc}
When we come to argue the validity of Theorem~\ref{general case}\eqref{g.c.1},\eqref{g.c.2} for a given admissible radius $\on{h}$, we may assume without loss of generality
 that $S$ is the  smallest set of primes for which $\on{h}\in\cO_S^\times$. Hence, without loss of generality we may (and will) assume that $\on{h}$ is divisible by all the primes in $S$. 
 We refer to such a radius $\on{h}$ as having \textit{full support}. 
 The assumption that an admissible radius has full support is equivalent to the fact that the weak stable algebra of $a_f(\on{h})$ attains the form 
\begin{equation}\label{the ws of the positive type}
(\gog_{S^*})_{a_f(\on{h})}^{\on{ws}}=\gog_\infty \oplus_{p\in S}\set{\pa{\begin{array}{ll}*&*\\ 0&*\end{array}}\in\gog_p}.
\end{equation}
\begin{definition}\label{shrinking transverse}
Let $V=\oplus_{v\in S^*}V_v$ be defined as follows
$$V_\infty=\set{\pa{\begin{array}{ll}0&*\\ *&0\end{array}}\in\gog_\infty};\;\;\textit{For }p\in S,\;V_p=\set{\pa{\begin{array}{ll}*&*\\ 0&*\end{array}}\in\gog_p}.$$
\end{definition}
\begin{lemma}
If the generalized branch $\cL_\om$ is non-degenerate then the subspace $V\subset \gog_{S^*}$ from Definition~\ref{shrinking transverse} is indeed a linear complement of 
$\lie(L_{\om})$
which is contained in $(\gog_{S^*})^{\on{ws}}_{a_f(\on{h})}$ for any admissible radius $\on{h}$ of full support. 
\end{lemma}
\begin{proof}
The fact that $V\subset (\gog_{S^*})^{\on{ws}}_{a_f(\on{h})}$ follows from the discussion preceding Definition~\ref{shrinking transverse}. Recall that $L_{\om}=A_\infty\times H_{\om}$ where 
$H_{\om}=\om^{-1}\overline{\idist{\ga_x}}_{G_S}\om$
(see~\eqref{H beta},\eqref{L beta}). Writing $\lie(L_{\om})=\oplus_{S^*}\gol_v$ we see that $V_\infty$ indeed complements $\gol_\infty$. 
Let $\bT$ be the algebraic subgroup of $\bG$ defined as the Zariski closure of the group generated by $\ga_x$. It is a one dimensional torus and $H_{\om}$ is a compact open subgroup of the conjugation
$\om^{-1}\bT(\prod_{v\in S}\bZ_v)\om$. It follows that for any $v\in S$ the dimension of $\gol_v$ is 1 and so in order to argue that it complements $V_v$ we only need to argue that the
inclusion $\gol_v\subset V_v$ does not hold.
Such an inclusion would imply that there is a neighborhood of the identity in $H_{\om}$ that consists of upper triangular matrices, which in turn would imply that a certain power of $\om^{-1}\ga_x\om$
is upper triangular, contradicting the assumption that the generalized branch is non-degenerate.
\end{proof}
Henceforth, when speaking about a linear complement $V$ to $\lie(L_{\om})$, we shall refer only to the subspace from Definition~\ref{shrinking transverse}. 
\begin{remark}\label{translated tubes}
Because of the inclusion $V\subset(\gog_{S^*})_{a_f(\on{h})}^{\on{ws}}$ (for any admissible radius $\on{h}$ of full support), we conclude that if $U_0\subset V$ is a small enough ball around 
zero, the conjugation $U_0^{a_f(\on{h})}$ will be contained in the domain of $\exp_{S^*}$. This implies 
that for any $U\subset U_0$ the identity $(\exp_{S^*} U)^{a_f(\on{h})}=\exp_{S^*}\pa{U^{a_f(\on{h})}}$ holds. 
It follows that if $\cT$ is a tube of width $U$ coming from $V$ around $y_{\om} L_{\om}$, then if the width $U$ is chosen within $U_0$, the pushed tube $\cT a_f(\on{h})$ satisfies
\begin{align}\label{t.t}
\cT a_f(\on{h})=y_{\om} L_{\om}\exp_S(U)a_f(\on{h})=y_{\om,\on{h}}L_{\om,\on{h}}\exp_S\pa{U^{a_f(\on{h})}}.
\end{align}
That is,  $\cT a_f(\on{h})$ is a tube of width $U^{a_f(\on{h})}$ around the compact orbit $y_{\om,\on{h}}L_{\om,\on{h}}$. Below, we will make the implicit assumption that all the widths 
considered are contained in the ball $U_0$.
\end{remark}
\subsubsection{The tubes $\cT^\del_\om$}
As explained above, we will need to construct tubes with shrinking real width component and with control on the subgroup of $K_S$ that stabilizes them. After describing this family of tubes
we state a few lemmas that describe their relevant properties. The proofs of these lemmas will be postponed till after concluding the proof of Theorem~\ref{general case}.

Let us denote by $\tilde{B}$ a compact open subgroup of the group of upperr triangular elements in $K_S$ that lies in the domain 
of $\log_{S}$ and for which Remark~\ref{translated tubes} applies (that 
is, all conjugations $\tilde{B}^{a_f(\on{h})}$ are in the domain of $\log_{S}$ for admissible radii $\on{h}$ of full support).
For $\del>0$ let $B_\del^{V_\infty}$ be the ball of radius $\del$ around $0$ in the $\infty$-component
of the linear complement $V$ from Definition~\ref{shrinking transverse}.  
\begin{lemma}\label{the shrinking tube}
There exists $\hat{\del}>0$ and an open compact subgroup $B=\prod_{S} B_p$ of $\tilde{B}$, such that for all $\del<\hat{\del}$, if we let $U^\del= B_\del^{V_\infty}\times \log_{S}(B)$, 
then for any $\om\in K_S$ such that the generalized branch $\cL_\om$ is non-degenerate, the set 
$\cT_\om^\del =y_\om L_\om\exp_S(U^\del)$ is a tube around $y_\om L_\om$; that is, the map $y_\om L_\om\times U^\del\to\cT_\om^{\del}$
is a homeomorphism and the set $\cT_\om^\del\subset X_S$ is open. Furthermore, the choice of $B,\hat{\del}$ depends only on the original class $x$ and the set of places $S$ at hand. In particular,
they are independent of $\om$.
\end{lemma}
\begin{lemma}\label{again finite index}
Let $B,\hat{\del}$ be as in Lemma~\ref{the shrinking tube} and let $\om\in K_S$ be such that the generalized branch $\cL_\om$ is non-degenerate.
\begin{enumerate}
\item\label{a.f.i.1} There exists an open compact product subgroup $K^*=\prod_{S}K_v^*<K_S$ which stabilizes the tube $\cT_\om^\del$ for any $\del<\hat{\del}$; that 
is $\cT_\om^\del k=\cT_\om^\del$ for any $k\in K^*,\del<\hat{\del}$. Moreover, if $x$ is non-split, we may choose $K^*$ to be independent of $\om$.
\item\label{a.f.i.2} The measures $m_S(\cT^\del_\om)$ satisfy $m_S(\cT^\del_\om)\gg_{x,S,\cL_\om} \del^2$. If $x$ is non-split, the implicit constant may be chosen to be independent of the generalized branch.
\end{enumerate}
\end{lemma}

\subsection{Concluding the main part of the proof}
\begin{proof}[Proof of parts\eqref{g.c.1},\eqref{g.c.2} of Theorem~\ref{general case}] We follow the strategy presented in~\S\ref{strategy} 
and use freely all the notation introduced so far. Let $\on{h}$ be an admissible radius and assume without loss of generality that it is of full support.
Let $\vphi_0\in\on{Lip}_\ka(X_\infty)\cap  L^2(X_\infty,m_\infty)$. 
We let $\vphi=\vphi_0\circ\pi$ be the lift of $\vphi_0$ to $X_S$.  As $\int_{X_\infty} \vphi_0 dm_\infty=\int_{X_S}\vphi dm_S$,
we see by Lemma~\ref{eta is a lift} that part~\eqref{g.c.1} of the theorem will follow once we prove 
\begin{equation}\label{maxi}
\av{\int_{X_S}\vphi d\eta_{\om,\on{h}}-\int_{X_S}\vphi dm_S}\ll_{x,S,\cL_\om,\eps} \max\set{\norm{\vphi}_2,\ka}\on{h}^{-\frac{\del_0}{2}+\eps}.
\end{equation}
Part~\eqref{g.c.2} will follow once we establish that in the non-split case, the implicit constant in~\eqref{maxi} may be chosen independent of the generalized branch.
Let $V<\gog_{S^*}$ be the linear complement from Definition~\ref{shrinking transverse}. We apply Lemma~\ref{again finite index} and use the notation introduced there to obtain 
a family of tubes $\cT_\om^\del$ around $y_{\om} L_{\om}$ coming from $V$. 

We denote $\cT^\del_{\om,\on{h}}$ the (pushed) tube $\cT^\del_\om a_f(\on{h})$ around the orbit $y_{\om,\on{h}}L_{\om,\on{h}}$, and $m_{\cT^\del_{\om,\on{h}}}$ the normalized restriction
of $m_S$ to $\cT^\del_{\om,\on{h}}$. The width of\footnote{The reader should not confuse the superscript $\del$ with our notation for conjugation.} $\cT^\del_{\om,\on{h}}$ is $U^{\del,\on{h}}=(U^{\del})^{a_f(\on{h})}$, where
$U^\del$ is as in Lemma~\ref{the shrinking tube} (see Remark~\ref{translated tubes}). 
We have
\begin{align}\label{losses 2}
\av{\int_{X_S}\vphi d\eta_{\om,\on{h}} -\int_{X_S}\vphi dm_S}\le & \\
\nonumber & \underbrace{\av{\int_{X_S}\vphi d\eta_{\om,\on{h}}-\int_{X_S}\vphi dm_{\cT_{\om,\on{h}}^\del}}}_{(*)}+
\underbrace{\av{\int_{X_S}\vphi dm_{\cT_{\om,\on{h}}^\del}-\int_{X_S}\vphi dm_S}}_{(**)}.
\end{align}
To estimate $(*)$ we define $\tilde{\vphi}_{\del,\on{h}}:\cT_{\om,\on{h}}^\del\to\bC$ by 
$\tilde{\vphi}_{\del,\on{h}}(z\exp_{S^*} u)=\vphi(z)$ for $z\in y_{\om,\on{h}}L_{\om,\on{h}}, u\in U^{\del,\on{h}}$ and extend it to be zero outside the 
tube $\cT_{\om,\on{h}}^\del$ to obtain a function on $X_S$. 
By Lemma~\ref{absolute 
continuity} it follows that
\begin{equation}\label{using the density 1}
\int_{X_S}\tilde{\vphi}_{\del,\on{h}}dm_{\cT_{\om,\on{h}}^\del}=\int_{y_{\om,\on{h}}L_{\om,\on{h}}}\int_{U^{\del,\on{h}}}\vphi(z)F(z,u)dm_{U^{\del,\on{h}}}(u)d\eta_{\om,\on{h}}(z)=\int_{X_S}\vphi d\eta_{\om,\on{h}}
\end{equation}
We therefore have the following estimate for $(*)$ 
\begin{align}\label{effective loss 1}
\nonumber(*)&=\av{\int_{X_S}\tilde{\vphi}_{\del,\on{h}} dm_{\cT_{\om,\on{h}}^\del}-\int_{X_S}\vphi dm_{\cT_{\om,\on{h}}^\del}}\\
&\le  \max\set{\av{\vphi(y\exp_{S^*}(w))-\vphi(y)}:y\in y_{\om,\on{h}} L_{\om,\on{h}}, w\in U^{\del,\on{h}}}.
\end{align}
Note that if we write $w\in U^{\del,\on{h}}$ as $(w_\infty,w_f)$, then for any $y\in y_{\om,\on{h}}L_{\om,\on{h}}$,  $\vphi(y\exp_{S^*}(w))=\vphi_0(\pi(y)\exp_\infty (w_\infty))$ by the $K_S$-invariance of $\vphi$.
As the maps induced by the actions of elements of the form $\exp_\infty(w_\infty), \norm{w_\infty}<1$ are all Lipschitz with some uniform Lipschitz constant $c_1$,
the distance between $\pi(y)$ and $\pi(y)\exp_\infty(w_\infty)$ is $\le c_1 \del$ and so by the Lipschitz assumption of 
$\vphi_0$ we obtain 
\begin{equation}\label{final first loss}
(*)\le c_1\ka \del.
\end{equation}
We now estimate $(**)$. 
Let $\cH=L^2(X_S,m_S)$ and denote $$w_1=\vphi,\;w_2=\frac{1}{m_S(\cT^\del_\om)}\chi_{\cT^\del_\om}.$$ 
In order to appeal to Theorem~\ref{effective dmc} we observe that $w_1$ is $K_S$-fixed and $w_2$ is $K^*$-fixed, where $K^*$ is as in Lemma~\ref{again finite index}. By Lemma~\ref{again finite index} the index $d=[K_S:K^*]$
depends only on $x,S$, and $\cL_\om$ and in the non-split case 
could be bounded by a number independent of the generalized branch.
 As for the norms,
$\norm{w_1}=\norm{\vphi}$, and for $w_2$ we have $\norm{w_2}=m_S(\cT^\del_\om)^{-\frac{1}{2}}$. By Lemma~\ref{again finite index} we have that 
$m_S(\cT^\del_\om)\gg_{x,S,\cL_\om} \del^2$ and so $\norm{w_2}\ll_{x,S,\cL_\om}\del^{-1}$. Furthermore, in the non-split case, the implicit constant can be taken to be independent of the generalized branch.

It now follows from Theorem~\ref{effective dmc}  that
\begin{align}\label{second loss 2}
(**)&=\av{\idist{\vphi,a_f(\on{h})^{-1}\pa{\frac{1}{m_S(\cT^\del_\om)}\chi_{\cT^\del_\om}}}-\int_{X_S}\vphi dm_S}\\
\nonumber &=\av{\idist{a_f(\on{h})w_1,w_2}-\idist{w_1,1}\idist{1,w_2}}\\
\nonumber &\ll_{x,S,\cL_\om,\eps} \norm{\vphi}_2\del^{-1}\on{h}^{-\del_0+\eps},
\end{align}
and that in the non-split case the implicit constant can be taken independent of the generalized branch.
Combining~\eqref{losses 2},\eqref{final first loss},\eqref{second loss 2}, and choosing $\del=c\on{h}^{\frac{1}{2}(-\del_0+\eps)}$ (the meaning of $c$ will become clear in a moment) we obtain~\eqref{maxi} as desired (with $\eps$ replaced by $\frac{\eps}{2}$). Here the constant $c$ is chosen to protect us from the  possible finitely many $\on{h}$'s for which the inequality $\on{h}^{\frac{1}{2}(-\del_0+\eps)}<\hat{\del}$ does not hold ($\hat{\del}$ as in Lemma~\ref{the shrinking tube}). Note that indeed, the constant $c$ depends only on $\hat{\del},S$, and $\eps$. By Lemma~\ref{the shrinking tube} we see that it actually depends on $x,S$, and $\eps$. This concludes the proof of Theorem~\ref{general case}.
\end{proof}
\subsection{Proofs of Lemmas~\ref{the shrinking tube},\ref{again finite index}}
We shall need the following auxiliary lemma which we leave without proof
\begin{lemma}\label{uniform injectivity}
There exists a neighborhood of the identity $W\subset G_{S^*}$, depending only on the class $x$, such that for any $\om\in K_S$ and any $g\in W$, 
if $y_\om L_\om g \cap y_\om L_\om\ne\varnothing$ then $g\in L_\om$.
\end{lemma}
\begin{proof}[Proof of Lemma~\ref{the shrinking tube}]
The first restriction we impose on $\hat{\del}$ is that it will be small enough so that in the real component, the map $(s,u)\mapsto\exp_\infty(s)\cdot\exp_\infty(u)$ from
$B_{\hat{\del}}^{\lie(A_\infty)}\times B_{\hat{\del}}^{V_\infty}\to G_\infty$ is a homeomorphism onto its open image. Choose $B=\prod_{S}B_p$ to be any product compact open subgroup of $\tilde{B}$ and define $U^\del$ as in the statement of the Lemma.
 At this stage we observe that for any $\del<\hat{\del}$ the map
$L_\om\times U^\del\to G_{S^*}$ given by $(g,u)\mapsto g\exp_{S^*}(u)$ has an open image. To see this, note that the image is a product of open sets in each component: 
In the real component
the image equals $A_\infty\cdot\exp_\infty B_\del^{V_\infty}$ which is open by the choice of $\hat{\del}$, while for any finite place $p\in S$, the $p$'th component of the 
image is $(H_\om)_p\cdot B_p$ which is seen to be open in the following way: Because of the fact that $V_p=\lie(B_p)$ is a linear complement to $\lie((H_\om)_p)$, the 
product $(H_\om)_p\cdot B_p$ clearly contains an open neighborhood of the identity in $G_p$. It now follows from the fact that both $(H_\om)_p,B_p$ are groups, that 
their product is actually an open set. 

The above establishes in particular, that the set $\cT_\om^\del=y_\om L_\om\exp_{S^*}(U^\del)\subset X_S$ is open. 
It follows that in order to conclude that $\cT_\om^\del$ is indeed a tube around $y_\om L_\om$, we only need to argue the injectivity of the  map 
$(z,u)\mapsto z\exp_{S^*}(u)$ from $y_\om L_\om\times U^\del$ to $X_S$. We denote this map by $\psi_\om$.

The second condition which we impose on $\hat{\del}$ and on the choice of $B$ is that the product  
$\pa{\exp_\infty(B_{\hat{\del}}^{V_\infty}  )}^2\cdot B^2\subset W$,
where $W$ is as in Lemma~\ref{uniform injectivity}.
Assuming the injectivity of $\psi_\om$ fails, we obtain elements $u_\infty^{(i)}\in B_\del^{V_\infty},b_i\in B$, $i=1,2$ and a non-trivial intersection of the form 
$$y_\om L_\om \exp_\infty(u_\infty^{(1)})b_1\cap y_\om L_\om \exp_\infty(u_\infty^{(2)})b_2.$$
This shows that $y_\om L_\om\cap y_\om L_\om \exp_\infty(u_\infty^{(2)})\exp_{\infty}(-u_\infty^{(1)})b_2b_1^{-1}\ne\varnothing$. 
It now follows from our choice of $\hat{\del}$ and $B$ (by Lemma~\ref{uniform injectivity}) that 
$\exp_\infty(u_\infty^{(2)})\exp_{\infty}(-u_\infty^{(1)})\in A_\infty$ 
and that $b_2b_1^{-1}\in H_\om$.
As $B$ is a group which intersects $H_\om$ trivially (this is our assumption that the generalized branch $\cL_\om$ is nondegenerte), 
we conclude that $b_1=b_2$.
Furthermore, from the fact that $\lie{(A_\infty)}\oplus V_\infty=\gog_\infty$, it is straightforward to deduce that if $\hat{\del}$ is chosen small enough, then the inclusion
$\exp_\infty(u_\infty^{(2)})\exp_{\infty}(-u_\infty^{(1)})\in A_\infty$ implies that $u_\infty^{(1)}=u_\infty^{(2)}$. This establishes the injectivity of $\psi_\om$ as desired. 
\end{proof}
\begin{proof}[Proof of Lemma~\ref{again finite index}]
We first argue the validity of part~\eqref{a.f.i.1}. As pointed out in the proof of Lemma~\ref{the shrinking tube} above, if $\om$ is such that $\cL_\om$ is non-degenerate,
then the set $H_\om\cdot B\subset G_S$ is open (here $B$ is as in Lemma~\ref{the shrinking tube}). 
Moreover, as $B$ is compact, there exist a neighborhood of the identity in $G_S$, and in particular, a compact open subgroup $K^*=\prod_{S} K_p^*$, with 
the property that for any $k\in K^*$ we have $Bk\subset H_\om B$. We now claim that for any tube $\cT_\om^\del$ as in Lemma~\ref{the shrinking tube} we have
$\cT_\om^\del k=\cT_\om^\del$. To argue the inclusion $\subset$ we note the following
$$\cT_\om^\del k=y_\om L_\om\exp_\infty(B_\del^{V_\infty}) Bk\subset y_\om L_\om\exp_\infty(B_\del^{V_\infty})H_\om B=y_\om L_\om\exp_\infty(B_\del^{V_\infty}) B=\cT_\om^\del.$$
The opposite inclusion follows by switching $k$ with $k^{-1}$. 

If $x$ is non-spilt, it is not hard to see that the intersection $\cap_{\om\in K_S} (H_\om\cdot B)$ contains an open neighborhood around $e_S$. It then readily follows that this intersection contains an open neighborhood of $B$. We conclude similarly to the argument presented above that the group $K^*$ may be chosen to work
for all the $\om$'s simultaneously.  

We briefly argue part~\eqref{a.f.i.2} of the lemma. For each relevant $\om$, it is not hard to see that the volume of the tube $m_S(\cT_\om^{\del})$ satisfies
$c_1m_{V_\infty}(B_\del^{V_\infty})\le m_S(\cT_\om^\del)\le c_2 m_{V_\infty}(B_\del^{V_\infty})$, where the constants $c_1,c_2$ are determined by the 
volume of the orbit $y_\om L_\om$ and the position of the linear space $V$, from which the width is coming,  with respect to $\lie(L_\om)$. As the 2-dimensional volume  
$m_{V_\infty}(B_\del^{V_\infty})$ is proportional to $\del^2$, the claim regarding a single $\om$ follows. In the non-split case, as the Lie algebras $\lie(L_\om)$ are 
uniformly transverse to $V$, the constant $c_1$ above can be taken to be uniform for all $\om$ which finishes the proof.
\end{proof}

\section{Proof of Lemma~\ref{total growth}}\label{proof of total growth}
Let $S$ be a finite set of primes. For an element $\del\in G_S$ we denote 
$\Sig_\del\defi\overline{\idist{\del}}_{G_S}$ and we say that $\del$ is of \textit{compact type} if $\Sig_\del$ is a  compact group. 
\begin{definition}\label{12.9.def}
Let $\del\in G_S$ be an element of compact type. Let us denote for any admissible radius $\on{h}$  by $k_{\onh}(\del)$ the minimal positive integer $k$ for which $\del^k$ belongs to the compact open subgroup $a_f(\onh)K_Sa_f(\onh^{-1})$. Equivalently, 
$k_{\onh}(\del)=\br{\Sig_\del:\Sig_\del\cap a_f(\onh)K_Sa_f(\onh)^{-1}}$.
\end{definition}
\begin{lemma}\label{growth again}
Let $\del\in G_S$ be an element of compact type such that no power of $\del$ has an upper triangular component. 
Then, the ratio $e_{\onh}(\del)\defi \frac{k_{\onh}(\del)}{\onh}$ 
attains only finitely many values and furthermore, if $\onh_n$ is a divisibility sequence (i.e.\ $\onh_{n}|\onh_{n+1}$), then the sequence $e_{\onh_n}(\del)$ stabilizes.
\end{lemma}
\begin{proof}
The strategy is to reduce to the case where $S$ consists of a single prime $p$ and $\del\in K_p$ satisfies a certain congruence 
assumption.  In this case the description of $k_{\onh}(\del)$ becomes explicit and simple.

\noindent\tbf{Step} 1 - \tbf{Reduction to one prime}. 
Let us denote for an admissible radius $\onh$ by $n_p(\onh)$ the integers satisfying $\onh=\prod_{p\in S}p^{n_p(\onh)}$. 
We have the equality
$$a_f(\onh)K_Sa_f(\onh^{-1})=\prod_{p\in S} a_p(p^{n_p(\onh)}) K_p a_p(p^{-n_p(\onh)}),$$ 
and so, if we denote $\del=(\del_p)_{p\in S}$, then  $k_{\onh}(\del)=\on{lcm}\set{k_{p^{n_p(\onh)}}(\del_p):p\in S}$.
From here 
it follows that the statement of the Lemma for a general finite set of primes $S$ 
follows from the corresponding statement for a single prime. 
Therefore, henceforth we assume that $S$ 
consists of a single prime $p$ and our objective is to show that the sequence $k_{p^n}(\del)/p^n$ stabilizes. 
For simplicity we denote $K_{p,n}\defi a_f(p^n)K_pa_f(p^{-n})$.

\noindent\tbf{Step} 2 - \tbf{Replacing $\del$ by a power}.
We wish to prove  that for any $\ell>0$ the statement of the Lemma for $\del$ is equivalent
to  the statement of the Lemma for $\del^\ell$.
We first prove that $\cap_{n}(\Sig_\del\cap K_{p,n})=\set{e}$.
To see this, note that  as 
\eq{Kpn}{
K_{p,n}=\set{\smallmat{a&p^{-n}b\\p^nc&d}:\smallmat{a&b\\c&d}\in K_p},
}
we see that 
$\cap_n K_{p,n}\subset\set{\smallmat{*&*\\ 0&*}}$, and so if the intersection $\cap_{n}(\Sig_\del\cap K_{p,n})$ 
was non-trivial, then it would
imply that $\Sig_\del$ contains a non-trivial upper triangular element, which in turn would imply that its (one-dimensional) 
Lie algebra is upper triangular. This is equivalent to saying that a power of $\del$ is upper triangular, which contradicts our assumption
on $\del$.

We conclude that for any $\ell>0$, as $\Sig_{\del^\ell}<\Sig_\del$ is an open subgroup, for $n$ large enough
we have that $\Sig_\del\cap K_{p,n}=\Sig_{\del^\ell}\cap K_{p,n}$. This implies that for $n$ large enough
\eq{redtopower}{
k_{p^n}(\del)=\br{\Sig_\del:\Sig_{\del^\ell}\cap K_{p,n}}=\br{\Sig_\del:\Sig_{\del^\ell}}\cdot\br{\Sig_{\del^\ell}:\Sig_{\del^\ell}\cap K_{p,n}}=\br{\Sig_\del:\Sig_{\del^\ell}} k_{p^n}(\del^\ell),
}
and so, in particular, the sequence $k_{p^n}(\del)/p^n$ stabilizes if and only if the sequence $k_{p^n}(\del^\ell)/p^n$ does, 
as desired. 

\noindent\tbf{Step} 3 - \tbf{Concluding the proof}. By Step 2 we may assume (by replacing $\del$ by a suitable power of $\del$ if 
necessary), that $\del$ belongs the (open) subgroup of $K_p$ consisting of elements congruent to the 
identity modulo $p^2$ (or said differently, to the kernel of the natural homomorphism from 
$K_p$ onto $\PGL_2(\bZ/p^2\bZ)$). 
Denote $B_n\defi K_p\cap K_{p,n}$.
A direct calculation shows 
$B_n=\set{\smallmat{a&b\\ c&d}\in K_p: \frac{c}{p^n}\in \bZ_p}.$
Under the above assumption we have that $k_{p^n}(\del)$ is the order of (the image of) $\del$ in the finite cyclic 
quotient $\Sig_\del/(\Sig_\del\cap B_n)$. Furthermore, as $B_{n+1}<B_n$, the divisibility relation
$k_{p^n}(\del)|k_{p^{n+1}}(\del)$ holds.
Let $n_0\defi\max\set{n>0:\del\in B_n}$. Our assumption that $\del$ is not upper triangular implies that $n_0$ is well defined.
The proof of the Lemma will be concluded once we establish the following

\noindent\textbf{Claim}: For any $n\ge n_0$ we have $k_{p^n}(\del)=p^{n-n_0}$ and moreover, $\del^{p^{n-n_0}}\in B_n\smallsetminus B_{n+1}$. 

We prove this claim by induction on $n$. For $n=n_0$ the validity of the claim follows from the choice of $n_0$. Let us assume it holds 
for $n$. 
As mentioned above, the divisibility relation
$k_{p^n}(\del)|k_{p^{n+1}}(\del)$ holds. Moreover, from our inductive hypothesis saying that 
$\del^{k_{p^n}(\del)}\notin B_{n+1}$ we know that this divisibility relation is strict. It follows that 
$k_{p^{n+1}}(\del)=j_0p^{n-n_0}$ where $j_0$ is the minimal positive integer $j$ so that $\del^{jp^{n-n_0}}\in B_{n+1}$, or said differently, such that the bottom left coordinate of 
$\del^{jp^{n-n_0}}$ is divisible by $p^{n+1}$ in $\bZ_p$. We will be finished once we show two things:
\begin{enumerate}
\item\label{12ntp1} First, that $j_0=p$ and so $k_{p^{n+1}}(\del)=p^{n+1-n_0}$, 
\item\label{12ntp2} and second, that the bottom left coordinate
of $\del^{p^{n+1-n_0}}$ is not divisible by $p^{n+2}$ and so $\del^{k_{p^{n+1}}(\del)}\in B_{n+1}\smallsetminus B_{n+2}$ which completes the inductive step. 
\end{enumerate}
Consider the sequence $\del^{jp^{n-n_0}}\defi\smallmat{ a_j& b_j\\ c_j& d_j}$, $j=1,2,\dots$ and note the recursive relation 
\begin{equation}\label{12.9.4}
c_{j+1}=c_1 a_j+ c_j d_1.
\end{equation}
We expand $c_1$ to a power series in $\bZ_p$ and use the inductive assumption that $\del^{p^{n-n_0}}\in B_n\smallsetminus B_{n+1}$ and write 
\begin{equation}\label{12.9.5}
c_1= m_1p^{n}+m_2p^{n+1}+up^{n+2},
\end{equation}
where $m_1\in\set{1,2,\dots,p-1}$, $m_2\in \set{0,1,\dots,p-1}$, $u\in\bZ_p$. We claim that for any $1\le j$ we have 
\begin{equation}\label{12.9.6}
c_j=j m_1 p^n+jm_2p^{n+1}+u_jp^{n+2}\textrm{ where }u_j\in\bZ_p.
\end{equation}
The validity of~\eqref{12ntp1},\eqref{12ntp2} follows at once from~\eqref{12.9.6} and the fact that $m_1\in\set{1,2,\dots p-1}$. We prove the validity of~\eqref{12.9.6} by induction on $j$. For $j=1$, this 
is exactly~\eqref{12.9.5}. Now assume it holds for $j$ and write (using the congruence assumption  on $\del$) 
\begin{align*}
&a_j=1+p^2A,\: d_j=1+p^2 D, \; A,D\in\bZ_p.
\end{align*}
Plugging this and~\eqref{12.9.5},\eqref{12.9.6} into the recursive relation~\eqref{12.9.4} we see that indeed $$c_{j+1}=(j+1)m_1p^n+(j+1)m_2 p^{n+1}+p^{n+2}(\dots)$$ as desired. This completes the proof of the Claim and by that concludes the proof of the Lemma.
\end{proof}
\begin{remark}\label{11.9.r.1}
It will be useful later on to note the following: A careful look at the argument giving Lemma~\ref{growth again} shows that for a fixed $\del\in G_S$ of compact type, we have that there exists a positive constant $c$ such that $c\le  e_{\onh}(\del)$ for any admissible radius $\onh$, where $c$ depends only on two things: 
\begin{enumerate}
\item The power $k_0$ which we need to raise $\del$ to so that each component $(\del)_p^{k_0}$ will be in $K_p$ and congruent to the identity mod $p^2$.
\item The maximal admissible radius $\onh=\prod_{p\in S} p^{n_p}$ for which for any $p\in S$, $(\del)_p^{k_0}\in B_{n_p}$ (this $\onh$ measures how close $\del^{k_0}$ is to being upper triangular).
\end{enumerate}
\end{remark}
Before turning to the proof of Lemma~\ref{total growth} we make yet another remark which will be used in the course of its proof.
\begin{remark}\label{11.9.r.2}
Given a class $x\in X_\infty$ with a periodic $A_\infty$-orbit and a representative $\Lam_x\in x$, then the matrix $a_\infty(t_x)=\diag{e^{\frac{t_x}{2}},e^{-\frac{t_x}{2}}}$ stabilizes the lattice $\Lam_x$ (that is, as a subset of $\bR^2$, $\Lam_x=\Lam_x a_\infty(t_x)$). As $\Lam_x$ is a lattice, it follows that $a_\infty(t_x)$ is conjugate 
to an integer matrix and so its eigenvalues $e^{\pm\frac{t_x}{2}}$, are algebraic integers of degree $2$. The quadratic extension $\bF_x$ from Definition~\ref{def per} is the one generated by them. 
As these eigenvalues are positive Galois conjugates whose product is equal to $1$, we conclude furthermore that they belong
to the group of  totally positive units in the ring of integers of $\bF_x$ (i.e.\ units all of whose embeddings into the reals are positive). 
As such,
by Dirichlet's unit theorem, they are integer powers of the fundamental unit of this field. In fact, we shall slightly abuse the classical
terminology and use the term \textit{fundamental unit} to refer
to the unit in the ring of integers which is of absolute value $>1$ and which generates the group of totally positive units. If the fundamental unit  is $\eps=e^{\frac{t_0}{2}}$, then the reader will easily verify that the image $\Lam_x a_\infty(t_0)$ is contained in the $\bQ$-span of $\Lam_x$. This shows that if we write $x=\Ga_\infty g_x$, then there is a rational matrix $\del_x$ which solves $\del_x g_x=g_xa_\infty(t_0)$ and in fact, 
$t_x=kt_0$ where $k$ is the minimal positive integer such that $\del_x^k$ is an integer matrix.
\end{remark}

\begin{proof}[Proof of Lemma~\ref{total growth}]
\eqref{t.g.1}.
A short counting argument shows that the cardinality of the sphere $\cS_{\on{h}}(x)$ is proportional to $\on{h}$ (were the proportionality constant depends on $S$). 
For each $x'$ on the sphere, let $s_{x'}$ be the minimal positive number
such that $x'a_\infty(s)$ returns to the sphere. The total length is then $\textbf{t}_x(\on{h})=\sum_{x'\in\cS_{\on{h}}(x)} s_{x'}$. 
We will  show below that for any $x'\in\cS_{\on{h}}(x)$ $s_{x'}\le t_x$. This will establish the inequality $\textbf{t}_{\on{h}}(x)\ll_{x,S}\on{h}$ which is half of of the statement in part~\eqref{t.g.1} of the 
Lemma. The other half, namely the inequality $\on{h}\ll_{x,S}\textbf{t}_{x}(\on{h})$, actually follows from part~\eqref{t.g.2} of the Lemma.  

Let $x'\in\cS_{\onh}(x)$ be given. By Lemma~\ref{get the sphere} we see that there exists $y\in\pi^{-1}(x)$
 such that $x'=\pi(ya_f(\on{h}))$. As $\pi$ intertwines the $A_\infty$-actions on $X_S,X_\infty$ we see that $x=xa_\infty(t_x)=\pi(ya_\infty(t_x))$ and so if we let 
$\tilde{y}=ya_\infty(t_x)$ then $\tilde{y}\in\pi^{-1}(x)$ and again by Lemma~\ref{get the sphere} we have that $x''=\pi(\tilde{y}a_f(\on{h}))\in\cS_{\on{h}}(x)$. The following calculation then shows that  indeed $s_{x'}\le t_x$ as was claimed:
$$x'a_\infty(t_x)=\pi(ya_f(\on{h})a_\infty(t_x))=\pi(ya_\infty(t_x)a_f(\on{h}))=\pi(\tilde{y}a_f(\on{h}))=x''\in\cS_{\on{h}}(x).$$

\eqref{t.g.2}. Let $\cL=\cL_{g_x,\om}$ be a non-degenerate generalized branch of $\cG_S(x)$ (here $x=\Ga_\infty g_x$ and $\om\in K_S$). 
Let $t_0>0$ be
 such that $e^{\frac{t_0}{2}}$ is the fundamental unit of $\bF_x$ as in Remark~\ref{11.9.r.2}. In the notation of the same remark, 
 let $\del_x$ be the rational matrix satisfying $\del_x g_x=g_xa_\infty(t_0)$. We replace the set of places $S$ by a bigger set
 if necessary $\widetilde{S}$, so that $\del_x\in \Ga_{\widetilde{S}}$. We then consider the bigger graph $\cG_{\widetilde{S}}(x)$ which contains the original graph and we further 
 consider its following generalized branch: Write $\widetilde{S}=S\cup T$ and define $\widetilde{\om}\in K_{\widetilde{S}}$ to be identical to $\om$ in the 
 components  corresponding to the primes in $S$ and equal the identity in the components corresponding to primes in $T$. We then define $\widetilde{\cL}$ to be the 
 generalized branch $\cL_{g_x,\widetilde{\om}}$ of $\cG_{\widetilde{S}}(x)$. Note that because of the way we defined $\widetilde{\om}$, the generalized branch $\widetilde{\cL}$ is non-degenerate as well. 
 
Denote as before by $x_{\on{h}}$ the class in $\widetilde{\cL}\cap\cS_{\on{h}}(x)$. We are interested in analyzing the length $t_{x_{\on{h}}}$ of the orbit 
$x_{\on{h}}A_\infty$. By Remark~\ref{11.9.r.2}, there exists a positive integer $\widehat{k}_{\on{h}}$  satisfying 
\begin{equation}\label{12.9.10}
t_{x_{\on{h}}}=\widehat{k}_{\on{h}} t_0.
\end{equation}
In fact, for later purposes, note that in our discussion
 $x$ and the representative $g_x$ are fixed but we will play with the branch later on, i.e.\ with the choice of $\widetilde{\om}$ (which in our setting is defined by $\om$), and so we should actually record the dependency in $\widetilde{\om}$ in our notation and denote $\widehat{k}_{\on{h}}(\widetilde{\om})$. The function $k_{\onh}(\cdot)$ from Definition~\ref{12.9.def} and $\widehat{k}_{\onh}(\cdot)$ are closely related as will be seen below.

 The number $\widehat{k}_{\on{h}}(\widetilde{\om})$ is by definition the minimal 
 positive integer such that $x_{\on{h}} a_\infty(kt_0)=x_{\on{h}}$ or, if we prefer working in the extension $X_{\widetilde{S}}$, it is  the minimal positive integer so that 
 $\Ga_{\widetilde{S}}(g_x,\widetilde{\om})a_f(\on{h})a_\infty(kt_0)$ returns to the fiber $\pi^{-1}(x_{\on{h}})$. 
 Because of the identity $\del_xg_x=g_xa_\infty(t_0)$ and the fact that $\del_x\in \Ga_{\widetilde{S}}$ we see that 
  $\Ga_{\widetilde{S}}(g_x,\widetilde{\om})a_f(\on{h})a_\infty(kt_0)=\Ga_{\widetilde{S}}(g_x,\del_x^{-k}\widetilde{\om} a_f(\on{h}))$, and so this point lies in the same fiber as $\Ga_{\widetilde{S}}(g_x,\widetilde{\om} a_f(\on{h}))$ (i.e.\ 
  above $x_{\on{h}}$) if and only if the quotient $a_f(\onh^{-1})\widetilde{\om}^{-1}\del_x^{k}\widetilde{\om} a_f(\on{h})$ belongs to $K_{\widetilde{S}}$ 
  (see Remark~\ref{12.9.r.1}). That is, $\widehat{k}_{\on{h}}(\widetilde{\om})$ is the minimal positive integer $k$ for which the $(\widetilde{\om}^{-1}\del_x\widetilde{\om})^k\in a_f(\onh)K_{\widetilde{S}}a_f(\onh^{-1})$.  This establishes the equality
  $$k_{\onh}(\del_x^{\widetilde{\om}})=\widehat{k}_{\onh}(\widetilde{\om}).$$
  The validity of part~\eqref{t.g.2} of the Lemma now follows immediately from Lemma~\ref{growth again} and~\eqref{12.9.10} which together imply $c_{\cL_{g_x,\om}}(\onh)= t_0e_{\onh}(\del_x^{\widetilde{\om}}). $
  
\eqref{t.g.3}. Assume first that $x$ is non-split with respect to $S$. As noted in Remark~\ref{11.9.r.1} the lower bound for the function 
$\onh\mapsto e_{\onh}(\del_x^{\widetilde{\om}})$, which gives us the lower bounds for the functions $c_{\cL_{g_x,\om}}(\onh)$, depends only on two things:
\begin{enumerate}
\item The smallest power $k_0$ for which $\del_x$ belongs to the subgroup
  of $K_{\widetilde{S}}$ consisting of elements congruent to the identity modulo $p^2$ in each component (note that we may ignore the conjugation by $\widetilde{\om}$ as this is a normal subgroup of $K_{\widetilde{S}}$). 
  \item The $p$-adic norms $\av{c_p}_p$, where $c_p$ is the left bottom coordinate of the $p$-component of $(\widetilde{\om}^{-1}\del_x\widetilde{\om})^{k_0}$ where $p\in \widetilde{S}$. 
  \end{enumerate}
  It is clear that $k_0$ depends only on $x$ and the original set of primes $S$ and does not vary with $\om$ (i.e.\ with the generalized branch). Also, for primes $p\in S$, the $p$-adic norm $\av{c_p}_p$ is bounded from below as $\om$ ranges over $K_S$ because $x$ is non split. Finally, for
  the primes $p\in \widetilde{S}\smallsetminus S$, as  the $p$'th component  of $\widetilde{\om}$ equals the identity, the $p$'th component of $(\widetilde{\om}^{-1}\del_x\widetilde{\om})^{k_0}$ is independent of $\om$. We conclude that $$\inf\set{c_{\cL_{g_x,\om}}(\onh):\om\in K_S,\onh\textrm{ is an admissible radius}}>0$$ as desired. We leave it as an exercise to the reader to show that in the split case this infimum equals zero.

\end{proof}

\section{Applications to continued fractions}\label{applications}
In this section we present the necessary terminology and results that will allow us to state and prove our main theorems regarding continued fractions and deduce Theorems~\ref{mtcf11},\ref{period cor new}.
\subsection{The $A_\infty$-orbit attached to $\al\in\qi$} 
\begin{definition}\label{def g alpha}
For $\al\in\qi$, let $\al'$ denote its Galois conjugate and let 
\begin{equation}\label{g alpha}
g_\al\defi\mat{\al&\al'\\
1&1
}\;\textrm{if $\al>\al'$, and }
g_\al\defi\pa{\begin{matrix}
\al&-\al'\\
1&-1\end{matrix}}\;\textrm{otherwise.}
\end{equation}
Furthermore, let $x_\al\defi \Ga_\infty g_\al\in X_\infty$.
\end{definition}
\begin{lemma}\label{compact A-orbit}
Let $\al\in\qi$. Then, the orbit $x_\al A_\infty \subset X_\infty$ 
is periodic.
\end{lemma}
\begin{proof}
Consider the $\bZ$-module $\Lam_\al=\on{span}_\bZ\set{1,\al}$ in the field $\bQ(\al)$. There exists a unit $\om$ in the ring of integers which stabilizes $\Lam_\al$ and furthermore by replacing $\om$ by
$\om^2$ if necessary we may assume that both $\om$ and its Galois conjugate $\om'$ are positive. Note that the diagonal matrix $\diag{\om,\om'}$ is an element of $A_\infty$. Let $\ga=\smallmat{n&m\\k&\ell}\in\GL_2(\bZ)$ be the matrix describing the
passage from the basis $\set{1,\al}$ to the basis $\set{\om,\om\al}$ of $\Lam_\al$. That is 
\begin{equation}\label{looop}
\pa{\begin{array}{ll}n&m\\k&\ell\end{array}}\pa{\begin{array}{ll}\al\\1\end{array}}=\pa{\begin{array}{ll}\om\al\\\om\end{array}}.
\end{equation} 
The reader will easily verify now that~\eqref{looop} implies that $\ga g_\al = g_\al \diag{\om,\om'}$ or in other words that in the space $X_\infty$ the orbit $x_\al A_\infty$ is periodic as desired.
\end{proof}
\begin{definition}\label{def per qi}
Given $\al\in\qi$, in the spirit of Definition~\ref{def per}, we denote $t_\al=t_{x_\al}$, $\mu_\al=\mu_{x_\al}$, and $\ga_\al=\ga_{x_\al}$, where $\ga_{x_\al}$ is defined 
using the representative $g_\al$ of $x_\al$.
\end{definition}

Fix a finite set of primes $S$. For $\al\in\qi$  
 consider the $S$-Hecke graph $\cG_S(x_\al)$ and recall that by Corollary~\ref{sphere again new}, for $\ga\in\Ga_S$ with $\height{\ga}=\onh$ we have that the class $\Ga_\infty\ga g_\al$ lies
 on the sphere $\cS_{\onh}(x_\al)$.
 \begin{definition}\label{defonthebranch}
 Let $\om\in K_S$ 
 \begin{enumerate}
 \item We say that $\ga\in\Ga_S$ \textit{lies on the generalized branch} $\cL_{g_\al,\om}$ if $\Ga_\infty\ga g_\al\in\cL_{g_\al,\om}$ and denote this by
 $\ga\in\br{\om}_{\on{br}}$. As will be explained shortly in Remrak~\ref{r.1629}, the question of whether or not $\ga\in\br{\om}_{\on{br}}$ is indeed independent of $\al$ as suggested by the notation.
  \item Similarly to the notation introduced in Remark~\eqref{j.27}\eqref{j.27.1}, we denote by 
$x_{\al,\om,\onh}$ the class in $\cL_{g_\al,\om}\cap\cS_{\onh}$. 
\end{enumerate}
 \end{definition}
 \begin{remark}\label{r.1629}
  With the above notation, for $\ga\in\Ga_S$ with $\height{\ga}=\onh$ we have that $\ga\in\br{\om}_{\on{br}}$ if and only if $x_{\al,\om,\onh}=\Ga_\infty\ga g_\al$.
 As mentioned in the proof of Corollary~\ref{sphere again new}, an element $\ga\in\Ga_S$ can be written as $\ga=\ga_1\diag{\onh,1}\ga_2$ with $\ga_i\in\Ga_\infty$ and $\onh=\height{\ga}$. It follows that for $\om\in K_S$ we have that $\ga\in\br{\om}_{\on{br}}$ if and only if $\pi(\Ga_S(g_\al,\ga^{-1})=\pi(\Ga_S(g_\al,\om)a_f(\onh))$. The 
 latter happens, by Remark~\ref{12.9.r.1}, exactly when $\ga\om a_f(\onh)$ belongs to $K_S$, or equivalently, when the lower left coordinate of $\ga_2$ is divisible by $\onh$ in each $\bZ_p$ for $p\in S$. 
 \end{remark}
 The following Lemma relates the orbit $\set{\ga\al:\ga\in\Ga_S}\subset\bR$ 
and the periodic $A_\infty$-orbits through points of $\cG_S(x_\al)$.
 \begin{lemma}\label{lemconnection}
 Let $S$ be a finite set of primes, $\al\in\qi$, $\om\in K_S$, and $\ga\in\br{\om}_{\on{br}}$, 
 with $\height{\ga}=\onh$. Denote by $\theta\in \Ga_\infty$ the element $\theta=\diag{1,-1}$. Then, one of the following two relations holds:
 Either 
 $\mu_{x_{\al,\om,\onh}}=\mu_{\ga\al}$, or
  $\mu_{x_{\al,\om,\onh}}=\theta_*\mu_{\ga\al}$.
 \end{lemma}
 \begin{proof}
Write $\ga=\smallmat{a&b\\c&d}$ and recall that the matrix $g_\al$ in Definition~\ref{def g alpha} 
has one of the forms $g_\al=\smallmat{\al&\al'\\1&1}$ or $g_\al=\smallmat{\al&\al'\\1&1}\theta$. 
The equation
 \begin{equation}\label{eqmat}
 \ga \mat{\al&\al'\\ 1&1}=\mat{\ga\al&(\ga\al)'\\ 1&1}\smallmat{\frac{1}{c\al+d}&0\\ 0&\frac{1}{c\al'+d}},
 \end{equation} 
 together with the fact that $g_{\ga\al}$ has either the form $\smallmat{\ga\al&(\ga\al)'\\ 1&1}$ or the form 
 $\smallmat{\ga\al&(\ga\al)'\\ 1&1}\theta$, imply that
 the two points $x_{\al,\om,\onh}=\Ga_\infty \ga g_\al$, $x_{\ga\al}=\Ga_\infty g_{\ga\al}$ are on the same orbit under the group generated by $A_\infty$ and $\theta$
 in $G_\infty$. This group contains $A_\infty$ as a subgroup of index $2$ and so either the two points are on the same $A_\infty$-orbit or otherwise, their $A_\infty$-orbits
 are related by the action of $\theta$. The translation of the latter statement to the $A_\infty$-invariant probability measures that are supported on these orbits is exactly 
 the statement sought.
 \end{proof}

The following Theorem relates the measures $\mu_\al$ from Definition~\ref{def per qi} to the measures $\nu_\al$ from Definition~\ref{defnuiota} and will allow us to translate
equidistribution results for geodesic loops to statements about periods of c.f.e while controlling error terms.
We will use the following terminology
\begin{definition}\label{parity}
For $\al\in\qi$, we denote 
\begin{equation*}
j_\al\defi\Big\{\begin{array}{ll} 1&  \textrm{\small{ if the size} }\av{P_\al}\textrm{\small{ is even,}}\\2&\textrm{\small{ if the size} }\av{P_\al}\textrm{\small{ is odd.}}
\end{array} 
\end{equation*}
\end{definition}
\begin{theorem}\label{toest}
Let  $\al\in \qi$. There exists an absolute constant $T_0>1$ so that if we assume that for some $T>T_0$ the estimate   
$\av{\int f d\mu_\al-\int fd m_\infty}\le \max\set{\norm{f}_2,\ka}T^{-1}$ holds for any $f\in\on{Lip}_\ka(X_\infty)\cap L^2(X_\infty,m_\infty)$,
then the following two statements hold
\begin{enumerate}
\item\label{toest11} For any $f\in\on{Lip}_\ka([0,1])$, and any $\eps>0$ 
$$\av{\int_0^1 fd\nu_\al-\int_0^1fd\nugauss}\ll_{ \eps} \max\set{\norm{f}_\infty,\ka}T^{-\frac{1}{3}+\eps}.$$
\item\label{toest12} There exists an absolute constant $c_0$ such that $\av{\frac{\av{P_\al}}{t_\al}-\frac{1}{j_\al c_0}}\ll_\eps T^{-\frac{1}{3}+\eps}.$
\end{enumerate}
\end{theorem}
\noindent Below we will use Theorem~\ref{toest}  while postponing its proof to~\S\ref{potoest}. 

 \subsection{Main Theorems regarding continued fractions}
 The following is the analogue of Theorem~\ref{general case} in the language of continued fractions. We prove it and then deduce Theorem~\ref{mtcf11}.
\begin{theorem}\label{mtcf}
Let $S$ be a finite set of primes and $\al\in\qi$. 
\begin{enumerate}
\item\label{p.t.1} Let $\om\in K_S$ be such that  the generalized branch $\cL_{g_\al,\om}$ is non-degenerate. Then for any $\ga\in\br{\om}_{\on{br}}$, any
$\eps>0$, and any  $f\in\on{Lip}_\ka([0,1])$ the following estimate holds
\begin{equation}\label{mtcfeq}
\av{\int_0^1fd\nu_{\on{Gauss}}-\int_0^1fd\nu_{\ga\al}}\ll_{\cL_{g_\al,\om},S,\eps}\max\set{\norm{f}_\infty,\ka}\on{ht}(\ga)^{-\frac{\del_0}{6}+\eps},
\end{equation}
\item\label{p.t.2} If all the primes in $S$ do not split in the quadratic extension $\bQ(\al)$, then the implicit constant in~\eqref{mtcfeq} may be taken to be independent of the generalized branch.
\item\label{p.t.3} On the other hand, if there exists $\om\in K_S$ such that the generalized branch $\cL_{g_\al,\om}$ is  degenerate, 
then there exists a sequence $\ga_n\in\br{\om}_{\on{br}}$ such that  
$\height{\ga_n}\to\infty$ and the cardinality of the periods $P_{\ga_n\al}$
is bounded. In particular, $\nu_{\ga_n\al}\nrightarrow\nugauss$.
\item\label{p.t.4} Rational generalized branches are always non-degenerate.
\item\label{p.t.5} Nevertheless, in case $x_\al$ is split, then one cannot take the implicit constant in~\eqref{mtcfeq} to be uniform along the rational generalized 
branches.
\end{enumerate}
\end{theorem}
\begin{proof}
$\eqref{p.t.1}$. Let $\ga\in\br{\om}_{\on{br}}$ and $\eps >0$ be given and denote $\onh=\height{\ga}$.
 By the corresponding part of Theorem~\ref{general case} we know that for the class $x_{\al,\om,\onh}\in\cL_{g_\al,\om}\cap\cS_{\onh}(x_\al)$,  the measure 
$\mu_{x_{\al,\om,\onh}}$ satisfies the estimate 
\begin{equation}\label{yaeq}
\av{\int fd\mu_{\al,\om,\onh}-\int fdm_\infty}\ll_{\cL_{g_\al,\om},S,\eps}\max\set{\ka,\norm{f}_2}\onh^{-\frac{\del_0}{2}+\eps},
\end{equation}
for any $f\in\on{Lip}_\ka(X_\infty)\cap L^2(X_\infty,m_\infty)$.
By Lemma~\ref{lemconnection} we know that either $\mu_{\ga\al}=\mu_{x_{\al,\om,\onh}}$ or $\theta_*\mu_{\ga\al}=\mu_{x_{\al,\om,\onh}}$. Assume that the first possibility holds. We now apply Theorem~\ref{toest}
with $T^{-1}=C\onh^{-\frac{\del_0}{2}+\eps}$ where $C$ is the implicit constant in~\eqref{yaeq} and obtain the desired~\eqref{mtcfeq}.
One 
remark is in order here: For finitely many heights $\onh$ it might happen that this choice of $T$ is not valid as $T$ needs to exceed the absolute constant $T_0$
from Theorem~\ref{toest}. We overcome this problem by choosing the implicit constant in~\eqref{mtcfeq} to be big enough so that this inequality will hold for these finitely many cases as well. 

Assume now that when we apply Lemma~\ref{lemconnection} we obtain that $\theta_*\mu_{\ga\al}=\mu_{x_{\al,\om,\onh}}$. 
As $\theta$ acts as an isometry of $X_\infty$ we have that~\eqref{yaeq} implies the same estimate for $\mu_{\ga \al}$ replacing $\mu_{x_{\al,\om,\onh}}$ and the argument concludes as before. 

The argument giving part~\eqref{p.t.2} of the Theorem is identical to the one giving part~\eqref{p.t.1} of the Theorem but uses as an input the corresponding part of Theorem~\ref{general case}. 

For part~\eqref{p.t.3} of the Theorem follows from Theorem~\ref{general case}\eqref{g.c.3} because of the following general 
fact\footnote{This fact will become clear in
~\S\ref{potoest}, in fact $\av{P_\be}\le \frac{t_\be}{\eps_0}$, where $\eps_0$ is as in Lemma~\ref{the set B1}.}: 
For $\be\in\qi$ $\av{P_\be}\ll t_\be$.

Part~\eqref{p.t.4} of the Theorem is included in Theorem~\ref{general case} and finally, part~\eqref{p.t.5} of the Theorem follows from part~\eqref{p.t.3}
of the Theorem in the same way that the corresponding implication of Theorem~\ref{general case} was proved in the beginning of~\S\ref{soft version}.  
\end{proof}
\begin{proof}[Proof of Theorem~\ref{mtcf11}]
\eqref{part1mt}.
Note that for $q\in\cO_S$ if we define $\ga_q=\diag{q,1}$ then $\ga_q\al=q\al$. Define 
$\Om\subset K_{S}$ as follows
 \eq{Omegaa}{
 \Om\defi\set{\om=(\om_p)_{p\in S}\in K_S: \om_p=\smallmat{1&0\\0&1}\textrm{ or }\om_p=\smallmat{0&1\\1&0}}.
 }
Then, we leave it as an exercise to the reader to verify (using Remark~\ref{r.1629}), that in the notation of Definition~\ref{defonthebranch}, for any $q\in\cO_S$, $\ga_q\in\br{\om}_{\on{br}}$
for some $\om\in\Om$. As by Theorem~\ref{mtcf}\eqref{p.t.4} the finitely many rational generalized branches $\cL_{g_\al,\om}$, $\om\in\Om$ are all non-degenerate
we conclude from Theorem~\ref{mtcf}\eqref{p.t.1} that the estimate~\eqref{eqmtcf1} indeed holds.

\eqref{part2mt}.
Given $\ga\in\Ga_S$ choose $\om\in K_S$ such that $\ga\in\br{\om}_{\on{br}}$. The estimate~\eqref{eqmtcf2} holds with implicit constant depending
 only on $\al,S,\eps$ but not on the the generalized branch.  
If we replace $\ga$ by $\ga\ga_0$ for some choice of $\ga_0\in\Ga_\infty$ we change the generalized branch but this does not  effect the right hand side of~\eqref{mtcfeq} by Theorem~\ref{mtcf}\eqref{p.t.2}.  We conclude that the implicit constant does
not depend on $\al$ but only on the orbit $\iota(\al)$. This gives us the desired estimate~\eqref{eqmtcf2}.

\eqref{part3mt}. Let $\del_n\in\Ga_S$ be a sequence such that $\nu_{\del_n\al}$ does not converge to $\nugauss$, 
as in part~\eqref{p.t.3} of Theorem~\ref{mtcf}. Write $\del_n=\ga'_n\diag{q_n,1}\ga_n$, where 
$\ga_n,\ga_n'\in \Ga_\infty$. By ~\eqref{mdch}, $\nu_{\del_n\al}=\nu_{q_n\ga_n\al}$ and so the sequences $q_n,\ga_n$ satisfy the statement. 
\end{proof}
The following Theorem is the most general statement we could extract from our analysis regarding the growth of the period. 
We prove it and then deduce Theorem~\ref{period cor new}.
\begin{theorem}\label{gpth}
Let $S$ be a finite set of primes. There exists  a positive function $c(\al,\om,\onh)$ on the set 
$\qi\times K_S\times (\cO_S^\times\cap \bN)$ satisfying the following: For any $\al\in\qi$
\begin{enumerate}
\item\label{gpth2} If the generalized brach $\cL_{g_\al,\om}$ is non-degenerate then 
\begin{enumerate}
\item\label{gpth2c} 
For any $\eps>0$, for any $\ga\in\br{\om}_{\on{br}}$ with $\height{\ga}=\onh$ we have
\begin{equation}\label{eqgpth}
\av{P_{\ga\al}}=c(\al,\om,\onh)\onh +O_{\al,\om,S,\eps}(1)\onh^{1-\frac{\del_0}{6}+\eps}.
\end{equation}
Moreover, if all the primes in $S$ do not split in $\bQ(\al)$ then the the function $O_{\al,\om,S,\eps}(1)$ in~\eqref{eqgpth} is in fact
$O_{\iota(\al),S,\eps}(1)$.
\item\label{gpth2a} The function $c$ attains only finitely many values along the branch corresponding to $\om$; that is, 
$\av{\set{c(\al,\om,\onh):\onh\in\cO_S^\times\cap\bN}}<\infty$.
\item\label{gpth2b} If $\onh_n\in\cO_S^\times\cap\bN$  satisfies $\onh_n|\onh_{n+1}$, then $c(\al,\om,\onh_n)$ stabilizes.
\end{enumerate}
\item\label{gpth1}  $\sup\set{ c(\al,\om,\onh):\om\in K_S, \onh\in\cO^\times_S\cap\bN}\ll_{\iota(\al),S} 1$.
\item\label{gpth3} All the primes in $S$ do not split in the quadratic extension $\bQ(\al)$ if and only if $\inf\set{ c(\al,\om,\onh):\om\in K_S, \onh\in\cO_S^\times\cap\bN}>0$.
\end{enumerate}
\end{theorem}
\begin{proof}[Proof of Theorem~\ref{gpth}]
Fix $\al\in\qi$ and let $x_\al=\Ga_\infty g_\al$. Define $c(\al,\om,\onh)\defi \frac{1}{j_\al c_0}c_{\cL_{g_\al,\om}}(\onh)$, where $c_{\cL_{g_\al,\om}}(\cdot)$ is defined in 
Lemma~\ref{total growth} by the equation $t_{x_{\al,\om,\onh}}=c_{\cL_{g_\al,\om}}(\onh)\onh$ and $c_0, j_\al$ are as in 
Theorem~\ref{toest}.  

Parts~\eqref{gpth2a},\eqref{gpth2b},\eqref{gpth1},\eqref{gpth3} of the Theorem follow directly from Lemma~\ref{total growth}.
 We now prove part~\eqref{gpth2c} in a similar manner to the argument for Theorem~\ref{mtcf}\eqref{p.t.1} given above. 
 For any $\eps>0$, $\onh\in\cO_S^\times\cap \bN$ we have by Theorem~\ref{general case} that the measure
 $\mu_{x_{\al,\om,\onh}}$ satisfies the estimate~\eqref{yaeq}.
 For any $\ga\in\br{\om}_{\on{br}}$ with $\height{\ga}=\onh$ we have by 
 Lemma~\ref{lemconnection}  that the measure $\mu_{\ga\al}$ is equal either to $\mu_{x_{\al,\om,\onh}}$ or to $\theta_*\mu_{x_{\al,\om,\onh}}$. In any case, as $\theta$ is an isometry of $X_\infty$, the measure $\mu_{\ga\al}$ satisfies~\eqref{yaeq} as well. 
 Applying Theorem~\ref{toest}
 with $T^{-1}=C\onh^{-\frac{\del_0}{2}+\eps}$, where $C$ is the implicit constant in~\eqref{yaeq}, we obtain  
 \begin{equation}\label{eqtoest1331}
 \av{\frac{\av{P_{\ga\al}}}{t_{\ga\al}}-\frac{1}{j_\al c_0}}\ll_{\al,\om,S,\eps} \onh^{-\frac{\del_0}{6}+\eps}.
 \end{equation}
Note though that we may apply Theorem~\ref{toest} only when $T>T_0$ and so we choose the implicit constant in~\eqref{eqtoest1331} 
to be big enough to handle the finitely many $\onh$'s
for which $T\le T_0$.  

 By Lemma~\ref{lemconnection} we have that $t_{x_{\al,\om,\onh}}=t_{\ga\al}$ and so by the definition of $c(\al,\om,\onh)$ we obtain that 
 $t_{\ga\al}=j_\al c_0c(\al,\om,\onh)\onh$. Substituting this in~\eqref{eqtoest1331} and recalling that by part~\eqref{gpth1} of the Theorem -- that was already established above --
  $c(\al,\om,\onh)\ll_{\al,S} 1$, we obtain $\av{P_{\ga\al}}=c(\al,\om,\onh)\onh+O_{\al,\om,S,\eps}(1)\onh^{1-\frac{\del_0}{6}+\eps}$ as desired. The last statement regarding the big $O$ in the non-split case follows from the fact that in this case the implicit
constant $C$ from~\eqref{yaeq} may be chosen independent of $\om$.
\end{proof}
\begin{proof}[Proof of Theorem~\ref{period cor new}]
For any $\al\in\qi$ and $\ga\in\Ga_S$, let $\onh=\height{\ga}$ and choose $\om\in K_S$ so that $\ga\in\br{\om}_{\on{br}}$. Let 
$\tilde{c}(\al,\ga)\defi c(\al,\om,\onh)$ where $c(\al,\om,\onh)$ is the function appearing in Theorem~\ref{gpth}. Note that although
the choice of $\om$ is not unique, the value $\tilde{c}(\al,\ga)$ is well defined. We prove that $\tilde{c}(\cdot,\cdot)$ satisfies 
the conclusions of the theorem where the change of notation is in order to avoid confusion.

Parts~\eqref{gpthintro3},\eqref{gpthintro5} of the theorem follow directly from the corresponding part of Theorem~\ref{gpth}. 
Part~\eqref{gpthintro1} of the theorem follows from the corresponding part of Theorem~\ref{gpth} with the additional remark that
for any $q\in \cO_S^\times$, $\ga_q\in\br{\om}_{\on{br}}$ for some $\om\in\Om$, where $\Om$ is as in~\eqref{Omegaa}. The 
finiteness of $\Om$ implies that the big $O$ in~\eqref{eqgpthintro1} is independent of the generalized branch 
(as opposed to the big $O$ in~\eqref{eqgpth}). For the same reason, part~\eqref{gpthintro4} of the theorem follows from the corresponding part of Theorem~\ref{gpth}. Finally, part~\eqref{gpthintro1c} of the theorem also follows from the corresponding part of Theorem~\ref{gpth} where here we need to remark that if $q_n=\ell^{(1)}_n/\ell^{(2)}_n$ is a sequence as in part~\eqref{gpthintro1c} of the theorem, then
$\ga_{q_n}\in\br{\om}_{\on{br}}$ for a fixed choice of $\om\in\Om$ 
\end{proof}
\section{Proof of Theorem~\ref{toest}}\label{potoest}
The proof of Theorem~\ref{toest} utilizes and expands on the tight connection between the geodesic flow and the Gauss map. 
This connection was discovered by Artin~\cite{Artin}, who used the flexibility of continued fractions to construct dense geodesics.
As we will need to use technical aspects of this connection, we choose to give below a brief -- essentially self
contained -- treatment which allows us to introduce the language and notation needed in the proof of Theorem~\ref{toest}. We 
follow closely the notation and exposition of~\cite[\S9.6]{EW} (see also~\cite{Series}).

Our notation henceforth will differ slightly from the notation used in previous sections. We elaborate about these changes:
Note that the natural map $\PSL_2\to\PGL_2$ induces an isomorphism between the quotients
 $\PSL_2(\bZ)\backslash\PSL_2(\bR)$ and $\PGL_2(\bZ)\backslash\PGL_2(\bR)$. Due to the geometric nature of the arguments
 below, it would be easier for us to work with the space $X\defi\Ga\backslash G$, where $G\defi \PSL_2(\bR)$ and  
 $\Ga\defi\PSL_2(\bZ)$ rather than with the quotient of $\PGL_2(\bR)$. As before, we shall abuse notation and treat elements of 
 $G$ as matrices rather than equivalence classes of such.
The group $G$ acts on the upper half plane $\bH\defi\set{z=x+iy:y>0}$ by M\"{o}bius transformations and this 
action preserves the hyperbolic metric $ds^2=\frac{dx^2+dy^2}{y^2}$ and so induces an action of $G$ on the unit tangent bundle 
$T^1\bH$. The action of $G$ on $T^1\bH$ is free and transitive hence allows us to identify $G$ with $T^1\bH$ once we choose a 
base point. We make the usual choice of the base point to be the tangent vector pointing upwards through $i\in\bH$. With this 
identification the geodesic flow on $G=T^1\bH$ corresponds to the action from the right of the positive diagonal subgroup  
$$A=\set{a(t)}=\set{\diag{e^{t/2},e^{-t/2}}:t\in \bR}<G.$$
\begin{remark}\label{the switch}
The reason we chose to work with $\PGL_2$ rather than $\PSL_2$ to begin with is as follows: We were trying to analyze the c.f.e of 
numbers of the form $q\al$ and therefore used the fundamental conjugation relation 
$\smallmat{q&0\\0&1}\smallmat{1&\al\\ 0&1}\smallmat{q^{-1}&0\\0&1}=\smallmat{1&q\al\\0&1}$. Working with $\PSL_2$ would have forced us to use conjugation by
say $\diag{q,q^{-1}}$ which would have produced results regarding $q^2\al$. 
\end{remark}
\subsection{Cross-sections}
 We now wish to introduce the notion of a cross-section. We are being rather restrictive below as we only want to discuss a specific example hence we see no use in greater generality. 
Given a Borel measurable set $C\subset X$,
we let $r_C:C\to\bR_{\ge 0}\cup\set{\infty}$ be defined by $r_C(x)=\inf\set{t>0:xa(t)\in C}.$ The function $r_C$ is called \textit{the return time} function to $C$. 
The set $C$
is called a \textit{cross-section} for $a(t)$ if the return time functions for positive and negative times are bounded from below by some fixed positive number  and the map $(x,t)\mapsto xa(t)$ from $\set{(x,t):x\in C, 0\le t<r_C(x)}\to X$ is a measurable isomorphism onto its image in $X$. 
The \textit{first return map} $T_C$ is defined to be $T_C(x)=xa(r_C(x))$, where this makes sense; i.e.\ for $x$ belonging to $\set{x\in C:r_C(x)<\infty}$. In fact, we will be interested only in points which return infinitely often in the future and past to $C$, thus we define the domain of the first return map to be
\begin{align}\label{domain of first return}
\on{Dom}_{T_C}=\{x\in C: &\textrm{ there are infinitely many}\\
 \nonumber&\textrm{ positive and negative $t$'s with $xa(t)\in C$}\}.
\end{align}
Note that $T_C:\on{Dom}_{T_C}\to\on{Dom}_{T_C}$ is invertible.

We now wish to define the relevant cross-section for the geodesic flow in $X$.
An element $g\in G$ represented by a matrix $\smallmat{a&b\\c&d}$ corresponds to a tangent vector of unit length to the upper half plane. It then defines a geodesic in $\bH$ which hits the boundary of $\bH$ in two points. We denote the \textit{endpoint} and \textit{startpoint} of the geodesic it defines by $e_+(g),e_-(g)$ respectively. Clearly we have $e_+(g)=\frac{a}{c}, e_-(g)=\frac{b}{d}$, where we allow $\infty$ as a possible value. Any element $g\in G$ has a unique decomposition  (the Iwasawa decomposition) of the form 
\begin{equation}\label{Iwasawa}
g=n(t)a(s)k_\theta=\pa{\begin{array}{ll}1&t\\0&1\end{array}}\pa{\begin{array}{ll}e^{s/2}&0\\0&e^{-s/2}\end{array}}\pa{\begin{array}{ll}\cos \theta&-\sin\theta\\\sin\theta&\cos \theta\end{array}},
\end{equation}
where $t,s\in\bR$, and $\theta\in[0,\pi)$. The notation $n(t),a(s),k_\theta$ should be understood from~\eqref{Iwasawa}. An element $g$ having the above decomposition corresponds to the tangent vector to the point $t+ie^s\in\bH$ of angel $2\theta$ in the \textit{clockwise} direction from the vector pointing upwards. 
 Consider the following sets:
\begin{align}\label{the cross section}
\nonumber \cC^+&=\set{g=a(s)k_\theta \in G : e_+(g)\in(0,1), e_-(g)<-1};\\
\cC^-&=\set{g=a(s)k_\theta\in G : e_+(g)\in(-1,0), e_-(g)>1};\\
\nonumber \cC&= \cC^+\cup \cC^-.
\end{align} 
The set $\cC$ consists of those tangent vectors whose base-point lies on the imaginary axis with some restriction on the angle $\theta$ related to the height $e^s$ of the base point. It should be clear from the geometric picture described above that the range of `allowed angles' for such a tangent vector, say in $\cC^+$, is a subinterval
of 
$(\frac{\pi}{4},\frac{\pi}{2})$ with $\frac{\pi}{2}$ being its right-end-point. In \S\ref{construction of phi} we will workout these intervals exactly.
Let $\pi:G\to X$ be the quotient map. 
We denote the sets $\pi(\cC),\pi(\cC^+),\pi(\cC^-)$ by $C,C^+,C^-$ respectively.
The following lemma  
is proved in~\cite[\S9.6]{EW}.
\begin{lemma}\label{l.1} 
The following hold
\begin{enumerate}
\item The set $\cC$ injects into $X$ under $\pi$; that is, for each $x\in C$ corresponds a unique $g\in\cC$ with $\pi(g)=x$.
\item The set $C$ is a cross-section for the geodesic flow on $X$. 
\item\label{l.1.2} The domain of $T_C$ corresponds to those $g\in\cC$ for which both $e_+(g),e_-(g)$ are irrational. 
\item\label{l.1.3} For $g\in\cC^+$, if $T_C(\pi(g))$ is defined, then $T_C(\pi(g))\in C^-$. An analogue statement with $+$ replaced by $-$ holds. 
\end{enumerate}
\end{lemma}
It will be convenient for us to introduce a `thickening' of the cross-section $C$ which will denoted by $B$. The following lemma is left to be verified by the reader.
\begin{lemma}\label{the set B1} There exists a constant $\eps_0>0$ (which will be fixed throughout) such that the following statements hold
\begin{enumerate}
\item\label{B.1} The the map $(g,t)\mapsto ga(t)$ from $\cC\times (0,\eps_0)$ to the set 
\begin{equation}\label{B}
\cB\defi\set{ga(t):g\in\cC,t\in(0,\eps_0)}
\end{equation}
 is  one to one and onto, and the set $\cB$ is open in $G$.
\item\label{B.2} Let $B\defi\pi(\cB)$. The restriction $\pi:\cB\to B$ is one to one and onto and the set $B\subset X$ is open. 
\end{enumerate}
\end{lemma} 
The constant $\eps_0$ introduced in the above lemma  is a lower bound for the return time function, $r_C$, to the cross-section $C$. 
The importance of part~\eqref{B.2} of the above lemma is that it gives us a well defined way of lifting points in $X$ near the cross-section to the group $G$ in which it is more convenient to work. The combination of parts~\eqref{B.1} and \eqref{B.2} gives us natural coordinates on $B$; any point  
$x\in B$ can be written uniquely as $x_Ca(t)$ where $x_C\in C$ and $t\in(0,\eps_0)$.

In our discussion we will encounter certain measures on  the cross-section $C$ which are invariant under the first return map and we will need a procedure to construct from them measures on the ambient space $X$ which are invariant under the geodesic flow; that is, under the action of the group $A$.

\begin{definition}\label{def suspension}
Let $\tilde{\mu}$ be a probability measure on $C$. We define the \textit{suspension} of $\tilde{\mu}$ to be the measure $\sig_{\tilde{\mu}}$ on $X$ which is given by the following rule of integration: For $f\in C_c(X)$
\begin{equation}\label{suspension}
\int_X f(x)d \sig_{\tilde{\mu}}(x)=\int_C\int_0^{r_C(x)}f(xa(t))dtd\tilde{\mu}(x).
\end{equation} 
\end{definition}
\begin{lemma}\label{basic facts 1}
If $\tilde{\mu}(\on{Dom}_{T_C})=1$ and $\tilde{\mu}$ is $T_C$-invariant, then the suspension $\sig_{\tilde{\mu}}$ is $A$-invariant. Furthermore, $\sig_{\tilde{\mu}}(X)=\int_C r_Cd\tilde{\mu}$. 
\end{lemma}
\begin{proof}
This is follows from~\cite[Lemma 9.23]{EW} taking into account that $T_C$ is invertible on $\on{Dom}_{T_C}$. 
\end{proof}

\begin{definition}\label{r.0}
Given a function $f:C\to \bC$, we denote by $\widehat{f}:X\to \bC$ the following function 
\begin{align*}
\widehat{f}(x)=\bigg\{\begin{array}{ll} f(x_C)&\textrm{ if }x\in B\textrm{ has coordinates }(x_C,t),\\ 0&\textrm{ if }x\notin B\end{array}
\end{align*}
\end{definition}
Note that with the above definition, given a measure $\tilde{\mu}$ on $C$ and a function $f:C\to \bC$,  equation~\eqref{suspension} translates to the following useful formula which will be used frequently below
\begin{equation}\label{suspension 2}
\int_X \widehat{f}d\sig_{\tilde{\mu}}=\eps_0\int_Cfd\tilde{\mu}.
\end{equation}

\subsection{The Gauss map}\label{g.m section}
Let $I=(0,1)$ and $S:I\to I$ be the Gauss map; i.e.\ the map defined by the formula $S(y)=\frac{1}{y}-\lfloor\frac{1}{y}\rfloor$. Note that strictly speaking $S(y)$ is not in $I$ for points of the form 
$y=\frac{1}{m}$. The reader will easily verify that $S^n(y)$ is well defined for all positive $n$ if and only if $y$ is irrational. This slight inconvenience will not bother us
as we will only apply the Gauss map to irrational points. Let $I_{\on{irr}}=I\smallsetminus \bQ$.  
Consider the following subsets of $\bR^2$: 
\begin{align}\label{domain of i.e}
D=\set{(y,z): y\in I, 0<z< \frac{1}{1+y}},\;
D_{\on{irr}}=\set{(y,z)\in D:y\in I_{\on{irr}}}. 
\end{align} 
Let $\bar{S}:D\to D$ be the map given by $\bar{S}(y,z)=(S(y),y(1-yz))$ and note similarly that strictly speaking, in order to iterate $\bar{S}$ as many times as we wish we
need to restrict to points in $D_{\on{irr}}$. 
Recall (see for example~\cite[\S3.4]{EW}) that the normalized restriction of the Lebesgue measure on $\bR^2$ to $D$, which we denote here by $\lam$, is an $\bar{S}$-invariant probability measure. 
This is the so called \textit{invertible\footnote{The term `invertible' refers to the fact that when restricted to a subset of $D$, $\bar{S}$ is indeed invertible. This subset
is obtained by neglecting a certain set of Lebesgue measure zero (see~\cite[Prop.\ 3.15]{EW}).} extension of the Gauss map} as when one projects on the first coordinates, one recovers the Gauss map and the Gauss-Kuzmin measure $\nugauss$ introduced in the introduction. That  is if $p:D\to I$ denotes the projection on the first coordinate, then
\begin{equation}\label{projects correctly}
p_*\lam=\nugauss.
\end{equation}

\subsection{Relation to the Gauss map}
%
%
 
%
Consider the maps $\tau_+:C^+\to D, \tau_-:C^-\to D$  defined by the following formulas: For $x=\pi(g)\in C$, where 
$g=\smallmat{a&b\\c&d}\in \cC$ : 
\begin{align}\label{tau}
\textrm{For }g\in \cC^+,\;\tau_+(x)&= (e_+(g),\frac{1}{e_+(g)-e_-(g)})=(\frac{a}{c},cd),\\
\nonumber \textrm{For }g\in\cC^-,\;\tau_-(x)&=(-e_+(g),\frac{1}{-e_+(g)+e_-(g)})=(-\frac{a}{c},-cd).
\end{align}
We let $\tau:C\to D$ be the union of $\tau_+$ and $\tau_-$. The formulas in~\eqref{tau} can be stated geometrically as follows: For a tangent vector $g\in\cC$ and $x=\pi(g)$, $\tau(x)=(y,z)\in D$, where $y$ is the absolute value of the end point of the semicircle corresponding to $g$ and $z^{-1}$ is the diameter of it. 
For any endpoint $y\in(0,1)$ (resp.\ $y\in(-1,0)$) and any diameter $z^{-1}>1$, we can attach a well defined semicircle in $\bH$ which corresponds to a unique point in $\cC^+$ (resp.\ $\cC^-$).
This shows that $\tau$ is one to one and onto (and in fact, a homeomorphism) from $C^+$ (resp.\ $C^-$) to $D$ which is the area below the graph of the function $y\mapsto(1+y)^{-1}$. 
The following basic lemma  is proved in~\cite[\S9.6]{EW}. It establishes the link between the geodesic flow and the Gauss map. 
\begin{lemma}\label{isomorphism} 
The following diagram commutes (for points $x\in C$ for which $T_C(x)$ is defined)
$$\xymatrix{C\ar[d]_{\tau}\ar[r]^{T_C}&C\ar[d]^\tau\\
D\ar[r]_{\bar{S}}&D.}$$
\end{lemma}
Note that $\tau:C\to D$ is `almost' an isomorphism (it is two to one), and so the above lemma basically says that any dynamical question about the system $\bar{S}:D\to D$ can be pulled to a corresponding question on $T_C:C\to C$. In our case the dynamical question is that of 
equidistribution of certain $\bar{S}$-invariant measures. Using the suspension construction we will see that the equidistribution questions for the dynamical system $T_C:C\to C$ translate to equidistribution questions of certain $A$-invariant measures on $X$.

We will be interested in two types of measures on the cross-section $C$ defined above. The first is the following version of the Lebesgue measure:
We use $\tau_+$ (resp.\ $\tau_-$) to pull the (normalized restriction of) Lebesgue measure $\lam$ from $D$ to $C^+$ (resp.\ $C^-$) and denote the resulting measure  by $\tilde{\lam}^+$ 
(resp. $\tilde{\lam}^-$). Further denote $\tilde{\lam}=\frac{1}{2}\tilde{\lam}^++\frac{1}{2}\tilde{\lam}^-$. 
Clearly $\tilde{\lam}$ is $T_C$-invariant and $\tau_*(\tilde{\lam})=\lam$. 

 The second type of measures on $C$ are those coming from quadratic irrationals. We recall Definitions~\ref{def g alpha}, \ref{def per qi}. 
 Let $\al\in\qi$ and let $g_\al$ be as in~\eqref{g alpha}. We chose to define $g_\al$ as we did so as to ensure that its determinant is positive and hence it corresponds naturally to an element of $G$ with endpoint
 $\al$. Let $x_\al\in X$ be the corresponding point (that is $x_\al=\pi(\frac{1}{\sqrt{\det g_\al}}g_\al)$) and $\mu_\al$ the $A$-invariant probability measure supported on the periodic orbit $x_\al A=\set{x_\al a(t):t\in[0,t_\al)}$, where $t_\al$ is the length of the orbit. 
 We claim that the intersection $C\cap x_\al A$ is a non-empty finite set contained in $\on{Dom}_{T_C}$. In fact, any geodesic in the upper half plane that corresponds to a semi-circle, projects to a set in $X$ that intersects $C$ non-trivially. By Lemma~\ref{l.1}\eqref{l.1.2}, if the end points of the geodesic are irrational, the intersection is in $\on{Dom}_{T_C}$. Finally, the finiteness follows from the fact that $C$ is a cross-section together with the fact that the orbit $x_\al A$ is of finite length. 

Let us denote by $\tilde{\mu}_\al$ the normalized counting measure on $C\cap x_\al A$. Clearly $\tilde{\mu}_\al$ is invariant under the first return map $T_C$. 
Let us denote the $G$-invariant probability measure on $X$ by $m_X$. The following lemma links between the measures $m_X,\mu_\al$ and the suspensions $\sig_{\tilde{\lam}}$, $\sig_{\tilde{\mu}_\al}$ given in Definition~\ref{def suspension}.
\begin{lemma}\label{using suspension}
Let $\al\in\qi$. The suspensions $\sig_{\tilde{\lam}},\sig_{\tilde{\mu}_\al}$ of the probability measures
$\tilde{\lam},\tilde{\mu}_\al$ are proportional to $m_X,\mu_\al$ respectively.
\end{lemma}
\begin{proof}
The fact that $\sig_{\tilde{\lam}}$ is proportional to the Haar measure $m_X$ is proved in~\cite[p.\ 325-326]{EW}. The outline of the proof is as follows: By Lemma~\ref{basic facts 1}, $\sig_{\tilde{\lam}}$ 
is $A$-invariant. One shows that it is absolutely continuous with respect to $m_X$ and deduces the result from the ergodicity of $m_X$ with respect to the $A$-action. 
Regarding $\sig_{\tilde{\mu}_\al}$, note that it is clearly a measure that is supported on the orbit $x_\al A$ and it is $A$-invariant by Lemma~\ref{basic facts 1}. The assertion now follows from the uniqueness (up to proportionality) of an $A$-invariant measure on the periodic orbit $x_\al A$.
\end{proof}
\begin{definition}\label{constants}
Let $c_0$ be  the absolute constant satisfying $\sig_{\tilde{\lam}}=c_0m_X$.
Similarly, for any
$\al\in \qi$ let $c_\al$ be the constant satisfying $\sig_{\tilde{\mu}_\al}=c_\al\mu_\al$.
\end{definition}
The following lemma is the last bit of information we need in order to translate the statement of Theorem~\ref{toest} to the cross-section.
\begin{lemma}\label{projects to nu alpha}
Let $p:D\to I$ be the projection on the first coordinate. Then 
\begin{align}\label{right projections}
&(p\circ\tau)_*(\tilde{\lam})=\nugauss,\\
\nonumber &(p\circ\tau)_*(\tilde{\mu}_\al)=\nu_\al.
\end{align}
\end{lemma}
\begin{proof}
The first equality in~\eqref{right projections} follows from the fact that $\tau_*(\tilde{\lam})=\lam$ (which is basically the definition of $\tilde{\lam}$) and the observation $p_*(\lam)=\nu$ which was pointed out in~\eqref{projects correctly}. We argue the second equality:
By Lemma~\ref{isomorphism} the measure $\tau_*(\tilde{\mu})$ is  $\bar{S}$-invariant. By the above discussion it is finitely supported. Since $x_\al A$ is a loop, the first return map $T_C$ acts transitively on the 
support of $\tilde{\mu}$ and so the support of  $\tau_*(\tilde{\mu})$ consists of a single $\bar{S}$ orbit. This implies that $(p\circ \tau)_*(\tilde{\mu})$
is supported on a single periodic orbit of the Gauss map $S$. Denote this period by $P_\al'$. We need to argue why $P_\al=P_\al'$, which is equivalent to $P_\al\cap P_\al'\ne\varnothing$.

Consider the matrix $g_\al$ defined in~\eqref{g alpha}. The tangent vector corresponding to $g_\al$ defines a geodesic in $T^1\bH$ which is a semicircle with endpoint 
$e_+(g_\al)=\al$. At some point along this geodesic we find a point $g$ which projects to $C^+$ under $\pi$. Let $x=\pi(g)\in C^+$ and $g'\in\cC^+$ the corresponding point in $\cC^+$. 
Clearly $x$ is in the support of $\tilde{\mu}_\al$ and hence the endpoint $e_+(g')=p\circ\tau(x)$ is a point of $P_\al'$. 
As $\pi(g)=x=\pi(g')$ we deduce that there exists $\ga\in\Ga$ such that $\ga g=g'$. Therefore
the semicircle corresponding to $g$  and to $g'$ are related by the action of $\ga$  as a M\"obius transformation. It follows that the endpoints $\al,e_+(g')$ are related  by the action of $\ga$ as well. By Theorem~\ref{tail equiv} this action can effect only finitely many digits of the c.f.e of $\al$ and we conclude that the periods of the c.f.e of $\al$ and of $e_+(g')$ must be the same (up to a possible cyclic rotation) which finishes the proof.
\end{proof}
Finally, in light of~\eqref{right projections},  Theorem~\ref{toest} will follow if we prove the following
\begin{theorem}\label{effective t.1 new}
Let  $\al\in \qi$. There exists an absolute constant $T_0>1$ so that if we assume that for some $T>T_0$ the estimate   
$\av{\int f d\mu_\al-\int fd m_X}\le \max\set{\norm{f}_2,\ka}T^{-1}$ holds for any $f\in\on{Lip}_\ka(X)\cap L^2(X,m_X)$,
then the following two statements hold
\begin{enumerate}
\item\label{toest111} For any $f\in\on{Lip}_\ka(D)$, and any $\eps>0$ 
\begin{equation}\label{eq521}
\av{\int_C f\circ\tau d\tilde{\mu}_{\al}-\int_Cf\circ\tau d\tilde{\lam}}\ll_{\eps} \max\set{\norm{f}_\infty,\ka}T^{-\frac{1}{3}+\eps}.
\end{equation}
\item\label{toest112} The constant $c_0$ from Definition~\ref{constants}  satisfies $\av{\frac{\av{P_\al}}{t_\al}-\frac{1}{j_\al c_0}}\ll_\eps T^{-\frac{1}{3}+\eps}.$
\end{enumerate}
\end{theorem}
The argument yielding Theorem~\ref{effective t.1 new} is slightly technical because of the following issue: 
We start with a $\ka$-Lipschitz function $f:D\to\bC$ and construct  from it the function 
$\widehat{f\circ \tau}:X\to\bC$ as in Definition~\ref{r.0}. As we wish to appeal to
Theorem~\ref{general case} we need to remedy $\widehat{f\circ\tau}$  to be Lipschitz in a way that will allow us to control its Lipschitz  constant. 
In order to achieve this we shall need the following technical lemma which is proved in \S\ref{construction of phi}. 
\begin{lemma}\label{t.l}
For any $M>1$ and $0<\rho<1$ there exist a function $\vphi=\vphi_{\rho,M}:X\to [0,1]$ with the following properties
\begin{enumerate}
\item\label{t.l.1} The function $\vphi$ is $\rho^{-1}$-Lipschitz.
\item\label{t.l.2} We have $\int_X 1-\vphi dm_X\ll M^{-1}+\rho\log M$.
\item\label{t.l.3} Given $f:D\to \bC$ a $\ka$-Lipschitz function, the product $\widehat{f\circ\tau}\cdot\vphi:X\to \bC$  is Lipschitz with Lipschitz constant $\ll\max\set{\norm{f}_\infty,\ka}\rho^{-1}M$.
\end{enumerate}
\end{lemma}
\begin{proof}[Proof of Theorem~\ref{effective t.1 new}]
\eqref{toest111}. Let $T>1$ and  $\eps>0$ be fixed. Under the assumption in the statement of the Theorem we 
need to argue the validity of~\eqref{eq521}.
Let $f\in\on{Lip}_\ka(D)$ be given. Let $c_0,c_\al$ be as in Definition~\ref{constants}.
Using~\eqref{suspension 2} we have the following estimate:
\begin{align}\label{new desire new}
\av{\int_{C}f\circ\tau d\tilde{\mu}_{\al}-\int_{C}f\circ\tau d\tilde{\lam}}=&\av{\frac{c_\al}{\eps_0}\int_{X}\widehat{f\circ\tau}d\mu_{\al}-\frac{c_0}{\eps_0}\int_{X}\widehat{f\circ\tau}dm_X}\\
\nonumber \le& \underbrace{\av{c_\al-c_0}}_{(*)}\norm{f}_\infty\eps_0^{-1}+c_0\eps_0^{-1}\underbrace{\av{\int_{X}\widehat{f\circ\tau}d\mu_{\al}-\int_{X}\widehat{f\circ\tau}dm_X}}_{(**)}.
\end{align}
We first estimate the expression $(**)$ in~\eqref{new desire new}.  Given  $M>1, 0<\rho<1$ we let $\vphi=\vphi_{\rho, M}$ be as in Lemma~\ref{t.l} and denote $\psi=1-\vphi$.
\begin{align}\label{new one new}
(**)&=\av{\int_{X}\widehat{f\circ\tau}\cdot(\vphi+\psi)d\mu_{\al}-\int_{X}\widehat{f\circ\tau}\cdot(\vphi+\psi)dm_{X}}\\
\nonumber&\le\av{\int_{X}\widehat{f\circ\tau}\cdot\vphi d\mu_{\al}-\int_{X}\widehat{f\circ\tau}\cdot\vphi dm_{X}}+\av{\int_{X}\widehat{f\circ\tau}\cdot\psi d\mu_{\al}}+
\av{\int_{X}\widehat{f\circ\tau}\cdot\psi dm_{X}}.
\end{align}
We will estimate each of the three summands in the right hand side of the inequality~\eqref{new one new}.
By Lemma~\ref{t.l}\eqref{t.l.2} we have 
\begin{equation}\label{n.d.1 new}
\av{\int_{X}\widehat{f\circ\tau}\cdot\psi dm_X}\le\norm{f}_\infty \int_{X}\psi dm_X\ll\norm{f}_\infty (M^{-1}+\rho\log M).
\end{equation}
Next, note that by Lemma~\ref{t.l}\eqref{t.l.1} $\psi$ is $\rho^{-1}$-Lipschitz and so
our assumption together with the estimate~\eqref{n.d.1 new} yields 
\begin{align}\label{n.d.2 new}
\nonumber\av{\int_{X}\widehat{f\circ\tau}\cdot\psi d\mu_{\al}}&\le\norm{f}_\infty\int_{X}\psi d\mu_{\al}\\
\nonumber &\le\norm{f}_\infty\pa{\int_{X}\psi dm_X+\max\set{1,\rho^{-1}}T^{-1}}\\
&\ll\norm{f}_\infty\pa{M^{-1}+\rho\log M+\rho^{-1}T^{-1}}.
\end{align}
Finally, by Lemma~\ref{t.l}\eqref{t.l.3} our assumption applies to the Lipschitz function $\widehat{f\circ\tau}\cdot\vphi$ and we conclude the following
\begin{align}\label{n.d.3' new}
\av{\int_{X}\widehat{f\circ\tau}\cdot\vphi d\mu_{\al}-\int_{X}\widehat{f\circ\tau}\cdot\vphi dm_X}\le\max\set{\norm{f}_\infty,\ka}\rho^{-1}M
T^{-1}.
\end{align}
We now make the choice $M=\rho^{-1}=T^{\frac{1}{3}-\frac{\eps}{2}}$ and combine estimates~\eqref{n.d.1 new}, \eqref{n.d.2 new}, \eqref{n.d.3' new} into~\eqref{new one new} to obtain
\begin{equation}\label{n.d.3 new}
(**)\ll_{\eps} \max\set{\norm{f}_\infty,\ka}T^{-\frac{1}{3}+\eps},
\end{equation}
where in the above estimate we used $\rho\log M \ll_\eps T^{-\frac{1}{3}+\eps}$.

In order to finish we need to further estimate $(*)$ in~\eqref{new desire new}. To obtain this estimation from the above we take $f:D\to\bC$ to be identically $1$ and note that in this case
 $\widehat{f\circ\tau}=\chi_B$ and so using~\eqref{suspension 2}  we have 
 \begin{equation}\label{n.d.4 new}
 \av{\int_{X}\widehat{f\circ\tau}d\mu_{\al}-\int_{X}\widehat{f\circ\tau}dm_X}=\av{\mu_{\al}(B)-m_X(B)}=\av{\frac{\eps_0}{c_\al}-\frac{\eps_0}{c_0}}.
 \end{equation}
 The left hand side of~\eqref{n.d.4 new} is $(**)$ for this choice of $f$ and so by~\eqref{n.d.3 new} we obtain 
 \begin{equation}\label{n.d.5 new}
 \av{c_\al^{-1}-c_0^{-1}}\ll_{\eps} T^{-\frac{1}{3}+\eps}.
 \end{equation} 
We choose the absolute constant $T_0$ so that the inequality~\eqref{n.d.5 new} (applied say with $T=T_0$ and $\eps=\frac{1}{6}$) implies that 
$c_\al^{-1}>c_0^{-1}/2$ and so is bounded away from 0 by an absolute constant.
As the derivative of the function $x\mapsto x^{-1}$ is bounded for $x$'s bounded away from 0, we conclude from~\eqref{n.d.5 new} that 
 \begin{equation}\label{n.d.5'}
 (*)=\av{c_\al-c_0}\ll_{\eps} T^{-\frac{1}{3}+\eps}.
 \end{equation} 
  
 Plugging this estimation of $(*)$ together with~\eqref{n.d.3 new} to~\eqref{new desire new}
 we obtain the desired inequality~\eqref{eq521}.
 
 \eqref{toest112}.
 Let $\wt{P}_{\al}$ denote the support of $\tilde{\mu}_\al$. It follows from~\eqref{right projections} that $p\circ \tau(\wt{P}_\al)=P_\al$. We will show below in Lemma~\ref{parity lemma} that the map $p\circ\tau:\wt{P}_\al\to P_\al$ is $j_\al$ to 1 (that is,  
 two to one if $\av{P_\al}$ is odd or one to one if it is even), and so the inequality sought will follow once we show 
$\av{\frac{|\wt{P}_\al|}{t_\al}-\frac{1}{c_0}}\ll_\eps T^{-\frac{1}{3}+\eps}.$   

Recall that by~\eqref{n.d.4 new}
$\mu_{\al}(B)=\frac{\eps_0}{c_\al}$.  On the other hand, the geodesic $x_{\al}A$ which is of length $t_{\al}$ penetrates $B$ exactly $|\wt{P}_{\al}|$ times
and stays in $B$ along a time interval of length $\eps_0$ each time and so $\mu_{\al}(B)=\frac{|\wt{P}_{\al}|\cdot\eps_0}{t_{\al}}$. It follows that 
$c_\al^{-1}=\frac{|\wt{P}_{\al}|}{t_{\al}}$ and we conclude from~\eqref{n.d.5 new} the desired inequality 
$\av{\frac{|\wt{P}_\al|}{t_\al}-\frac{1}{c_0}}\ll_\eps T^{-\frac{1}{3}+\eps}.$ 

\end{proof}
\begin{lemma}\label{parity lemma}
For any $\al\in\qi$ the map $p\circ\tau:\wt{P}_\al\to P_\al$ is $j_\al$ to 1.
\end{lemma}
\begin{proof}
In what follows we do not always distinguish between the cross-section $C$ and the subset $\cC\subset G$ used to define it.
We first observe that if a semicircle in the upper half plane that corresponds to the geodesic $\set{ga(t)}$ projects under $\pi$ to a 
periodic geodesic then $e_-(g),e_+(g)\in\bR$ are quadratic irrationals that are Galois conjugates of each other. This implies in particular, that if we denote the lift of $\wt{P}_\al$ from $C$ to $\cC$ by $\wt{\cP}_\al$ and by $\wt{\cP}^\pm=\cC^{\pm}\cap\wt{\cP}_\al$, then $p\circ \tau$ is injective when viewed as a map from either $\wt{\cP}^+$ or from $\wt{\cP}^-$. 
This follows from the fact that $\tau:\cC^\pm\to D$ is one to one and onto and that according to the observation made above, the 
first coordinate $p\circ\tau(g)$ determines the second one as it is the reciprocal of the diameter of the corresponding semicircle.

This shows that the pre-image of a point in $P_\al$ is of size 1 or 2. Choose $\be\in P_\al$ and a pre-image of it $g\in\wt{\cP}_\al$. 
We  Apply Lemma~\ref{isomorphism} and follow the orbits $S^i(\be)$ $i=0,1,\dots$ and the orbit  $T^i_C(\pi(g))$ above it. If  
$\av{P_\al}$ is odd, then 
Lemma~\ref{l.1}\eqref{l.1.3} tells us that when the orbit in the unit interval closes up, the orbit in the cross-section cannot close up
(as it switched from $C^+$ to $C^-$ or vice versa), and therefore we see that each of $\wt{\cP}^\pm$ projects onto $P_\al$ and so the map is 2 to 1.
Similarly, in case $\av{P_\al}$ is even, when the orbit in the unit interval closes up 
Lemma~\ref{l.1}\eqref{l.1.3} tells us that the orbit in the cross-section must return to the the same set $C^+$ or $C^-$ that $\pi(g)$ 
belongs to and therefore it must close up by the injectivity which was observed at the beginning. It follows that
one of the sets $\wt{\cP}^\pm$ is empty while the other one projects onto $P_\al$, and so the map is 1 to 1.
\end{proof}
\section{Construction of $\vphi$ - Proof of Lemma~\ref{t.l}}\label{construction of phi}
\subsection{Motivation}
We start with a function $f:D\to \bC$ which is $\ka$-Lipschitz and we consider the function $\tilde{f}:X\to\bC$ given by 
$\tilde{f}= \widehat{f\circ\tau}$. The points of discontinuity of $\tilde{f}$ are contained in $\partial B$.  We wish to find an approximation of $\tilde{f}$
which is  not only continuous but for which we will have clear control on its Lipschitz constant. To achieve this, we  construct an auxiliary function $\vphi$ which vanishes
in an $\eps$-thickening of $\partial B$ and is equal to 1 outside a $2\eps$-thickening of $\partial B$. This will clearly make $\tilde{f}\cdot\vphi$ continuous, but in order to control 
its Lipschitz constant we will have to make $\vphi$ vanish `high in the cusp' where the differential of $\tau$ explodes (see Lemma~\ref{tau bound} below). Along the 
construction we need to pay attention to two more quantities which we should control: The Lipschitz constant of $\vphi$ and $\int\psi$, where $\psi=1-\vphi$. These clearly  fight one against the other; in order to make $\int \psi$ small we wish to take $\eps$ (which control the above thickening) to be small which makes the Lipschitz constant of
$\vphi$ large.

Below, in~\S\ref{g.m.o}-\ref{e.n.d},  we discuss a somewhat eclectic collection of observations that we will use  in order to carry out  the arguments in~\S\ref{t.a} with 
little interruption. 
\subsection{General metric observations}\label{g.m.o}
Let $(Y,\on{d})$ be a metric space. For a subset $F\subset Y$ we denote $$(F)_\eps=\set{y\in Y:\on{d}(y,F)\le\eps};$$that is, the set of all points of distance $\le\eps$ from $F$. The following general
construction allows us to build Lipschitz functions in abundance. The proof is left to the reader.
\begin{lemma}[Fundamental construction]\label{Lip construction}
Let $(Y,\on{d})$ be a metric space and $F\subset Y$ a subset. For $\eps>0$ define $\vphi_{\eps,F}:Y\to[0,1]$ by $\vphi_{\eps,F}(y)=\min\set{1,\eps^{-1}\on{d}(y,F)}$. Then $\vphi_{\eps,F}$ attains the constant 
values $0$ on $F$  and $1$ on $Y\smallsetminus (F)_\eps$. Furthermore, $\vphi_{\eps,F}$ is $\eps^{-1}$-Lipschitz.  
\end{lemma}
We now make two remarks regarding Lipschitz constants:
\begin{remark}\label{shabat r.1}
Consider two functions, $f:Y\to\bC$ and $\vphi: Y\to[0,1],$ on a metric space $(Y,\on{d})$ and assume that they are $\ka_f,\ka_\vphi$-Lipschitz respectively with $\ka_\vphi\ge 1$.
Then, for any $x,y\in Y$ we have 
\begin{align*}
\av{f\cdot\vphi(x)-f\cdot\vphi(y)}&\le\av{f(x)-f(y)}\vphi(x)+\av{f(y)}\av{\vphi(x)-\vphi(y)}\\
&\le2\max\set{\ka_f,\norm{f}_\infty}\ka_\vphi\on{d}(x,y),
\end{align*}
that is $f\cdot\vphi$ has Lipschitz constant $\ll\max\set{\ka_f,\norm{f}_\infty}\ka_\vphi$.
\end{remark}
\begin{remark}\label{shabat r.2}
Let $f:Y\to\bC$ be a continuous function on a metric space $(Y,\on{d})$ in which between any two points $x,y$ there exists a path whose length equals $\on{d}(x,y)$.
Suppose there is an open cover $\set{U_i}$ of $\on{supp}(f)$ such that for each $i$ the restriction $f:U_i\to \bC$ is $\ka$-Lipschitz. Then we claim that $f$ is $\ka$-Lipschitz as a function on $Y$. To see this, take two points $x,y\in Y$ and connect them by a path $\ga$ whose length is $\on{d}(x,y)$. As $f$ is assumed to be continuous we can turn the open cover $\set{U_i}$ of the support of $f$ to an open cover of $Y$ by joining in the open set $U_0=Y\smallsetminus\on{supp}(f)$. Clearly
$f$ is $\ka$-Lipschitz on $U_0$ as well. Now let $\eps>0$ be a Lebesgue number for the induced open cover of the path $\ga$. Choose points 
$x=x_0,x_1\dots x_n=y$ on $\ga$ in a monotone way (so that $\on{d}(x,y)=\sum_1^n\on{d}(x_i,x_{i-1})$) and such that the distance between $x_i$ to $x_{i-1}$ is less
than $\eps$. It follows that for each $1\le i\le n$ there exists an open set from the cover $U_{j_i}$ such that $x_{i-1},x_i\in U_{j_i}$. As $f$ is assumed to be $\ka$-Lipschitz on $U_{j_i}$, we conclude that 
$$\av{f(x)-f(y)}\le\sum_1^n\av{f(x_i)-f(x_{i-1})}\le\sum_1^n\ka\on{d}(x_i,x_{i-1})=\ka\on{d}(x,y).$$
\end{remark}
\subsection{Coordinates}
We wish to define a convenient coordinate system which will allow us to carry out the relevant computations.
Recall  the open subsets $B,\cB$ of  of $X,G$ respectively that were defined in Lemma~\ref{the set B1}. 
We define similarly to~\eqref{the cross section} 
\begin{align}
\cB^+&=\set{ga(t):g\in\cC^+,t\in(0,\eps_0)}\\
\nonumber\cB^-&=\set{ga(t):g\in\cC^+,t\in(0,\eps_0)}.
\end{align}
A point $g\in \cB$ can be written uniquely in the form $a(s)k_\theta a(t)$ where $s\in\bR$, $t\in(0,\eps_0)$ and the angle $\theta\in[0,\pi)$ has some restrictions on it, arising from the requirements about the endpoints of the semicircle corresponding to $g$.  We shall refer to $(s,\theta,t)$ as the \textit{coordinates} of the point $g\in\cB$ or of the corresponding point $\pi(g)\in B$.

As the action of $a(t)$ from the right does not effect the endpoints, the restrictions on the $\theta$-coordinate are a function of $s$ alone. We workout these restrictions for, say, 
$g\in\cB^+$: We already observed (after~\eqref{the cross section}) that $\theta\in(\frac{\pi}{4},\frac{\pi}{2})$ (in order to ensure that $e_+(g)\in(0,1)$).
It is easy to see from the definition of the start and end points that for $s\in\bR$, $a(s)k_\theta\in\cC^+$, where $\theta\in[0,\pi)$, 
if and only if $e^s\cot \theta\in(0,1)$ and $-e^s\tan\theta<-1$. This is equivalent to saying $\tan\theta\in (\min\set{e^{s},e^{-s}},\infty)$. We choose an inverse $\tan^{-1}:\bR\to (0,\frac{\pi}{2})$ and conclude
that for a given $s$, the range of allowed angles for points $g\in\cB^+$ with coordinates $(s,\theta,t)$, is an interval $I_s^+$ which is defined by
\begin{equation}\label{theta interval}
I_s^+=(\theta_{\on{min}}(s),\frac{\pi}{2}),\textrm{ where }\theta_{\on{min}}(s)=\tan^{-1}(\min\set{e^s,e^{-s}})>\frac{\pi}{4}.
\end{equation}
Let us denote 
\begin{equation}
\cE^+=\set{(s,\theta,t)\in\bR^3:s\in\bR,t\in(0,\eps_0),\theta\in I_s^+},
\end{equation}
and define similarly $\cE^-$ and $\cE=\cE^+\cup\cE^-.$
Let $\xi:\bR^3\to G$ be the function
\begin{equation}\label{xi}
\xi(s,\theta,t)= a(s)k_\theta a(t).
\end{equation}
Clearly, we have $\xi(\cE)=\cB$, $\xi(\cE^+)=\cB^+$, and $\xi(\cE^-)=\cB^-$.

\begin{lemma}\label{cover link}
There is an absolute constant $c$ such that for any $\eps>0$, an $\eps$-ball in $\cE$ is mapped by $\xi$ into a ball of radius $c\eps$ in $\cB$.
\end{lemma}
In the course of the proof of Lemma~\ref{cover link} we will use the following elementary observation
\begin{lemma}\label{lengths of paths}
Let $h(t)$ be a one parameter subgroup of $G$. Then for any $g\in G$, $\on{d}_{G}(g,gh(t))\le||\dot{h}(0)||t$, where 
$||\dot{h}(0)||$ is the norm of the derivative of $h(t)$ at the identity.
\end{lemma}
\begin{proof}
 We link any two points $g_i=\xi(s_i,\theta_i,t_i)\in\cB, i=1,2$ by the path which changes linearly the $s$-coordinate first, then the $\theta$-coordinate, and finally the $t$-coordinate. Each such change corresponds to the action from the right by a one-parameter subgroup $h(t)$ as in Lemma~\ref{lengths of paths}. 
The change in
the $s$-coordinate corresponds to $h(s)=a(-t_1)k_{-\theta_1}a(s)k_{\theta_1}a(t_1)$, the change in the $\theta$-coordinate corresponds to $h(\theta)=a(-t_1)k_\theta a(t_1)$, and finally, the change in the $t$-coordinate corresponds to $h(t)=a(t)$. As the family of one-parameter subgroups that are involved in this process are  
conjugations of $a(t)$ and  $k_\theta$, where the conjugating element is varying in a compact set, we conclude that the norm of the derivative at the identity $\dot{h}(0)$ is $\ll 1$ for some absolute implicit constant. Lemma~\ref{lengths of paths} implies then that 
$$\on{d}_G(g_1,g_2)\ll \av{s_1-s_2}+\av{\theta_1-\theta_2}+\av{t_1-t_2},$$
which establishes the claim.
\end{proof}
\subsection{Height} 
The map $\tau$ defined in~\eqref{tau} was considered so far as a map from the cross-section $C$. As we wish to use differentiation it will be more convenient to extend it to a map $\tau:B\to D$ in the 
following way: Given a point $x\in B$ it can be written uniquely as $x_Ca(t)$ where $x_C\in C$ and $t\in(0,\eps_0)$. We define $\tau(x)=\tau(x_C)$; that is, we view $\tau$ as a function on $B$ which is 
constant along the direction of the geodesic flow.

As will be seen shortly, the norm of the differential of $\tau:B\to D$ is not bounded and so, in order to be able to control the Lipschitz constant
of the function appearing in Lemma~\ref{t.l}\eqref{t.l.3} we need to force its support to be contained in a domain in which we have some control on $\norm{\on{d}\tau}$.

Recall the Iwasawa decomposition~\eqref{Iwasawa}. Let $\cF$ denote the usual fundamental domain of $\Ga$ in $G$, that is, 
\begin{align}\label{f.d.1}
\cF&=\set{n(t)a(s)k_\theta\in G:\av{t}< \frac{1}{2}, t^2+e^{2s}> 1},\\
\nonumber \overline{\cF}&=\set{n(t)a(s)k_\theta\in G:\av{t}\le \frac{1}{2}, t^2+e^{2s}\ge 1}
\end{align}
We define the height function $\on{ht}:G\to\bR$ to be $\on{ht}(g)=e^s$ if $g=n(t)a(s)k_\theta$. 
This is indeed the imaginary coordinate of the base-point of  the tangent vector to $\bH$ corresponding to $g$.
This function respects the identifications induced by $\Ga$ on the boundary of $\overline{\cF}$ and so descends to a function (which we
continue to denote $\on{ht}(\cdot)$) on $X$.
For any $M>1$ we let 
\begin{align}\label{the height}
\cH_M&=\set{g\in \overline{\cF}:\on{ht}(g)\ge M},\quad\cK_M=\set{g\in \overline{\cF}:\on{ht}(g)< M};\\
\nonumber H_M&=\set{x\in X:\on{ht}(x)\ge M},\quad K_M=\set{x\in X:\on{ht}(x)< M}.
\end{align}
\begin{remark}\label{r.7}
It is well known that $m_X(H_M)=m_G(\cH_M)=M^{-1}$, which is an identity that will be needed later
(need to add reference).
\end{remark}
\subsection{Estimating norms of differentials}\label{e.n.d}
\begin{lemma}\label{tau bound}
The differentials of $\tau:B\to D$ and $\on{ht}:X\to\bR$ at a point $y$ satisfy $\norm{\on{d}_y\tau}\ll\on{ht}(y),\norm{\on{d}_y(\on{ht})}\ll \on{ht}(y)$.
\end{lemma}
\begin{proof}
We calculate for example $\norm{\on{d}_y\tau}$ for $y\in B^+$ (here $B^+=\pi(\cB^+)$).
Let $N,H,$ and $W$ denote the respective derivatives at time $t=0$ of the one parameter subgroups $n(s),a(t),$ and $k_\theta$ which appear in~\eqref{Iwasawa};
$$N=\pa{\begin{array}{ll}0&1\\ 0&0\end{array}},\;H=\pa{\begin{array}{ll}1&0\\ 0&-1\end{array}},\; W=\pa{\begin{array}{ll}0&1\\ -1&0\end{array}}.$$ 
Let $g\in \cB^+$ be such that $y=\pi(g)$ and write $g=\smallmat{a&b\\c&d}$ so that $\tau(y)=\pa{\frac{a}{c},cd}$ as given in~\eqref{tau}. The tangent space $T_y(X)$ is identified (as an inner product space) with $T_g(G)$ which is  in turn identified with
the Lie algebra $\gog=\gos\gol_2(\bR)$ via the map sending a matrix $V\in\gog$ to $gV$; here we make a choice of an inner product on $\gog$ which induces the left-invariant Riemannnian metric on $G$ and hence on the quotient $X$. Thus, we will obtain an upper bound for the norm of $\on{d}_y\tau$ if  we calculate an upper bound for the norms in $\bR^2$ of the vectors $\on{d}_y\tau(gV)$ for $V=N,H,W$ (where here we abuse notation and think of $\on{d}_y\tau$ as a map from $T_g(G)$ to $\bR^2$).
 
We may think of the above $2\times 2$ matrices as vectors in $\bR^4$ (where the first row corresponds to the first two coordinates) and then we get that $\on{d}_y\tau$ is given by the matrix
$$\on{d}_y\tau=\pa{\begin{array}{cccc}\frac{1}{c}&0&-\frac{a}{c^2}&0\\ 0&0&d&c\end{array}}.$$
A short calculation shows that 
$$\on{d}_y\tau(gN)=\pa{\begin{array}{c}0\\ c^2\end{array}},\; \on{d}_y\tau(gH)=\pa{\begin{array}{c}0\\ 0\end{array}},\;\on{d}_y\tau(gW)=\pa{\begin{array}{c}c^{-2}\\ c^2-d^2\end{array}}.$$
We conclude that $\norm{\on{d}_y\tau}\ll\max\set{c^2,c^{-2},d^2}$, where the implicit constant comes from the fact that we did not specify an inner product on $\gog$. Writing $g$ in its $(s,\theta,t)$-coordinates $g=a(s)k_\theta a(t)$ we calculate $c,d$ and conclude that as $\av{t}\le\eps_0$, $\norm{\on{d}_y\tau}\ll e^{\av{s}}$. Remark~\ref{r.8.1} now gives $\norm{\on{d}_y\tau}\ll\on{ht}(y)$ as desired.

We briefly describe the estimate for $\on{d}_y(\on{ht})$. Let $g\in\overline{\cF}$ be such that $y=\pi(g)$. Assume for a start that the Iwasawa decomposition of 
$g$ is given by $g=n(t)a(s)$. Then the derivative in the directions of $W$ and $N$ are trivial (because the actions from the right of the one parameter groups 
$k_\theta,u(t)$ do not change the height). The derivative in the direction of $H$ is $e^{s}$ which equals $\on{ht}(y)$. It follows that for such points 
$\norm{\on{d}_y(\on{ht})}\ll\on{ht}(y)$.
Now for the general case, let $g=n(t)a(s)k_\theta\in\overline{\cF}$ be the Iwasawa decomposition and consider the composition $G\to G\to\bR$ given by first acting on the 
right by $k_{-\theta}$ and then applying $\on{ht}$. As $\on{ht}$ is invariant under the action from the right by $k_{-\theta}$, this composition equals $\on{ht}$. Its 
differential at $y$ equals by the chain rule to the composition of the differential of right multiplication by $k_{-\theta}$ at the point $y$ and the differential of $\on{ht}$
at the point $y'=\pi(g')$, where $g'=n(t)a(s)$. As right multiplication by $k_{-\theta}$ is an isometry the first differential has norm 1 
(here we use the fact that the left invariant Riemannian metric we chose on $G$ is also right $\set{k_\theta}$-invariant). We evaluated the norm of the second 
differential before and we conclude that the composition satisfies the desired estimate. 
\end{proof}
\begin{remark}\label{r.7.1}
As the differential of $\on{ht}:X\to \bR$ is $\ll M$ on $K_M$. It follows that it is Lipschitz there with a Lipschitz constant
$\ll M$ (see Remark~\ref{r.9}). We conclude that there exists some absolute constant $\ell$ (which is the implicit constant in the estimate $\norm{\on{d}_y(\on{ht})}\ll \on{ht}(y)$), such that the following two statements hold
\begin{enumerate}
\item\label{r-1} For any $0<\eps<1$, $(H_M)_\eps\subset H_{\frac{M}{\ell}}$.  
\item\label{r-2} For any $0<\eps<1$, $(K_M)_\eps\subset K_{\ell M}.$
\end{enumerate}
To see~\eqref{r-1} for example, note that if this was false, then we could find $x\in K_{\frac{M}{\ell}}$ the distance of which from $H_M$ is $\le 1$. We conclude that 
there must be a point $x'$ such that $\on{ht}(x')=M$ and $\on{d}_X(x,x')\le 1$. This of course contradicts the fact that $\on{ht}$ is $M$-Lipschitz on $K_M$.
\end{remark}

\begin{remark}\label{r.8.1}
We wish to comment on the height of a point $y=\pi(g)\in B$, where $g\in\cB$ has coordinates $(s,\theta,t)$. By Lemma~\ref{lengths of paths}, if we let $g'\in\cC$ be the point with coordinates
$(s,\theta,0)$, then $\on{d}_G(g,g')\ll\eps_0$ (here we take $h(t)=a(t)$ to `cancel' the $t$-coordinate in at most $\eps_0$ time). The height of $g'$ is by definition $\on{ht}(g')=e^{\av{s}}$ (the reason for the absolute value is that $g'$ might be in the lower fundamental domain $k_{\frac{\pi}{2}}\cF$). We conclude from parts~\eqref{r-1},\eqref{r-2} of 
Remark~\ref{r.7.1} that $$\av{s}-\log\ell\le\log(\on{ht}(g))\le \av{s}+\log\ell.$$  
\end{remark}
\subsection{The argument}\label{t.a}  
\begin{proof}[Proof of Lemma~\ref{t.l}]
Fix $M>1$ and $0<\eps<1$ (below $\eps$ replaces the number $\rho$ in the statement of Lemma~\ref{t.l}). Let $F\subset X$ be defined by
\begin{equation}\label{F}
F=(\partial B)_\eps\cup H_M.
\end{equation}
Define $\vphi_{\eps,F}:X\to[0,1]$ as in Lemma~\ref{Lip construction}. To ease the notation
we simply denote it by $\vphi$ bearing in mind the dependencies on $\eps,M$. Lemma~\ref{Lip construction} implies the assertion in Lemma~\ref{t.l}\eqref{t.l.1}. 
Let $\psi=1-\vphi$. As $\vphi$ attains the value 1 on $X\smallsetminus (F)_\eps$ we have that $\psi\le\chi_{(F)_\eps}$. Furthermore, by Remark~\ref{r.7.1}\eqref{r-1} and  from the definitions we see that
 $$(F)_{\eps}\subset (\partial B)_{2\eps}\cup(H_{M})_\eps\subset  \pa{(\partial B)_{2\eps}\cap K_{M}}\cup H_{\frac{M}{\ell}}.$$
It follows that 
$$\int_X\psi dm_X\le m_X\pa{\pa{(\partial B)_{2\eps}\cap K_{M}}}+m_X(H_{\frac{M}{\ell}}).$$ 
Hence, by Remark~\ref{r.7}, Lemma~\ref{t.l}\eqref{t.l.2} will follow once we show that the following estimate holds for all $M>1$
\begin{align}
\label{v.e.2} m_X\pa{\pa{(\partial B)_{2\eps}\cap K_M}}\ll \eps\log M.
\end{align}
In order to establish~\eqref{v.e.2} we argue as follows: We first want to pull the calculation to $G$ and then to $\bR^3$. It is clear that $\pi(\partial\cB\cap\cK_M)=\partial B\cap K_M$ and as $\pi$ can only decrease 
distances (that is $\pi$ is $1$-Lipschitz), we must have $\pi( (\partial\cB)_{2\eps}\cap\cK_M)\supset (\partial B)_{2\eps}\cap K_M.$ By the definition of the measure $m_X$
it follows that 
\begin{equation}\label{shabat 2}
m_X((\partial B)_{2\eps}\cap K_M)\le m_G((\partial\cB)_{2\eps}\cap\cK_M).
\end{equation}
Hence, we are reduced to estimate $m_G((\partial\cB)_{2\eps}\cap \cK_M)$. We will workout below the estimation for $m_G\pa{(\partial\cB^+)_{2\eps}\cap\cK_M}$ only. 
Let $N_\eps(L)$  denote the number of $\eps$-balls needed to cover a set $L$. Clearly,
$$N_{3\eps}((\partial\cB^+)_{2\eps}\cap\cK_M)\le N_\eps(\partial \cB\cap \cK_M).$$
We know that a ball of radius $\eps$ in $G$ has volume $\ll \eps^3$ and so we deduce that 
\begin{equation}\label{shabat 1}
m_G((\partial\cB^+)_{2\eps}\cap\cK_M)\ll \eps^3 N_\eps(\partial \cB^+\cap \cK_M).
\end{equation}
Consider the following four subsets of $\overline{\cE}^+\subset \bR^3$ which are mapped by $\xi$ onto the boundary $\partial\cB$
\begin{align*}
 \cQ_1&=\set{(s,\theta,t):s\in\bR, t\in(0,\eps_0),\theta=\theta_{\on{min}}(s)};\\
 \cQ_2&=\set{(s,\theta,t):s\in\bR, t\in(0,\eps_0),\theta=\frac{\pi}{2}};\\
 \cQ_3&=\set{(s,\theta,t):s\in\bR, \theta\in I^+_s, t=0};\\
 \cQ_4&=\set{(s,\theta,t):s\in\bR, \theta\in I^+_s, t=\eps_0}.
\end{align*}
Let $\cQ=\cup_{i=1}^4\cQ_i$.
A point in $\cB\cap \cK_M$ with coordinates $(s,\theta,t)$ must satisfy $\av{s}\le\log M+\log\ell$ as explained in Remark~\ref{r.8.1}. 
Hence, we conclude by Lemma~\ref{cover link} that
\begin{equation}\label{shabat 4}
N_\eps(\partial\cB^+\cap\cK_M)\ll N_{c^{-1}\eps}(Q\cap\set{(s,\theta,t):\av{s}\le\log M+\log \ell}).
\end{equation}
This reduces the problem to a Euclidean one: For each $1\le i\le 4$ the surface $$\cQ_i\cap\set{(s,\theta,t):\av{s}\le\log M+\log 2}$$ is a graph of a function from a domain in $\bR^2$ to $\bR$. The variables vary in 
a range that is of bounded length in one direction and of length $2(\log M+\log\ell)$ in the other. As all these functions have derivatives which are uniformly bounded
(in fact, all of them are constant apart from the function $(s,t)\mapsto \theta_{\on{min}}(s)$ corresponding to $\cQ_1$, see~\eqref{theta interval}), we deduce that
\begin{equation}\label{shabat 3}
N_{c^{-1}\eps}(Q\cap\set{(s,\theta,t):\av{s}\le\log M+\log \ell})\ll \frac{\log M}{\eps^2}.
\end{equation}
Combining~\eqref{shabat 3},\eqref{shabat 4},\eqref{shabat 1}, and~\eqref{shabat 2} gives~\eqref{v.e.2}, which as explained above concludes the proof of
Lemma~\ref{t.l}\eqref{t.l.2}. We turn now to the proof of Lemma~\ref{t.l}\eqref{t.l.3}.

Let $f:D\to\bC$ be $\ka$-Lipschitz and denote $\tilde{f}=\widehat{f\circ\tau}$. The support of the product $\tilde{f}\cdot\vphi$ is contained in the intersection of 
the supports of $\tilde{f}$ and $\vphi$. By definition of the $\widehat{\;}$ operator, the support of $\tilde{f}$ is contained in $B$. By definition of $\vphi$ its support
is contained in the intersection $\set{x\in X:\on{d}_X(x,\partial B)\ge \eps}\cap \overline{K}_M$. It follows that 
\begin{equation}\label{support contained}
\on{supp}(\tilde{f}\cdot\vphi)\subset \set{x\in B : \on{d}_X(x,\partial B)\ge \eps}\cap \overline{K}_M.
\end{equation}
As the points of discontinuity of $\tilde{f}$ are contained in $\partial B$ we conclude that $\tilde{f}\cdot\vphi:X\to\bC$ is continuous. In order to estimate its Lipschitz constant
we wish to appeal to Remark~\ref{shabat r.2}.
Cover the open set $\cB\cap\cK_{2M}$ by open balls $\cU_i\subset \cB\cap\cK_{2M}$. Note 
that each $\cU_i$ is contained in either $\cB^+$ or $\cB^-$.
Consider the open cover $\set{U_i}$ of $\on{supp}(\tilde{f}\cdot\vphi)$, where $U_i=\pi(\cU_i)$. By Remark~\ref{shabat r.2},
Lemma~\ref{t.l}\eqref{t.l.3} will follow once we prove that $\tilde{f}\cdot\vphi:U_i\to\bC$ is $\max\set{\ka,\norm{f}_\infty}\eps^{-1}M$-Lipschitz. As $\vphi$ is $\eps^{-1}$-Lipschitz
 we see that by Remark~\ref{shabat r.1} it is enough to argue that for each $i$, $\tilde{f}:U_i\to\bC$ is Lipschitz with Lipschitz constant $\ll \ka M$. As $U_i\subset K_{2M}$ we know by Lemma~\ref{tau bound} that the norm of the differential of $\tau$ is $\ll M$ on $U_i$. It follows that the Lipschitz constant of the composition $\tilde{f}=f\circ\tau$ is  $\ll \ka M$ as desired.
 \end{proof}
 \begin{remark}\label{r.9}
We remark here about a slight inaccuracy in the arguments presented above and how to remedy it: Let $M,N$ be two Riemannian manifolds  and  $f:U\to N$ a smooth map
from an open set $U\subset M$. Assume the differential of $f$ has norm bounded by some constant $\ka$ on $U$. We used above (in two places) the conclusion that $f$ must be $\ka$-Lipschitz. Strictly speaking, this shows indeed that $f$ is $\ka$-Lipschitz, but with respect to the metric induced from the restriction of the Riemannian metric from $M$ to $U$. This need not be the restricted metric on $U$ in which we are interested. In order to remedy this, one needs to prove that the following property holds: 
\textit{There exists some absolute constant $c$ such that given any two points in $x,y\in U$ one is able to find a path connecting them inside $U$ of length $\le c\on{d}(x,y)$ (here $\on{d}$ is the metric of the ambient space containing $U$).} 

Once this property is established, the conclusion is that $f$ has Lipschitz constant $\ll \ka$. The above property clearly holds in any Euclidean ball. Using the fact that the exponential map from the Lie algebra to $G$ is bi-Lipschitz when restricted to a small enough neighborhood of zero, we see that any image of a small enough  Euclidean ball around zero is an open neighborhood of the identity in $G$ which satisfies the desired property. Using left translations (which are isometries of $G$) we see that each point of $G$ has a basis of neighborhoods satisfying the above properties. Regarding the argument in the very end of the proof of Lemma~\ref{t.l}, we should simply define the sets $\cU_i$ to be such neighborhoods instead open balls. Regarding the use of this in Remark~\ref{r.7.1}, we leave the details to the reader.
 \end{remark}


\def\cprime{$'$} \def\cprime{$'$} \def\cprime{$'$}

\end{document}